\documentclass[reqno,10pt]{amsart}

\usepackage{amsmath, amsfonts, amsthm, amssymb,mathrsfs, stmaryrd, color, cite}

\usepackage{hyperref}
\usepackage{bm}
\usepackage{graphicx}
\usepackage{epstopdf}
\usepackage{subfigure}
\usepackage{indentfirst}
\usepackage{verbatim}
\usepackage{slashbox}
\usepackage{tikz}
\usepackage{cases}
\newtheorem{theorem}{Theorem}[section]
\newtheorem{lemma}{Lemma}[section]
\newtheorem{proposition}{Proposition}[section]
\newtheorem{assumption}{Assumption}[section]

\theoremstyle{definition}
\newtheorem{definition}{Definition}[section]

\theoremstyle{remark}

\numberwithin{equation}{section}
\allowdisplaybreaks

\title[Martingale solutions to the SPLLG equations]
{Weak Martingale Solutions of the Stochastic Schr\"{o}dinger-Poisson-Landau-Lifshitz-Gilbert System}

\author[Y. Wei]{Yurong Wei}
\address{College of Mathematics and Statistics, Chongqing University, Chongqing, 401331, China.}
\email{yurongw@163.com}

\author[H. Wang]{Huaqiao Wang}
\address{College of Mathematics and Statistics, Chongqing University, Chongqing, 401331, China.}
\email{wanghuaqiao@cqu.edu.cn}


\date{\today}

\begin{document}
\begin{abstract}
The Schr\"{o}dinger-Poisson-Landau-Lifshitz-Gilbert (SPLLG) system can characterize the spin transfer torque mechanism transferring the spin angular momentum to the magnetization dynamics through spin-magnetization coupling. We study the three-dimensional stochastic SPLLG system driven by a multiplicative stochastic force containing a continuous noise and a small jump noise. We establish the existence of weak martingale solutions based on the penalized functional technique, the Faedo-Galerkin approximation, stochastic compactness method, and a careful identification of the limit. Due to the strong coupling and strong nonlinearity caused by the SPLLG system and stochastic effects, some crucial difficulties have been encountered in obtaining energy estimates and avoiding non-negativity of the test function.  We mainly utilize the structure of equations and the property of martingales developing the new energy estimates, and apply the three-layer approximation to overcome these difficulties. In particular, we extend the results by Z. Brze\'{z}niak and U. Manna (Comm. Math. Phys., 2019) and by  L.H. Chai, C.J. Garc\'{\i}a-Cervera and X. Yang (Arch. Ration. Mech. Anal., 2018)  to both the stochastic case and the coupling case.
\end{abstract}

\maketitle

\section{Introduction}
The ferromagnetic chain equation is an important class of magnetization dynamics equations that frequently appear in research in condensed matter physics. As theoretical studies delve deeper, physicists have proposed more refined models, such as equations describing the magnetization dynamics of ferromagnetic materials. In magnetization dynamics, spin and magnetization are inherently linked, hence it is necessary to consider spin-polarized transport effects in the study of ferromagnetic materials. The Schr\"{o}dinger-Poisson-Landau-Lifshitz-Gilbert (SPLLG) system \cite{CGY1} can be used to describe the spin transfer torque mechanism, transferring spin angular momentum to magnetization dynamics through spin-magnetization coupling (see \cite{Ber1,Ber2,Slon}). The SPLLG system is an effective microscopic model that characterizes the coupling between the spin of conduction electrons and magnetization in ferromagnetic materials. The system combines two different models: one describes the spin of conduction electrons using the Schr\"{o}dinger-Poisson (SP) equation, and the other describes the magnetization dynamics behavior using the Landau-Lifschitz-Gilbert (LLG) equation.

The Schr\"{o}dinger-Poisson equation (see \cite{CGY, CLZ, PT}) can be written as
\begin{align}\label{S0}
i\varepsilon\partial_t\bm{\psi}_j^{\varepsilon}(\bm{x},t)
=-\frac{{\varepsilon}^2}{2}\Delta\bm{\psi}_j^{\varepsilon}(\bm{x},t)
+V^{\varepsilon}\bm{\psi}_j^{\varepsilon}(\bm{x},t)
-\frac{\varepsilon}{2}\bm{m}^{\varepsilon}\cdot\hat{\bm{\sigma}}\bm{\psi}_j^{\varepsilon}(\bm{x},t),
\end{align}
where $\{\bm{\psi}_{j}^{\varepsilon}\}_{j=1}^{\infty}$ are the wave functions, $0<\varepsilon\ll1$ is the renormalized Planck constant in a semiclassical system, and $\bm{\psi}_{j}^{\varepsilon}=(\psi_{j,+}^{\varepsilon}$, $\psi_{j,-}^{\varepsilon})^{\rm{T}}$  represents the jth spinor, here $\pm$ indicate up and down spins respectively. The Pauli matrix $\hat{\bm{\sigma}}=(\sigma_1,\sigma_2,\sigma_3)^{\rm{T}}$ is defined by
\begin{align*}
\sigma_1=\begin{pmatrix}0 & 1 \\ 1 & 0\end{pmatrix},
\sigma_2=\begin{pmatrix}0 & -i \\ i & 0\end{pmatrix},
\sigma_3=\begin{pmatrix}1 & 0 \\ 0 & -1\end{pmatrix}.
\end{align*}
The position density $\rho^{\varepsilon}$, the current density $\bm{j}^{\varepsilon}$, the spin density $\bm{s}^{\varepsilon}$, and the spin current $\bm{J}_{s}^{\varepsilon}$ are given by:
\begin{align*}
&\rho^{\varepsilon}(\bm{x},t)
=\sum_{j=1}^{\infty}\lambda_{j}^{\varepsilon}|\bm{\psi}_{j}^{\varepsilon}(\bm{x},t)|^2,\quad
\bm{j}^{\varepsilon}(\bm{x},t)
=\varepsilon\sum_{j=1}^{\infty}\lambda_{j}^{\varepsilon}
{\rm{Im}}\big({\bm{\psi}_{j}^{\varepsilon}}^{\dag}(\bm{x},t)
\nabla_{\bm{x}}\bm{\psi}_{j}^{\varepsilon}(\bm{x},t)\big),\\
&\bm{s}^{\varepsilon}(\bm{x},t)
=\sum_{j=1}^{\infty}\lambda_{j}^{\varepsilon}
{\rm{Tr}}_{{\mathbb{C}}^2}\left(\hat{\bm{\sigma}}\big({\bm{\psi}_{j}^{\varepsilon}}(\bm{x},t)
{\bm{\psi}_{j}^{\varepsilon}}^{\dag}(\bm{x},t)\big)\right),\\
&\bm{J}_{s}^{\varepsilon}(\bm{x},t)
=\varepsilon\sum_{j=1}^{\infty}\lambda_{j}^{\varepsilon}
{\rm{Im}}\left({\rm{Tr}}_{{\mathbb{C}}^2}\left(\hat{\bm{\sigma}}
\otimes\nabla_{\bm{x}}\bm{\psi}_{j}^{\varepsilon}(\bm{x},t)
{\bm{\psi}_{j}^{\varepsilon}}^{\dag}(\bm{x},t)\right)\right).
\end{align*}
Here the coefficient $\lambda_{j}^{\varepsilon}\geq 0$ represents the occupation probability of the  $\mathbb{L}^2(\mathbb{R}^3)$-orthonormal initial states $\{\bm{\varphi}_{j}^{\varepsilon}\}_{j=1}^{\infty}$, Im denotes the imaginary part of a complex-valued function, ${\bm{\psi}_{j}^{\varepsilon}}^{\dag}$ is the complex conjugate transpose of $\bm{\psi}_{j}^{\varepsilon}$, ${\rm{Tr}}_{\mathbb{C}^2}$ is the trace operator of a $2\times 2$ complex matrix, and $\otimes$ represents a tensor product of two three-dimensional vectors. Thus $\bm{s}^{\varepsilon}$ is the three-dimensional vector, and $\bm{J}_{s}^{\varepsilon}$ is a $3\times 3$ matrix. The self-consistent Coulomb potential $V^{\varepsilon}$ in \eqref{S0} is defined as $V^{\varepsilon}=-N\ast \rho^{\varepsilon}$, with the kernel function $N(\bm{x})=-\frac{1}{4\pi|\bm{x}|}$, and $\ast$ denotes the convolution operator. The last term in \eqref{S0} describes the ``s-d'' Hamiltonian term for spin-magnetic interactions.

The Landau-Lifschitz-Gilbert equation \cite{CGY, G, LL} is expressed in the following form:
\begin{eqnarray}\label{LLG}
\left\{\begin{array}{ll}
\!\!\!\partial_t\bm{m}^{\varepsilon}=-\bm{m}^{\varepsilon}\times\bm{H}_{\rm{eff}}^{\varepsilon}
+\alpha \bm{m}^{\varepsilon}\times\partial_t\bm{m}^{\varepsilon},\, |\bm{m}^{\varepsilon}(\bm{x},t)|=1, \,\bm{x}\in D,\\
\!\!\!\partial_\nu \bm{m}^{\varepsilon}=0,\,\bm{x}\in \partial D,
\end{array}\right.
\end{eqnarray}
where $\bm{m}^{\varepsilon}$ represents the magnetization, $\alpha$ is the dimensionless damping constant, $D$ is a bounded region with smooth boundary, and $\nu$ is the unit outward normal vector on $\partial D$.
The effective field $\bm{H}_{\rm{eff}}^{\varepsilon}$ is the variational derivative with respect to $\bm{m}^{\varepsilon}$ of the Landau-Lifschitz energy $F_{\rm{LL}}=\int_{D}\left(\frac{1}{2}|\nabla\bm{m}^{\varepsilon}|^2+w(\bm{m}^{\varepsilon})
-\frac{1}{2}\bm{H}_{s}^{\varepsilon}\cdot\bm{m}^{\varepsilon}
-\frac{\varepsilon}{2}\bm{s}^{\varepsilon}\cdot\bm{m}^{\varepsilon}\right)d\bm{x}$, that is,
\begin{align*}
\bm{H}_{\rm{eff}}^{\varepsilon}
=-\frac{\delta F_{\rm{LL}}}{\delta\bm{m}^{\varepsilon}}
=\Delta\bm{m}^{\varepsilon}-w'(\bm{m}^{\varepsilon})+\bm{H}_{s}^{\varepsilon}
+\frac{\varepsilon}{2}\bm{s}^{\varepsilon}.
\end{align*}
The terms on the right-hand side of the above equation represent the exchange field, the anisotropy field, the stray field, and additional torque, respectively. In the Landau-Lifschitz energy,
$w(\bm{m}^{\varepsilon})$ represents the anisotropy energy, usually $w\geq 0$ is a polynomial up to degree 4.
Specifically, for a given uniaxial anisotropy, $w(\bm{m}^{\varepsilon})=({m_{2}^{\varepsilon}})^2+({m_{3}^{\varepsilon}})^2$
satisfies this assumption. Let $w'(\bm{m}^{\varepsilon}):=\nabla_{\bm {m}}w(\bm{m}^{\varepsilon})$ denote the variational derivative.
The coupling term $\bm{s}^{\varepsilon}\cdot\bm{m}^{\varepsilon}$ generates the spin-transfer torque, converting spin angular momentum
to magnetization dynamics, with $\frac{\varepsilon}{2}$ representing the corresponding coupling strength.
Let $\bm{H}_{s}^{\varepsilon}=-\nabla u$, where the magnetostatic potential is given
by $u=\nabla N\ast\cdot\bm{m}^{\varepsilon}$. Therefore, $\bm{H}_{s}^{\varepsilon}(\bm{x})=-\nabla(\nabla N\ast\cdot\bm{m}^{\varepsilon})$.

The LLG equation and the SPLLG equation are widely used in the field of materials science, magnetism,
and electromagnetism. For example, the LLG equation has been applied to the study of energy efficiency, stability,
and storage density of magnetic storage devices. The SPLLG system has been applied to the study of spin transport and
magnetization reversal in ferromagnetic materials, playing important roles in domain wall motion control and magnetization reversal
in magnetic multilayers. In recent years, numerical simulations of the SPLLG system have been conducted
by Chen \textit{et al.} \cite{CGY1,CGY2,CLZ}. Chai, Garc\'{\i}a-Cervera and Yang \cite{CGY} proved the existence of weak solutions for the SPLLG system
and obtained the semiclassical limit equation when $\varepsilon$ tends toward 0.
In the SPLLG system, if only electron spins are considered, the SP system can be used to describe the motion of conducting electrons. The existence and uniqueness of global classical solutions and $L^2$ solutions for the three-dimensional SP system can be found in \cite{BM, C}. The existence of radially symmetric solutions and ground state solutions for the SP system is presented in \cite{Am, CV, LG, SM, ZZ}.

In the SPLLG system, if only the magnetization of ferromagnetic materials is considered, the LLG equation can be used to describe the magnetization dynamics. The research on ferromagnetic theory can be found in \cite{B,W} and the references therein. Landau and Lifshitz \cite{LL} and Gilbert \cite{G0} further developed the theory. Additionally, many scholars discussed the mathematical theory of the LLG equation. Alouges and Soyuer \cite{AS} proved the existence and non-uniqueness of global weak solutions of the Cauchy problem of the three-dimensional LLG equation. Carbou and Fabrie \cite{CF} showed the local existence and uniqueness of regular solutions to the initial-boundary value problem for the three-dimensional LLG equation; for the two-dimensional case, they established the global existence of regular solutions with sufficiently small initial data. Garc\'{\i}a-Cervera and Wang \cite{GW} coupled the LLG system with the spin transport equation for studying the spin polarization dynamics and demonstrated the existence and non-uniqueness of weak solutions in three-dimensional space. Later, Guo and Pu \cite{GP, PG} obtained the existence of global smooth solutions for one and two-dimensional spin-polarized transport system. When the dimension $n\geq3$ and the norm of the initial gradient i.e. $|\nabla\bm{m}_0|_{L^n}$ is sufficiently small, Melcher \cite{Me} obtained the  existence, uniqueness and asymptotic behavior of global smooth solutions to the Cauchy problem of the LLG equation. The existence of any regular solution to the three-dimensional LLG equation with homogeneous Neumann boundary conditions was considered by Feischl and  Tran \cite{FT} when the initial data is a small perturbation near constant states. Recently, in the case of dimension $n\geq2$, Wang and Guo \cite{WG} studied the regularity and established blow-up criteria of strong solutions to the LLG equation with Neumann boundary conditions.

Considering random disturbances caused by phonons, nuclear spins, and conduction electrons, an important issue in ferromagnetic theory is the description of phase transitions between different equilibrium states caused by thermal fluctuations of the effective field. To incorporate random fluctuations of the effective field into magnetization dynamics and describe noise-induced phase transitions between equilibrium states of ferromagnetic materials, the LLG equation needs to be modified. N\'{e}el\cite{N}  first proposed the analysis of noise-induced phase transitions, and Brown \cite{B0}, Kamppeter \textit{et al.} \cite{KMMSB} further extended this work. Recently, Brze\'{z}niak, Goldys and Jegaraj \cite{BGJ, BGJ1} considered the well-posedness of solutions to the stochastic LLG equation. A simple approach for introducing noise into the LLG equation is to use Gaussian noise to perturb the effective field, for example, the simplified stochastic LLG equation is given by:
\begin{align*}
\frac{\partial\bm{m}}{\partial t}=\lambda_1\bm{m}\times(\Delta\bm{m}+dW)-\lambda_2\bm{m}\times(\bm{m}\times(\Delta\bm{m}+dW)),
\end{align*}
where $\lambda_1\in\mathbb{R}$, $\lambda_2>0$. The effective field is given by $\bm{H}_{\rm{eff}}=\Delta\bm{m}+dW$.

The references \cite{B0, GL, KMMSB} indicate that in the case of coordinate transformation, the addition of noise to the LLG equation should remain invariant. In physics \cite{GL, KRV}, if $\lambda_2$ is small, the noise term on the right-hand side of the above equation can be neglected. Based on this perspective, Brze\'{z}niak, Goldys and Jegaraj \cite{BGJ} considered the following stochastic LLG equation:
\begin{align*}
d\bm{m}=(\lambda_1\bm{m}\times\Delta\bm{m}-\lambda_2\bm{m}\times(\bm{m}\times\Delta\bm{m}))dt+\bm{m}\circ dW,
\end{align*}
where $\circ dW$ represents the Stratonovitch differential (for more details, see \cite{KPP}). Alouges, Bouard and Hocquet \cite{ABH} conducted numerical simulations of the stochastic LLG equation with continuous noise under the Stratonovitch integration. More numerical results can be found in \cite{BBP, BBNP, BBNP1}. Goldys, Le and Tran \cite{GLT} proved the existence of weak martingale solutions for the three-dimensional stochastic LLG equation. Guo and Pu \cite{GP1} established the existence of global weak solutions of the Cauchy problem to the $d$ ($d\geq 2$)-dimensional stochastic LL equations. Recently, Brze\'{z}niak \textit{et al.} \cite{BGJ, BL} proved the existence of weak martingale solutions for the three-dimensional stochastic LLG equation; They \cite{BGJ1} demonstrated the existence and uniqueness of pathwise solutions for the stochastic LLG equation on a bounded interval and also established the large deviations of the solutions. In the case of the stochastic LLG equation driven by  L\'{e}vy noise, in order to maintain the constraint condition, it is necessary for the noise to remain invariant under coordinate transformation. Applebaum \textit{et al.} \cite{A,CP,K,Mar} provided a framework for addressing this technical issue. Very recently, Brze\'{z}niak and Manna \cite{BM1} considered the following stochastic LLG equation:
\begin{align*}
d\bm{m}=(\lambda_1\bm{m}\times\Delta\bm{m}-\lambda_2\bm{m}\times(\bm{m}\times\Delta\bm{m}))dt
+\bm{m}\times\bigg(\sum_{i=1}^{N}e_i\diamond dL_i\bigg),
\end{align*}
where $L:=(L_1, L_2, \cdots, L_N)\in\mathbb{R}^N $ represents a pure L\'{e}vy process, and $\diamond dL_i$ denotes the Marcus differential (see \cite{KPP}), $e_i$ is a given bounded function. Brze\'{z}niak, Manna and Mukherjee \cite{BMM} obtained the existence and regularity of pathwise solutions to the one-dimensional stochastic LLG equation containing only non-zero exchange energy. The authors also demonstrated the existence of weak martingale solutions for the initial-boundary value problem of the three-dimensional stochastic LLG equation with pure jump noise in \cite{BM1, BM22}. Specifically, when verifying the constraint condition $|\bm{m}(t,\bm{x})|_{\mathbb{R}^3}=1$, a non-negative test function is required.

Due to the complexity of the real world, both continuous and jump models exist, and both large and small jumps are permissible. For example, random effects caused by phonons, charge transport in electrons, and nuclear spins are continuous under favorable conditions; But when they are suddenly disturbed by other media, adsorbed escape, vibration or other extreme activities, the random effects may be jumps. Considering the combined impact of continuous  and discontinuous random effects to explain the thermal fluctuations of the magnetic moment direction, it is reasonable to study the stochastic LLG equation driven by the multiplicative noise. Therefore, based on the SPLLG system \cite{CGY, CLZ}, we investigate the following stochastic SPLLG system:
\begin{eqnarray}\label{S1}
\left\{\begin{array}{ll}
\!\!\!i\partial_t\bm{\psi}_j(\bm{x},t)
=-\frac{1}{2}\Delta\bm{\psi}_j(\bm{x},t)+V\bm{\psi}_j(\bm{x},t)-\frac{1}{2}\bm{m}\cdot\hat{\bm{\sigma}}\bm{\psi}_j(\bm{x},t),
~ j\in \mathbb{N},\\
\!\!\!\partial_t\bm{m}=-\bm{m}\times\bm{H}+\alpha\bm{m}\times\partial_t\bm{m}+(\bm{m}\times G(\bm{m}))\circ\frac{dW}{dt}
+(\bm{m}\times F(\bm{m},l))\frac{dL(t)}{dt},\\
\!\!\!V=-N\ast\rho[\bm{\Psi}]\\
\end{array}\right.
\end{eqnarray}
with $\bm{m}(\bm{x},0)=\bm{m}_{0}(\bm{x})$, $\bm{x}\in D$; $\bm{\psi}_j(\bm{x},0)=\bm{\varphi}_{j}(\bm{x})$, $\bm{x}\in\mathbb{R}^3$; $\partial_\nu \bm{m}(\bm{x},t)=0$, $(\bm{x},t)\in \partial D\times\mathbb{R}^{+}$. where $D$ is an open bounded domain.
The effective field $\bm{H}$ is given by:
\begin{align*}
\bm{H}=\Delta\bm{m}-w'(\bm{m})+\bm{H}_{s}+\frac{1}{2}\bm{s}[\bm{\Psi}].
\end{align*}
$\rho[\bm{\Psi}]$, $\bm{s}[\bm{\Psi}]$ and $\bm{H}_s$ are respectively defined as:
\begin{align}\label{ss}
\rho[\bm{\Psi}]=\sum_{j=1}^{\infty}\lambda_{j}|\bm{\psi}_{j}|^2,~
\bm{s}[\bm{\Psi}]=\sum_{j=1}^{\infty}\lambda_{j}{\rm{Tr}}_{{\mathbb{C}}^2}\big(\hat{\bm{\sigma}}\big({\bm{\psi}_{j}}
{\bm{\psi}^{\dag}_{j}}\big)\big),~
\bm{H}_{s}=-\nabla(\nabla N\ast\cdot\bm{m}).
\end{align}
Moreover, denote $\bm{\Psi}:=\{\bm{\psi}_j\}_{j\in \mathbb{N}}$, $\bm{\Psi}_0:=\{\bm{\varphi}_j\}_{j\in \mathbb{N}}$.
$W(t)=(W_1(t), W_2(t), \cdots, W_N(t))$ is an $\mathbb{R}^{N}$-valued Wiener process, $\frac{dW}{dt}$ formally denotes ``white noise'' in physics, $\circ$ stands for the Stratonovich differential. 
$L(t)=(L_1(t), L_2(t), \cdots, L_N(t))$ is an $\mathbb{R}^{N}$-valued L\'{e}vy process with pure jump, i.e.
\begin{align*}
L(t)=\int_{0}^{t}\int_{B}l\tilde{\eta}(ds,dl)+\int_{0}^{t}\int_{B^c}l\eta(ds,dl), \quad t\geq0,
\end{align*}
where $B:=\mathbb{B}(0,1)\subset\mathbb{R}^{N}$, $\eta$ and $\tilde{\eta}$  represent a time homogeneous Poisson random measure and a time homogeneous compensated Poisson random measure with a compensation ${\rm{Leb}}\otimes\mu$, i.e. $\tilde{\eta}:=\eta-{\rm{Leb}}\otimes\mu$, respectively. $\frac{dL(t)}{dt}$ is a formal notation. 
It is worth noting that in order to close the energy estimates, we need to express the stochastic SPLLG system  in the form \eqref{S1}.

In this paper, we consider the existence of weak martingale solutions to the three-dimensional stochastic SPLLG system with the continuous noise and small jump noise. Compared with the deterministic SPLLG system \cite{CGY} studied by Chai, Garc\'{i}a-Cervera and Yang, the sample space $\Omega$ in the stochastic system \eqref{S1} does not have a topological structure, and no compactness of the sample points $\omega$ can be obtained. To overcome the obstacle, we employ the energy method and the stochastic compactness method. For the stochastic LLG equation, due to the magnetization $\bm{m}_n$ satisfies Neumann boundary conditions, the estimate $|\bm{m}_n|_{\mathbb{L}^2}\leq|\nabla\bm{m}_n|_{\mathbb{L}^2}$ cannot be utilized, leading to the inability to close the energy estimate. To this end, we establish the $\mathbb{H}^1$-estimate of $\bm{m}_n$ and the $\mathbb{L}^2$-estimate of $G_n(\bm{m}_n)$ generated by the stochastic terms using the H\"{o}lder inequality, the structure of equations, and the properties of martingales. Then, we can get the uniform energy estimate of each term in the stochastic SPLLG equation in the sense of expectation. Compared with the stochastic LLG equation \cite{BGJ,BM1} studied by Brze\'{z}niak \textit{et al.}, the system we investigate exhibits strong coupling and strong nonlinearity, which renders the stochastic term meaningless in the sense of Marcus, and the closure of the energy estimate is more complex than that of the individual stochastic LLG equation and the SPLLG equation. Firstly, we look for suitable function spaces to ensure that the stochastic perturbation is well-defined. Notice that in utilizing the Faedo-Galerkin approximation and the stochastic compactness method to establish the existence of solutions, Brze\'{z}niak \textit{et al.} \cite{BGJ,BM1} needed the non-negativity of the test function when verifying the saturation constraint conditions for $\bm{m}$. To avoid this, we introduce the penalized function with a penalization coefficient $k$ to transform the stochastic LLG equation into a penalized problem, and then use the three-layer approximation and the stochastic compactness method to prove the global existence of weak martingale solutions. Specifically, we first consider the case of a bounded domain $K=\{\bm{x}\in\mathbb{R}^3, |\bm{x}|<R\}$ for the SP system. In the first-layer approximation ($n\rightarrow\infty$), by applying the Faedo-Galerkin approximation method, we establish uniform estimates of $\bm{m}_n$ and $\bm{\psi}_{jn}$ with respect to $n$, and obtain the existence of weak martingale solutions to the penalized problem through the stochastic compactness method. In the second-layer approximation ($k\rightarrow\infty$), we use the established uniform energy estimates, the stochastic compactness method and a careful identification of the limit procedure to obtain the existence of weak martingale solutions for the original problem in a bounded domain. Noted that proving the tightness of the approximation sequence, it is difficult to directly prove that $\Delta\bm{m}^k$ and $\Delta\bm{\psi}_j^k$ satisfy the Aldous condition. To overcome this difficulty, we mainly utilize the definition of the fractional Laplacian operator and the extension theorem. In particular, we can prove that $\bm{m}$ satisfies the saturation constraint condition as $k\rightarrow\infty$, and we do not need the non-negativity of the test function. In the third-layer approximation ($R\rightarrow\infty$), we use the domain extension method, uniform energy estimates, and the stochastic compactness method to obtain the existence of weak martingale solutions to the SP system in the whole space $\mathbb{R}^3$. Therefore, we establish the global existence of the weak martingale solution of the system \eqref{S1}.

The structure of this paper is arranged as follows. In Section 2, we introduce some notations and assumptions, the definition of weak martingale solutions to the stochastic SPLLG system \eqref{S1}, and then state the main results. In Section 3, we establish energy estimates for the approximation sequences corresponding to the wave function and magnetization intensity. In Sections 4 and 5, we demonstrate the tightness of the approximation sequence $(\bm{m}_n,\bm{\psi}_{jn})$, and prove the existence of weak martingale solutions to the stochastic system \eqref{SLLG3}. In Sections 6 and 7, we establish the tightness of the approximation sequence $(\bm{m}^k,\bm{\psi}_j^k)$, and prove the existence of weak martingale solutions to the stochastic SPLLG system \eqref{S2} and \eqref{S1}. For convenience, some important lemmas are listed in Section 8.
\section{Preliminaries and main results}
\subsection{Notations and assumptions}
The function space $\mathbb{H}^1(D,\mathbb{R}^3)=:\mathbb{H}^1$ is defined as
\begin{align*}
\mathbb{H}^1(D,\mathbb{R}^3)=\big\{\bm{m}\in\mathbb{L}^2(D,\mathbb{R}^3):
\frac{\partial \bm{m}}{\partial x_i}\in\mathbb{L}^2(D,\mathbb{R}^3),~ i=1,2,3\big\}.
\end{align*}
We denote $\mathbb{L}^p(D,\mathbb{R}^3)=:\mathbb{L}^p~(p>0)$, $\mathbb{W}^{k,p}(D,\mathbb{R}^3)=:\mathbb{W}^{k,p}~(p>0)$ and we use $V$ to represent the Hilbert space $\mathbb{H}^1$. The dual space of $V$ is denoted by $V^{\ast}$. We use the notation $\langle\cdot,\cdot\rangle$ and $|\cdot|$ to represent the inner product and the norm of the space $\mathbb{L}^2$, $_{V}\langle\cdot,\cdot\rangle_{V^{\ast}}$ represents the duality product of the space $V$ and its dual space $V^{\ast}$. The $-\Delta$ operator A with Neumann boundary conditions is defined as follows:
\begin{eqnarray*}
\left\{\begin{array}{ll}
\!\!\!D(A):=\{\bm{m}\in\mathbb{H}^2: \frac{\partial\bm{m}}{\partial \nu}=0~on~\partial D\},\\
\!\!\!A\bm{m}:=-\Delta\bm{m},\quad\bm{m}\in D(A),
\end{array}\right.
\end{eqnarray*}
where $\nu=(\nu_1,\nu_2,\nu_3)$ is the unit outward normal vector on $\partial D$.
\begin{definition}\label{def}
For any non-negative real number $\beta$, the Hilbert space $X^{\beta}:=D(A_{1}^{\beta})$ is defined by the domain of the fractional power operator $A_{1}^{\beta}$, and $|\cdot|_{X^{\beta}}:=|A_{1}^{\beta}\cdot|_{\mathbb{L}^2}$. For positive real number $\beta$, $X^{-\beta}$ denotes the dual of $X^{\beta}$, and when $x\in\mathbb{L}^2$, we have $|x|_{X^{-\beta}}:=|A_{1}^{-\beta}x|_{\mathbb{L}^2}$, $X^0=\mathbb{L}^2$ is consistent with its dual.
\end{definition}
\begin{proposition}[\!\!\cite{BM1,T}]
\begin{eqnarray*}
X^{\beta}=D(A_1^{\beta})\!=\!\left\{\begin{array}{ll}
\!\!\!\{\bm{m}\in\mathbb{H}^{2\beta}: \frac{\partial\bm{m}}{\partial \nu}=0\},\quad &~2\beta>\frac{3}{2},\\
\!\!\!\mathbb{H}^{2\beta},\quad &~2\beta<\frac{3}{2}.
\end{array}\right.
\end{eqnarray*}
\end{proposition}
For each $\bm{\lambda}=\{\lambda_j\}_{j=1}^{\infty}$ and $s\in\mathbb{R}$, the Hilbert norm of $\bm{\Psi}$ on the measurable set $K\subset\mathbb{R}^3$ is defined as:
\begin{align*}
|\bm{\Psi}|_{\mathcal{H}_{\lambda}^{s}(K)}^{2}
:=\sum_{j=1}^{\infty}\lambda_{j}|\bm{\psi}_{j}|_{\mathbb{H}^{s}(K)}^2.
\end{align*}
If $|\bm{\Psi}|_{\mathcal{H}_{\lambda}^{s}(K)}<\infty$, then $\bm{\Psi}\in \mathcal{H}_{\lambda}^{s}(K)$, and we denote $\mathcal{L}_{\lambda}^{2}(K):=\mathcal{H}_{\lambda}^{0}(K)$.
\begin{assumption}
{\rm (a)} Let $(\Omega, \mathscr{F}, (\mathscr{F}_t)_{t\geq 0}, \mathbb{P})$ be a probability space with a filtration satisfying the usual assumptions, i.e., satisfying {\rm (i)} $\mathbb{P}$ is complete on $(\Omega, \mathscr{F})$, {\rm (ii)} for each $t\geq 0$, $\mathscr{F}_t$ contains all $(\mathscr{F}, \mathbb{P})$-null sets, and {\rm (iii)} $\mathscr{F}_t$ filtration is right-continuous;

{\rm (b)} Let $W(t)$ be an $\mathbb{R}^N$-valued, $\mathscr{F}_t$-adapted Wiener process defined on the above probability space;

{\rm (c)} Let $L(t)$ be an $\mathbb{R}^N$-valued, $\mathscr{F}_t$-adapted L\'{e}vy process with a pure jump defined on the above probability space, corresponding to a time homogenous Poisson random measure $\eta$;

{\rm (d)} $\bm{m}\times G(\bm{m})\in L^2(0,T;\mathbb{W}^{-1,2})$, and $\bm{m}\times F(\bm{m},l): \mathbb{H}^1\times\mathbb{H}^1\times B\rightarrow\mathbb{W}^{-2,2}$ is a measurable mapping.
\end{assumption}
\begin{assumption}\label{ass1}
Assume that there exist two positive constants $K_1$ and $K_2$, such that for all $\bm{m}_1, \bm{m}_2$ and $t\in[0,T]$, it holds that
\begin{align*}
&\bigg(\sum_{i\geq1}|G_i(\bm{m}_{1})-G_i(\bm{m}_{2})|\bigg)^2
+|G'(\bm{m}_{1})[G(\bm{m}_{1})]-G'(\bm{m}_{2})[G(\bm{m}_{2})]|^2\\
&+\int_B|F(\bm{m}_{1},l)-F(\bm{m}_{2},l)|^2\mu(dl)
\leq K_1|\bm{m}_1-\bm{m}_2|^2;
\end{align*}
For all $\bm{m}$, it holds that
\begin{align*}
\bigg(\sum_{i\geq1}|G_i(\bm{m})|\bigg)^2+|G'(\bm{m})[G(\bm{m})]|^2
+\int_B|F(\bm{m},l)|^2\mu(dl)\leq K_2(1+|\bm{m}|^2).
\end{align*}
\end{assumption}
\subsection{Main results} First, we introduce the definition of weak martingale solutions to the stochastic SPLLG system.
\begin{definition}
Let $\bar{\bm{\Psi}}_0\in\mathcal{H}_{\lambda}^{1}(\mathbb{R}^3)$, $\bar{\bm{m}}_0\in\mathbb{H}^{1}(D)$, $|\bar{\bm{m}}_0|=1$, a.e. in $D$. For any given $T\in(0,\infty)$, then $(\bar{\Omega}, \bar{\mathscr{F}}, (\bar{\mathscr{F}}_t)_{t\geq0}, \bar{\mathbb{P}}, \bar{\bm{\Psi}}, \bar{\bm{m}}, \bar{W}, \bar{\eta})$ is a weak martingale solution of the SPLLG system provided that the following hold:\\
(a) $(\bar{\Omega}, \bar{\mathscr{F}}, (\bar{\mathscr{F}}_t)_{t\geq 0}, \bar{\mathbb{P}})$ is a filtered probability space with a filtration $(\bar{\mathscr{F}}_t)_{t\geq 0}$;\\
(b) $\bar{W}$ is an $\bar{\mathscr{F}}_t$-adapted Wiener process; $\bar{\eta}$ is a time homogeneous Poisson random measure from the
measurable space $(B,\mathscr{B}(B))$ to $(\bar{\Omega}, \bar{\mathscr{F}}, (\bar{\mathscr{F}}_t)_{t\geq 0}, \bar{\mathbb{P}})$ with intensity ${\rm{Leb}}\otimes\mu$;\\
(c) $\bar{\bm{m}}:[0,T]\times\bar{\Omega}\rightarrow\mathbb{H}^1$ is $(\bar{\mathscr{F}_t)}_{t\geq 0}$ progressively measurable and weakly c\`{a}dl\`{a}g process, and for $\beta=\frac{1}{2}$, $\bar{\mathbb{P}}$-a.s. $\bar{\omega}\in\bar{\Omega}$,
\begin{align*}
\bar{\bm{m}}(\cdot,\bar{\omega})\in \mathbb{D}([0,T];X^{-\beta}),
\text{and} \sup_{t\in[0,T]}|\bar{\bm{m}}(t,\bar{\omega})|_{\mathbb{H}^1}<\infty,
\end{align*}
$\bar{\bm{\Psi}}:[0,T]\times\bar{\Omega}\rightarrow\mathcal{H}_{\lambda}^{1}$ is $(\bar{\mathscr{F}_t)}_{t\geq 0}$ progressively measurable and weakly c\`{a}dl\`{a}g process, and for $\beta=\frac{1}{2}$, $\bar{\mathbb{P}}$-a.s. $\bar{\omega}\in\bar{\Omega}$,
\begin{align*}
\bar{\bm{\Psi}}(\cdot,\bar{\omega})\in \mathbb{D}([0,T];X^{-\beta}),
\text{and} \sup_{t\in[0,T]}|\bar{\bm{\Psi}}(t,\bar{\omega})|_{\mathcal{H}_{\lambda}^{1}}<\infty,
\end{align*}
where $\mathbb{D}([0,T],X^{-\beta})$ represents Skorokhod space;\\
(d) $\bar{\bm{m}}(t,\bm{x})$ satisfies the constraint condition, i.e.
\begin{align*}
|\bar{\bm{m}}(t,\bm{x})|_{\mathbb{R}^3}=1,~\text{for Lebesgue}~a.e.~\bm{x}\in D, \text{for all}~t\in[0,T],~\bar{\mathbb{P}}\text{-}a.s.;
\end{align*}
(e) For all $\vartheta\in L^4(\bar{\Omega};L^4(0,T;\mathbb{H}^1(\mathbb{R}^3)))$ and $\chi\in L^4(\bar{\Omega};L^4(0,T;\mathbb{H}^1))$, we have
\begin{align*}
i\int_{0}^{T}\int_{\mathbb{R}^3}\partial_t\bar{\bm{\psi}}_j\vartheta d\bm{x}dt
=&\frac{1}{2}\int_{0}^{T}\int_{\mathbb{R}^3}\nabla\bar{\bm{\psi}}_j\cdot\nabla\vartheta d\bm{x}dt
+\int_{0}^{T}\int_{\mathbb{R}^3}\bar{V}\bar{\bm{\psi}}_j\vartheta d\bm{x}dt\\
&-\frac{1}{2}\int_{0}^{T}\int_{\mathbb{R}^3}\bar{\bm{m}}\cdot\hat{\bm{\sigma}}\bar{\bm{\psi}}_j\vartheta d\bm{x}dt,
~\bar{\mathbb{P}}\text{-}a.s.,
\end{align*}
and
\begin{align*}
&\int_{0}^{T}\int_{D}\partial_t\bar{\bm{m}}\cdot\chi d\bm{x}dt
=\int_{0}^{T}\int_{D}\alpha(\bar{\bm{m}}\times\partial_t\bar{\bm{m}})\cdot\chi d\bm{x}dt
-\int_{0}^{t}\int_{D}(\bar{\bm{m}}\times\bar{\bm{H}})\cdot\chi d\bm{x}dt\\
&+\frac{1}{2}\int_{0}^{T}\int_{D}(\bar{\bm{m}}\times G'(\bar{\bm{m}})[G(\bar{\bm{m}})])\cdot\chi d\bm{x}dt
+\int_{0}^{T}\int_{D}(\bar{\bm{m}}\times G(\bar{\bm{m}}))\cdot\chi d\bm{x}d\bar{W}(t)\\
&+\int_{0}^{T}\int_B\int_{D}(\bar{\bm{m}}\times F(\bar{\bm{m}},l))\cdot\chi d\bm{x}\tilde{\bar{\eta}}(dt,dl),~\bar{\mathbb{P}}\text{-}a.s.,
\end{align*}
where
\begin{align*}
\int_{0}^{T}\int_{D}(\bar{\bm{m}}\times \bar{\bm{H}})\cdot\chi d\bm{x}dt
&=\int_{0}^{T}\int_{D}\bar{\bm{m}}\times \big(\bar{\bm{H}}_s+\frac{1}{2}\bar{\bm{s}}-w'(\bar{\bm{m}})\big)\cdot\chi d\bm{x}dt\\
&\quad -\int_{0}^{T}\int_{D}(\bar{\bm{m}}\times \nabla\bar{\bm{m}})\cdot\nabla\chi d\bm{x}dt,
\end{align*}
and $\bar{V}, \bar{\rho}, \bar{\bm{s}}$ and $\bar{\bm{H}}_s$ are defined in \eqref{ss}.
\end{definition}

Now, we are ready to state the main results of this paper.
\begin{theorem}\label{them1}
Assume that $D$ is a bounded domain with smooth boundary. Let $\bar{\bm{\Psi}}_0\in\mathcal{H}_{\lambda}^{1}(\mathbb{R}^3)$ and $\bar{\bm{m}}_0\in \mathbb{H}^{1}(D)$. For any given $0<T<\infty$,  then the stochastic SPLLG system \eqref{S1} admits a weak martingale solution $(\bar{\Omega}, \bar{\mathscr{F}}, (\bar{\mathscr{F}}_t)_{t\geq0}, \bar{\mathbb{P}},\bar{\bm{\Psi}}, \bar{\bm{m}}, \bar{W}, \bar{\eta})$ on $[0,T]$.
\end{theorem}
In the coupled stochastic SPLLG system, we first consider the Schr\"{o}dinger equation in $K=\{\bm{x}\in\mathbb{R}^3, |\bm{x}|<R\}$. And then one can obtain the existence of solutions in $\mathbb{R}^3$ using the domain extension method. Thus, for each $\bm{\lambda}=\{\lambda_j\}_{j=1}^{\infty}$, we study the following system:
\begin{eqnarray}\label{S2}
\left\{\begin{array}{ll}
\!\!\!i\partial_t\bm{\psi}_j(\bm{x},t)=-\frac{1}{2}\Delta\bm{\psi}_j(\bm{x},t)+V\bm{\psi}_j(\bm{x},t)
-\frac{1}{2}\bm{m}\cdot\hat{\bm{\sigma}}\bm{\psi}_j(\bm{x},t),~ j\in \mathbb{N},(\bm{x},t)\in K\times\mathbb{R}^+,\\
\!\!\!\partial_t\bm{m}=-\bm{m}\times\bm{H}+\alpha \bm{m}\times\partial_t\bm{m}
+(\bm{m}\times G(\bm{m}))\circ \frac{dW}{dt}+(\bm{m}\times F(\bm{m},l))\frac{dL(t)}{dt},(\bm{x},t)\in D\times\mathbb{R}^+,\\
\!\!\!-\Delta V=\rho[\bm \Psi],~(\bm{x},t)\in K\times\mathbb{R}^{+},\\
\!\!\!\partial_\nu\bm{m}=0,~(\bm{x},t)\in \partial D\times\mathbb{R}^{+}
\end{array}\right.
\end{eqnarray}
with $\bm{m}(\bm{x},0)=\bm{m}_{0}(\bm{x})$, $\bm{x}\in D$; $\bm{\psi}_j(\bm{x},0)=\bm{\varphi}_{j}(\bm{x})$, $\bm{x}\in K$; $\bm{\psi}_j(\bm{x},t)=0$, $V(\bm{x},t)=0$, $(\bm{x},t)\in\partial K\times\mathbb{R}^{+}$. Here $G:\mathbb{H}^1\rightarrow\mathbb{H}^1$, $F:\mathbb{H}^1\times B\rightarrow\mathbb{H}^1$ are measurable mappings. $\int_{0}^{t}\bm{m}\times G(\bm{m})dW(s)$ is a well-defined $(\mathscr{F}_t)$-martingale, taking values in the space $\mathbb{W}^{-1,2}$. In fact, let $W(t)=\sum_{i\geq1}W_i(t)\tilde{e}_i$, where $\{W_i\}_{i\geq1}$ are real-valued standard independent Wiener processes, $\{\tilde{e}_i\}_{i\geq1}$ are complete orthonormal functions in a separable Hilbert space $\mathbb{U}$. If $\bm{m}\in L^4(\Omega;L^\infty(0,T;\mathbb{H}^1))$, according to the embedding $\mathbb{L}^2\hookrightarrow\mathbb{W}^{-1,2}$ and Assumption \ref{ass1}, we can infer that
\begin{align*}
&\big|(\bm{m}\times G(\bm{m}))\big|_{L^2(\mathbb{U},\mathbb{W}^{-1,2})}^2
=\sum\limits_{i\geq1}|(\bm{m}\times G(\bm{m}))\tilde{e}_i|_{\mathbb{W}^{-1,2}}^2\\
&\leq\sum\limits_{i\geq1}|(\bm{m}\times G(\bm{m}))\tilde{e}_i|_{\mathbb{L}^2}^2
\leq C(1+|\bm{m}|_{\mathbb{L}^4}^4),
\end{align*}
that is the mapping $\bm{m}\times G(\bm{m})$ belongs to $L^2(\mathbb{U},\mathbb{W}^{-1,2})$, where $L^2(\mathbb{U},\mathbb{W}^{-1,2})$ represents the space of Hilbert-Schmidt operators from $\mathbb{U}$ to $\mathbb{W}^{-1,2}$; thus we get
\begin{align*}
&\sup\limits_{t\in[0,T]}\mathbb{E}\bigg[\bigg|\int_{0}^{t}\bm{m}\times G(\bm{m})dW(s)\bigg|_{\mathbb{W}^{-1,2}}^2\bigg]\\
&\leq\mathbb{E}\bigg[\sup\limits_{t\in[0,T]}\bigg|\sum\limits_{i\geq1}\int_{0}^{t}(\bm{m}\times G(\bm{m}))\tilde{e}_idW_i(s)\bigg|_{\mathbb{W}^{-1,2}}^2\bigg]\\
&\leq C\mathbb{E}\bigg[\int_0^T\big|(\bm{m}\times G(\bm{m}))\big|_{L^2(\mathbb{U},\mathbb{W}^{-1,2})}^2dt\bigg]
\leq C\mathbb{E}\bigg[\int_0^T\big(1+|\bm{m}|_{\mathbb{L}^4}^4\big)dt\bigg]
<\infty.
\end{align*}
In addition, we can deduce that $\bm{m}\times F(\bm{m},l): \mathbb{H}^1\times\mathbb{H}^1\times B\rightarrow\mathbb{W}^{-2,2}$ is a measurable mapping, $\int_{0}^{t}\bm{m}\times F(\bm{m},l)dL(s)$ thus is well-defined. Indeed, by Assumption \ref{ass1}, we have
\begin{align*}
&\sup\limits_{t\in[0,T]}\mathbb{E}\bigg[\bigg|\int_{0}^{t}\bm{m}\times F(\bm{m},l)dL(s)\bigg|_{\mathbb{W}^{-2,2}}^2\bigg]
\leq\mathbb{E}\bigg[\sup\limits_{t\in[0,T]}\bigg|\int_{0}^{t}\bm{m}\times F(\bm{m},l)dL(s)\bigg|_{\mathbb{W}^{-2,2}}^2\bigg]\\
&\leq C\mathbb{E}\bigg[\int_0^T\int_B|\bm{m}\times F(\bm{m},l)|_{\mathbb{W}^{-2,2}}^2\mu(dl)dt\bigg]
\leq C\mathbb{E}\bigg[\int_0^T\big(1+|\bm{m}|_{\mathbb{L}^2}^4\big)dt\bigg]
<\infty.
\end{align*}
Suppose that the initial data satisfies $|\bm{m}_0(\bm{x})|\equiv1$, a.e. $D$, $\mathbb{P}$-a.s., and $\bm{\Psi}\equiv\bm{0}$ on $({\mathbb{R}^3}\setminus{K})\times\mathbb{R}^+$; $\bm{m}\equiv\bm{0}$ on $({\mathbb{R}^3}\setminus{\bar{D}})\times\mathbb{R}^+$.
\begin{theorem}\label{youjieyudejie}
Let $D$ be a bounded domain with smooth boundary, and let $K\subset\mathbb{R}^3$ be a sufficiently large ball such that $D\subset K$. Suppose that $\bar{\bm{\Psi}}_0\in\mathcal{H}_{\lambda}^{1}(K)$ and $\bar{\bm{m}}_0\in \mathbb{H}_{1}(D)$, then the stochastic SPLLG system \eqref{S2} admits a weak martingale solution $(\bar{\Omega}, \bar{\mathscr{F}}, (\bar{\mathscr{F}}_t)_{t\geq 0}, \bar{\mathbb{P}}, \bar{\bm{\Psi}}, \bar{\bm{m}}, \bar{W}, \bar{\eta})$ on $[0,T]$ for any given $0<T<\infty$.
\end{theorem}
Motivated by \cite{AS, CGY}, we consider the following stochastic system with the penalty function term:
\begin{eqnarray}\label{SLLG3}
\left\{\begin{array}{ll}
\!\!\!i\partial_t\bm{\psi}_j(\bm{x},t)=-\frac{1}{2}\Delta\bm{\psi}_j(\bm{x},t)+V\bm{\psi}_j(\bm{x},t)
-\frac{1}{2}\bm{m}\cdot\hat{\bm{\sigma}}\bm{\psi}_j(\bm{x},t),~ j\in \mathbb{N},\\
\!\!\!\alpha\partial_t\bm{m}=-\bm{m}\times\partial_t\bm{m}+\bm{H}-k(|\bm{m}|^2-1)\bm{m}-G(\bm{m})\circ
\frac{dW}{dt}-F(\bm{m},l)\frac{dL(t)}{dt},\\
\!\!\!-\Delta V=\rho[\bm \Psi],
\end{array}\right.
\end{eqnarray}
where $k>0$ is the penalization coefficient.

Given the $\mathscr{F}_0$-measurable random variable $\bm{m}_0$, we consider the case with $\mu=0$ and $\eta=0$ on $B^c$, i.e., neglecting the large jumps and only studying the case of the small jumps. Then, the integral form of the second equation of system \eqref{SLLG3} is:
\begin{align*}
\alpha\bm{m}(t)=&\alpha\bm{m}(0)-\int_{0}^{t}(\bm{m}(s)\times\partial_s\bm{m}(s))ds
+\int_{0}^{t}\bm{H}(s)ds-\int_{0}^{t}k(|\bm{m}(s)|^2-1)\bm{m}(s)ds\\
&-\int_{0}^{t}G(\bm{m}(s))\circ dW(s)-\int_{0}^{t}F(\bm{m}(s),l)dL(s)\\
=&\alpha\bm{m}(0)-\int_{0}^{t}(\bm{m}(s)\times\partial_s\bm{m}(s))ds
+\int_{0}^{t}\bm{H}(s)ds-\int_{0}^{t}k(|\bm{m}(s)|^2-1)\bm{m}(s)ds\nonumber\\
&-\int_{0}^{t}G(\bm{m}(s))dW(s)-\frac{1}{2}\int_{0}^{t}G'(\bm{m}(s))[G(\bm{m}(s))]ds
-\int_{0}^{t}\int_BF(\bm{m}(s-),l)\tilde{\eta}(ds,dl).
\end{align*}
\section{Faedo-Galerkin approximation and energy estimates}
We construct an approximate solution sequence using the Faedo-Galerkin approximation method, transforming it from an infinite-dimensional space to a finite-dimensional space.
\subsection{The Faedo-Galerkin approximation}
Let $\{\theta_h\}_{h\in\mathbb{N}}$ be the normalized characteristic functions of the equation
\begin{eqnarray}\label{bzh1}
\left\{\begin{array}{ll}
\!\!\!-\Delta\theta=\tilde{\lambda}\theta,~&x\in K,\\
\!\!\!\theta=0,~&x\in\partial K.
\end{array}\right.
\end{eqnarray}
Let $\{e_h\}_{h\in\mathbb{N}}$ be the normalized characteristic functions of the equation
\begin{eqnarray}\label{bzh2}
\left\{\begin{array}{ll}
\!\!\!-\Delta e=\tilde{\lambda}e, ~&x\in D, \\
\!\!\!\partial_{\nu}e=0,~&x\in\partial D.
\end{array}\right.
\end{eqnarray}
where $\mathbb{H}_n^{K}:=\text{linspan}\{\theta_1, \theta_2, \cdots, \theta_n\}$, $\mathbb{H}_n^{D}:=\text{linspan}\{e_1, e_2, \cdots, e_n\}$,
and $\mathbb{H}_n=\mathbb{H}_n^{K}\times\mathbb{H}_n^{D}$.
Define orthogonal projections $\pi_{n}^{K}$ and $\pi_{n}^{D}$ respectively by
\begin{align*}
\pi_{n}^{K}\bm{u}=\sum_{h=1}^{n}(\bm{u}, \theta_h)_{\mathbb{L}^2(K)}\theta_h,~
\pi_{n}^{D}\bm{v}=\sum_{h=1}^{n}(\bm{v}, e_h)_{\mathbb{L}^2(D)}e_h,~\forall\bm{u}\in \mathbb{H}^1(K),\bm{v}\in \mathbb{H}^1(D)
\end{align*}
for all $h\in\mathbb{N}$, here $\theta_h\in C^{\infty}(\bar{K})$, $e_h\in C^{\infty}(\bar{D})$.
Considering the approximation solutions $\bm{\Psi}_{n}^{k}:=\{\bm{\psi}_{jn}^k\}_{j\in \mathbb{N}}$ and $\bm{m}_n^k$,
to further simplify the notation, we also denote $\bm{\Psi}_{n}:=\bm{\Psi}_{n}^{k}$ and $\bm{m}_n:=\bm{m}_n^k$. Thus,
\begin{align}
\bm{\psi}_{jn}(\bm{x},t)=\pi_{n}^{K}\bm{\psi}_{j}=\sum_{h=1}^{n}(\bm{\psi}_{j}, \theta_h)_{\mathbb{L}^2(K)}\theta_h
=\sum_{h=1}^{n}\bm{\alpha}_{jh}(t)\theta_h(\bm{x}),\label{jn}\\
\bm{m}_{n}(\bm{x},t)=\pi_{n}^{D}\bm{m}=\sum_{h=1}^{n}(\bm{m}, e_h)_{\mathbb{L}^2(D)}e_h
=\sum_{h=1}^{n}\bm{\beta}_{h}(t)e_h(\bm{x}),\label{n}
\end{align}
where $\bm{\alpha}_{jh}$ and $\bm{\beta}_{h}$ are two-dimensional and three-dimensional vector-valued functions. Notice that $\bm{\Psi}_{n}$ and $\bm{m}_n$ satisfy the following equations, respectively:
\begin{eqnarray}\label{S3}
\left\{\begin{array}{ll}
\!\!\!(i\partial_t\bm{\psi}_{jn}+\frac{1}{2}\Delta\bm{\psi}_{jn}-V_n\bm{\psi}_{jn}
+\frac{1}{2}\bm{m}_n\cdot\hat{\bm{\sigma}}\bm{\psi}_{jn},\theta_h)=0,\\
\!\!\!\bm{\psi}_{jn}(\bm{x},0)=\Pi_{n}^{K}\bm{\varphi}_{j}(\bm{x}),
\end{array}\right.
\end{eqnarray}
and
{\small
\begin{eqnarray}\label{SLLG4}
\left\{\begin{array}{ll}
\!\!\!\big(\alpha\partial_t\bm{m}_n+\bm{m}_n\times\partial_t\bm{m}_n-\bm{H}_n+k(|\bm{m}_n|^2-1)\bm{m}_n
+G_n(\bm{m}_n)\circ \frac{dW}{dt}+F_n(\bm{m}_n,l)\frac{dL(t)}{dt},e_h\big)=0,\\
\!\!\!\bm{m}_n(\bm{x},0)=\pi_{n}^{D}\bm{m}_{0}(\bm{x}),
\end{array}\right.
\end{eqnarray}}
where $V_n$ satisfies
\begin{eqnarray*}
\left\{\begin{array}{ll}
\!\!\!-\Delta V_n=\rho_n,\\
\!\!\!V_n|_{\partial K}=0,
\end{array}\right.
\end{eqnarray*}
\begin{align*}
&\bm{H}_n=\Delta\bm{m}_n-w'(\bm{m}_n)+\bm{H}_{sn}+\frac{1}{2}\bm{s}_n,~
\bm{H}_{sn}=-\nabla(\nabla N\ast\cdot\bm{m}_n),\\
&\rho_n=\sum_{j=1}^{\infty}\lambda_{j}|\bm{\psi}_{jn}|^2,~
\bm{s}_n=\sum_{j=1}^{\infty}\lambda_{j}{\rm{Tr}}_{{\mathbb{C}}^2}(\hat{\bm{\sigma}}({\bm{\psi}_{jn}}
{\bm{\psi}^{\dag}_{jn}})).
\end{align*}
We know by transformation that the equivalent ordinary differential equation of \eqref{SLLG4} has the following form:
\begin{align*}
\partial_t\bm{m}_n
=\left[\alpha \bm{I}+\begin{pmatrix}0 & -m_{n,3} & m_{n,2} \\ m_{n,3} & 0 & -m_{n,1} \\ -m_{n,2} & m_{n,1} & 0\end{pmatrix}\right]^{-1}
f(\bm{m}_n,t)
=:[\alpha \bm{I}+\bm{M}]^{-1}f(\bm{m}_n,t),
\end{align*}
where
\begin{align*}
f(\bm{m}_n,t)=\bm{H}_n-k(|\bm{m}_n|^2-1)\bm{m}_n
-G_n(\bm{m}_n)\circ \frac{dW}{dt}-F_n(\bm{m}_n,l)\frac{dL(t)}{dt}.
\end{align*}
and according to the norm of a matrix, we infer that $[\alpha I+M]^{-1}$ is bounded. In fact,
\begin{align*}
[\alpha \bm{I}+\bm{M}]^{-1}=
\big(\alpha^3+\alpha|\bm{m}|^2\big)^{-1}
\begin{pmatrix}\alpha^2+m_{n,1}^2 & m_{n,1}m_{n,2}-\alpha m_{n,3} & m_{n,1}m_{n,3}-\alpha m_{n,2} \\
m_{n,1}m_{n,2}+\alpha m_{n,3} & \alpha^2+m_{n,2}^2 & m_{n,2}m_{n,3}+\alpha m_{n,1} \\
m_{n,1}m_{n,3}+\alpha m_{n,2} & m_{n,2}m_{n,3}-\alpha m_{n,1} & \alpha^2+m_{n,3}^2\end{pmatrix},
\end{align*}
then applying $|A|_{\mathbb{L}^{\infty}}=\max\limits_{1\leq i\leq n}\sum\limits_j^n|a_{ij}|$, where $a_{ij}$  represents the element in the ith row and jth column of the matrix, we deduce that
\begin{align*}
\big|[\alpha \bm{I}+\bm{M}]^{-1}\big|_{\mathbb{L}^{\infty}}\leq\big(\alpha^3+\alpha|\bm{m}|^2\big)^{-1}
\cdot2(\alpha^2+|\bm{m}|^2\big)\leq\frac{2}{\alpha}.
\end{align*}
Define the mapping:
\begin{align*}
F_n^1:&~\mathbb{H}_n^{K}\ni\bm{\psi}_j\mapsto\Delta\bm{\psi}_j\in\mathbb{H}_n^{K},~~
F_n^2:~\mathbb{H}_n^{K}\ni\bm{\psi}_j\mapsto V\bm{\psi}_j\in\mathbb{H}_n^{K},\\
F_n^3:&~\mathbb{H}_n^{K}\times\mathbb{H}_n^{D}\ni(\bm{\psi}_j,\bm{m})\mapsto
\bm{m}\cdot\hat{\bm{\sigma}}\bm{\psi}_j\in\mathbb{H}_n,~~
G_n^1:~\mathbb{H}_n^{D}\times\mathbb{H}_n^{K}\ni(\bm{m},\bm{\psi}_j)\mapsto\bm{H}\in\mathbb{H}_n,\\
G_n^2:&~\mathbb{H}_n^{D}\ni\bm{m}\mapsto(|\bm{m}|^2-1)\bm{m}\in\mathbb{H}_n^{D}.
\end{align*}
Consider the following equations:
\begin{align}\label{jS3}
i\bm{\psi}_{jn}(t)
=i\bm{\psi}_{jn}(0)-\int_{0}^{t}\bigg(\frac{1}{2}F_n^1(\bm{\psi}_{jn}(s))-F_n^2(\bm{\psi}_{jn}(s))
+\frac{1}{2}F_n^3(\bm{\psi}_{jn}(s),\bm{\bm{m}_n}(s))\bigg)ds,
\end{align}
\begin{align}\label{jSLLG4}
\bm{m}_n(t)&=\bm{m}_n(0)+\int_{0}^{t}[\alpha \bm{I}+\bm{M}]^{-1}\left(G_n^1(\bm{m}_n(s),\bm{\psi}_{jn}(s))
-G_n^2(\bm{m}_n(s))\right)ds\nonumber\\
&\quad -\frac{1}{2}\int_{0}^{t}[\alpha \bm{I}+\bm{M}]^{-1}G'_{n}(\bm{m}_n(s))[G_{n}(\bm{m}_n(s))]ds
-\int_{0}^{t}[\alpha \bm{I}+\bm{M}]^{-1}G_{n}(\bm{m}_n(s))dW(s)\nonumber\\
&\quad -\int_{0}^{t}\int_B[\alpha \bm{I}+\bm{M}]^{-1}F_{n}(\bm{m}_n(s-),l)\tilde{\eta}(ds,dl)\\
&=\alpha\bm{m}_n(0)-\int_{0}^{t}[\alpha \bm{I}+\bm{M}]^{-1}\left(G_n^1(\bm{m}_n(s),\bm{\psi}_{jn}(s))
-G_n^2(\bm{m}_n(s))\right)ds\nonumber\\
&\quad -\frac{1}{2}\int_{0}^{t}[\alpha \bm{I}+\bm{M}]^{-1}G'_{n}(\bm{m}_n(s))[G_{n}(\bm{m}_n(s))]ds
-\int_{0}^{t}[\alpha \bm{I}+\bm{M}]^{-1}\sum\limits_{i\geq1}G_{n}(\bm{m}_{n}(s))\tilde{e}_idW_i(s)\nonumber\\
&\quad -\int_{0}^{t}\int_B[\alpha \bm{I}+\bm{M}]^{-1}F_{n}(\bm{m}_n(s-),l)\tilde{\eta}(ds,dl),\quad t\geq0.\nonumber
\end{align}
For proving the existence of solutions of \eqref{S3} and \eqref{SLLG4}, we will show that $F_n^1$-$F_n^3$, $G_n^1$-$G_n^2$, $G_n$, $G'_n[G_n]$ and $F_n$ are Lipschitz.
\begin{lemma}\label{1}
For each $n\in\mathbb{N}$, the mapping $F_n^1$ is global Lipschitz, as well as $F_n^2$ and $F_n^3$ are local Lipschitz.
\end{lemma}
\begin{proof}
Fixed $\bm{\psi}_j\in\mathbb{H}_n^{K}$, then $\Delta\bm{\psi}_j\in\mathbb{H}_n^{K}$. According to \eqref{bzh1},
for any $\bm{\psi}_{j1}, \bm{\psi}_{j2}\in\mathbb{H}_n^{K}$, we obtain
\begin{align*}
|F_n^1(\bm{\psi}_{j1})-F_n^1(\bm{\psi}_{j2})|_{\mathbb{L}^2}
=|\Delta\bm{\psi}_{j1}-\Delta\bm{\psi}_{j2}|_{\mathbb{L}^2}
\leq\sum\limits_{h=1}^{n}\tilde{\lambda}_h|\bm{\psi}_{j1}-\bm{\psi}_{j2}|_{\mathbb{L}^2}.
\end{align*}
Hence, $F_n^1$ is global Lipschitz. Furthermore, employing the Poincar\'{e} inequality,
Young inequality for convolution, and the Calder\'{o}n-Zygmund inequality (see \cite{CW}), we get
\begin{align*}
|V|_{\mathbb{L}^2}
\leq C|{\nabla}^2 V|_{\mathbb{L}^2}
\leq C|{\nabla}^2 N|_{\mathbb{L}^2}|\rho[\bm{\Psi}]|_{\mathbb{L}^1}
\leq C\sum\limits_{j=1}^{\infty}\lambda_j|\bm{\psi}_j|^2_{\mathbb{L}^2}.
\end{align*}
Using the equivalence of norms on finite-dimensional spaces, we compute
\begin{align*}
&|F_n^2(\bm{\psi}_{j1})-F_n^2(\bm{\psi}_{j2})|_{\mathbb{L}^2}
=|V_1\bm{\psi}_{j1}-V_2\bm{\psi}_{j2}|_{\mathbb{L}^2}\\
&\leq|V_1-V_2|_{\mathbb{L}^2}|\bm{\psi}_{j1}|_{\mathbb{L}^{\infty}}
+|\bm{\psi}_{j1}-\bm{\psi}_{j2}|_{\mathbb{L}^2}|V_2|_{\mathbb{L}^{\infty}}\\
&\leq C\sum\limits_{j=1}^{\infty}\lambda_j
(|\bm{\psi}_{j1}-\bm{\psi}_{j2}|_{\mathbb{L}^2}|\bm{\psi}_{j1}|_{\mathbb{L}^{\infty}}
+|\bm{\psi}_{j2}|^2_{\mathbb{L}^{\infty}})|\bm{\psi}_{j1}-\bm{\psi}_{j2}|_{\mathbb{L}^2}\\
&\leq C\sum\limits_{j=1}^{\infty}\lambda_j|\bm{\psi}_{j1}-\bm{\psi}_{j2}|_{\mathbb{L}^2}
\end{align*}
for $|\bm{\psi}_{j1}|, |\bm{\psi}_{j2}|\leq R$, which implies that $F_n^2$ is local Lipschitz. Moreover, for any $\bm{m}\in\mathbb{H}_n^{D}$ and $|\bm{m}|\leq R$, we have
\begin{align*}
&|F_n^3(\bm{\psi}_{j1},\bm{m})-F_n^3(\bm{\psi}_{j2},\bm{m})|_{\mathbb{L}^2}
=|\bm{m}\cdot\hat{\bm{\sigma}}\bm{\psi}_{j1}-\bm{m}\cdot\hat{\bm{\sigma}}\bm{\psi}_{j2}|_{\mathbb{L}^2}\\
&\leq C|\bm{m}|_{\mathbb{L}^{\infty}}|\bm{\psi}_{j1}-\bm{\psi}_{j2}|_{\mathbb{L}^2}
\leq C|\bm{\psi}_{j1}-\bm{\psi}_{j2}|_{\mathbb{L}^2}.
\end{align*}
Thus, $F_n^3$ is Lipschitz with respect to $\bm{\psi}_j$.
\end{proof}
\begin{lemma}\label{2}
For $n\in\mathbb{N}$, $G_n^2$ is local Lipschitz, as well as $G_n^1$, $G_n$, $G'_n[G_n]$, and $F_n$ are global Lipschitz.
\end{lemma}
\begin{proof}
For any $\bm{m}_1, \bm{m}_2\in\mathbb{H}_n^{D}$, Since
\begin{align*}
&|G_n^1(\bm{m}_1,\bm{\psi}_j)-G_n^1(\bm{m}_2,\bm{\psi}_j)|_{\mathbb{L}^2}
=|\bm{H}_1-\bm{H}_2|_{\mathbb{L}^2}\\
&\leq|\Delta(\bm{m}_1-\bm{m}_2)|_{\mathbb{L}^2}+|w'(\bm{m}_1)-w'(\bm{m}_2)|_{\mathbb{L}^2}
+\frac{1}{2}|\bm{s}[\bm{\Psi}]-\bm{s}[\bm{\Psi}]|_{\mathbb{L}^2}\\
&\quad +|\nabla(\nabla N\ast\cdot\bm{m}_1)-\nabla(\nabla N\ast\cdot\bm{m}_2)|_{\mathbb{L}^2},
\end{align*}
then using \eqref{bzh2}, we get $|\Delta(\bm{m}_1-\bm{m}_2)|_{\mathbb{L}^2}\leq C|\bm{m}_1-\bm{m}_2|_{\mathbb{L}^2}$.
Furthermore, by the definition of $w'(\bm{m})$, we have
\begin{align*}
|w'(\bm{m}_1)-w'(\bm{m}_2)|_{\mathbb{L}^2}=&|(0,2m_{1,2},2m_{1,3})-(0,2m_{2,2},2m_{2,3})|_{\mathbb{L}^2}
\leq2|\bm{m}_1-\bm{m}_2|_{\mathbb{L}^2}.
\end{align*}
Next, applying the differential property of convolution, we deduce
\begin{align*}
|\nabla(\nabla N\ast\cdot\bm{m}_1)-\nabla(\nabla N\ast\cdot\bm{m}_2)|_{\mathbb{L}^2}
=|\nabla^2 N\ast\cdot(\bm{m}_1-\bm{m}_2)|_{\mathbb{L}^2}
\leq C|\bm{m}_1-\bm{m}_2|_{\mathbb{L}^2}.
\end{align*}
Thus, $|G_n^1(\bm{m}_1,\bm{\psi}_{j})-G_n^1(\bm{m}_2,\bm{\psi}_{j})|_{\mathbb{L}^2}\leq C|\bm{m}_1-\bm{m}_2|_{\mathbb{L}^2}$. Therefore, $G_n^1$ is global Lipschitz with respect to $\bm{m}$. Notice that for $|\bm{m}_1|, |\bm{m}_2|\leq R$, we have
\begin{align*}
&|G_n^2(\bm{m}_1)-G_n^2(\bm{m}_2)|_{\mathbb{L}^2}
=|k(|\bm{m}_1|^2-1)\bm{m}_1-k(|\bm{m}_2|^2-1)\bm{m}_2|_{\mathbb{L}^2}\\
&\leq k|(|\bm{m}_1|^2-|\bm{m}_2|^2)\bm{m}_1|_{\mathbb{L}^2}+k|(|\bm{m}_2|^2-1)(\bm{m}_1-\bm{m}_2)|_{\mathbb{L}^2}\\
&\leq k|\bm{m}_1-\bm{m}_2|_{\mathbb{L}^2}^2|\bm{m}_1|_{\mathbb{L}^{\infty}}
+k||\bm{m}_2|^2-1|_{\mathbb{L}^{\infty}}|\bm{m}_1-\bm{m}_2|_{\mathbb{L}^2},
\end{align*}
hence according to the equivalence of norms on finite-dimensional spaces, we infer that $G_n^2$ is local Lipschitz. By Assumption \ref{ass1}, we have
\begin{align*}
&\bigg|\sum\limits_{i\geq1}\big(G_{n}(\bm{m}_{1}(s))\tilde{e}_i-G_{n}(\bm{m}_{2}(s))\tilde{e}_i\big)\bigg|_{\mathbb{L}^2}^2
+|G'_n(\bm{m}_1)[G_n(\bm{m}_1)]-G'_n(\bm{m}_2)[G_n(\bm{m}_2)]|_{\mathbb{L}^2}^2\\
&\quad +\int_B|F_n(\bm{m}_1,l)-F_n(\bm{m}_2,l)|_{\mathbb{L}^2}^2\mu(dl)
\leq K_1|\bm{m}_1-\bm{m}_2|_{\mathbb{L}^2}^2.
\end{align*}
Therefore, $G_n$, $G'_n[G_n]$, and $F_n$ are global Lipschitz.
\end{proof}
Since \eqref{S3} and \eqref{SLLG4} are equivalent to ordinary differential equations in $\mathbb{R}^n$.
Lemma \ref{1} and Lemma \ref{2} imply that the existence of local solutions $\bm{\Psi}_n=\{\bm{\psi}_{jn}\}_{j\in\mathbb{N}}$ and $\bm{m}_n$.
\subsection{A priori estimates}
\begin{lemma}\label{P}
Let $T\in(0,\infty)$, $2\leq p<\infty$, then for each $n=1,2,\cdots$, $t\in[0,T]$, we have
\begin{align*}
\mathbb{E}\big[|\bm{\psi}_{jn}(t,\cdot)|_{\mathbb{L}^2}^p\big]
=\mathbb{E}\big[|\bm{\psi}_{jn}(0,\cdot)|_{\mathbb{L}^2}^p\big]
=\mathbb{E}\big[|\pi_{n}^{K}\bm{\varphi}_{j}|_{\mathbb{L}^2}^p\big].
\end{align*}
\end{lemma}
\begin{proof}
Multiplying the both sides of \eqref{S3} by $\sum_{h=1}^{n}\bm{\alpha}_{jh}^{\dag}(t)$, we get
\begin{align*}
&\sum\limits_{h=1}^{n}\bm{\alpha}_{jh}^{\dag}\bigg(\int_{K}i\partial_t\bm{\psi}_{jn}\theta_h
+\frac{1}{2}\Delta\bm{\psi}_{jn}\theta_h-V_n\bm{\psi}_{jn}\theta_h
+\frac{1}{2}\bm{m}_n\cdot\hat{\bm{\sigma}}\bm{\psi}_{jn}\theta_hd\bm{x}\bigg)=0.
\end{align*}
Since
\begin{align}\label{P2}
&\partial_t(\bm{\alpha}^{\dag}_{jh}\cdot\bm{\alpha}_{jh})
=2{\rm Re}(\bm{\alpha}^{\dag}_{jh}\cdot\partial_t\bm{\alpha}_{jh})
=2{\rm Re}(\partial_t\bm{\alpha}^{\dag}_{jh}\cdot\bm{\alpha}_{jh}),\nonumber\\
&{\rm Im}(\bm{\alpha}^{\dag}_{jh}\cdot\partial_t\bm{\alpha}_{jh})
=-{\rm Im}(\partial_t\bm{\alpha}^{\dag}_{jh}\cdot\bm{\alpha}_{jh}),
\end{align}
it follows from \eqref{jn} that
\begin{align*}
&i\int_0^t\int_{K}\sum\limits_{h=1}^{n}\bm{\alpha}_{jh}^{\dag}\partial_s\bm{\psi}_{jn}\theta_hd\bm{x}ds
=i\int_0^t\int_{K}\sum\limits_{h=1}^{n}\bm{\alpha}_{jh}^{\dag}\partial_s\bm{\alpha}_{jh}\theta_h^2d\bm{x}ds\\
&=\frac{i}{2}\int_0^t\frac{d}{ds}\int_{K}\sum\limits_{h=1}^{n}|\bm{\alpha}_{jh}|^2\theta_h^2d\bm{x}ds
-\int_0^t\int_{K}\sum\limits_{h=1}^{n}{\rm Im}
(\bm{\alpha}^{\dag}_{jh}\cdot\partial_s\bm{\alpha}_{jh})\theta_h^2d\bm{x}ds\\
&=\frac{i}{2}\int_0^t\frac{d}{ds}\int_{K}|\bm{\psi}_{jn}|^2d\bm{x}ds
-\int_0^t\int_{K}\sum\limits_{h=1}^{n}{\rm Im}
(\bm{\alpha}^{\dag}_{jh}\cdot\partial_s\bm{\alpha}_{jh})\theta_h^2d\bm{x}ds.
\end{align*}
Using the integration by parts, we get
\begin{align*}
&\frac{1}{2}\int_0^t\int_{K}\sum\limits_{h=1}^{n}\bm{\alpha}_{jh}^{\dag}\Delta\bm{\psi}_{jn}\theta_hd\bm{x}ds
=-\frac{1}{2}\int_0^t\int_{K}\sum\limits_{h=1}^{n}
\bm{\alpha}_{jh}^{\dag}\bm{\alpha}_{jh}|\nabla\theta_{h}|^2d\bm{x}ds\\
&=-\frac{1}{2}\int_0^t\int_{K}\nabla\bm{\psi}^{\dag}_{jn}\nabla\bm{\psi}_{jn}d\bm{x}ds
=-\frac{1}{2}\int_0^t\int_{K}|\nabla\bm{\psi}_{jn}|^2d\bm{x}ds.
\end{align*}
In view of \eqref{jn}, we deduce
\begin{align*}
&-\int_0^t\int_{K}\sum\limits_{h=1}^{n}\bm{\alpha}_{jh}^{\dag}V_n\bm{\psi}_{jn}\theta_hd\bm{x}ds
=-\int_0^t\int_{K}\sum\limits_{h=1}^{n}\bm{\alpha}_{jh}^{\dag}\theta_h
V_n\sum\limits_{h=1}^{n}\bm{\alpha}_{jh}\theta_hd\bm{x}ds\\
&=-\int_0^t\int_{K}\sum\limits_{h=1}^{n}|\bm{\alpha}_{jh}|^2\theta_h^2V_nd\bm{x}ds
=-\int_0^t\int_{K}V_n|\bm{\psi}_{jn}|^2d\bm{x}ds,
\end{align*}
and
\begin{align*}
&\frac{1}{2}\int_0^t\int_{K}\sum\limits_{h=1}^{n}\bm{\alpha}_{jh}^{\dag}
\bm{m}_n\cdot\hat{\bm{\sigma}}\bm{\psi}_{jn}\theta_hd\bm{x}ds
=\frac{1}{2}\int_0^t\int_{K}\sum\limits_{h=1}^{n}\bm{\alpha}_{jh}^{\dag}\theta_h
\bm{m}_n\cdot\hat{\bm{\sigma}}\sum\limits_{h=1}^{n}\bm{\alpha}_{jh}\theta_hd\bm{x}ds\\
&=\frac{1}{2}\int_0^t\int_{K}\sum\limits_{h=1}^{n}\bm{\alpha}_{jh}^{\dag}\theta_h
(m_{n,1}\sigma_1+m_{n,2}\sigma_2+m_{n,3}\sigma_3)\sum\limits_{h=1}^{n}\bm{\alpha}_{jh}\theta_hd\bm{x}ds\nonumber\\
&=\frac{1}{2}\int_0^t\int_{K}\sum\limits_{h=1}^{n}
\big(\bar{\bm{\alpha}}_{jh,-}\bm{\alpha}_{jh,+}+\bar{\bm{\alpha}}_{jh,+}\bm{\alpha}_{jh,-}\big)\theta_h^2m_{n,1}d\bm{x}ds\\
&\quad +\frac{i}{2}\int_0^t\int_{K}\sum\limits_{h=1}^{n}
\big(\bar{\bm{\alpha}}_{jh,-}\bm{\alpha}_{jh,+}-\bar{\bm{\alpha}}_{jh,+}\bm{\alpha}_{jh,-}\big)\theta_h^2m_{n,2}d\bm{x}ds\\
&\quad +\frac{1}{2}\int_0^t\int_{K}\sum\limits_{h=1}^{n}
\big(\bar{\bm{\alpha}}_{jh,+}\bm{\alpha}_{jh,+}-\bar{\bm{\alpha}}_{jh,-}\bm{\alpha}_{jh,-}\big)\theta_h^2m_{n,3}d\bm{x}ds.
\end{align*}
Using the fact that the imaginary part is 0, we get $|\bm{\psi}_{jn}(t)|_{\mathbb{L}^2}^2=|\bm{\psi}_{jn}(0)|_{\mathbb{L}^2}^2$, then
$\mathbb{E}\big[|\bm{\psi}_{jn}(t,\cdot)|_{\mathbb{L}^2}^p\big]=\mathbb{E}\big[|\bm{\psi}_{jn}(0,\cdot)|_{\mathbb{L}^2}^p\big]$.
\end{proof}
\begin{lemma}\label{PP}
Let $T\in(0,\infty)$, then for each $n=1,2,\cdots$, $t\in[0,T]$, it holds that:
\begin{align}\label{PP1}
({\rm{i}})~~
&\mathbb{E}\bigg[\int_{K}\sum\limits_{j=1}^{\infty}\lambda_j|\nabla\bm{\psi}_{jn}(t)|^2d\bm{x}\bigg]
+\mathbb{E}\bigg[\int_{K}|\nabla V_n(t)|^2d\bm{x}\bigg]
=\mathbb{E}\bigg[\int_0^t\int_{K}\partial_s\bm{s}_{n}\cdot\bm{m}_nd\bm{x}ds\bigg]\nonumber\\
&\quad +\mathbb{E}\bigg[\int_{K}\sum\limits_{j=1}^{\infty}\lambda_j|\nabla\bm{\psi}_{jn}(0)|^2d\bm{x}\bigg]
+\mathbb{E}\bigg[\int_{K}|\nabla V_n(0)|^2d\bm{x}\bigg],
\end{align}
\begin{align}\label{M01}
({\rm{ii}})~~
&2\alpha\mathbb{E}\bigg[\int_0^t\int_{D}|\partial_s\bm{m}_n(s)|^2d\bm{x}ds\bigg]
+\mathbb{E}\bigg[\int_{D}|\nabla\bm{m}_n(t)|^2d\bm{x}\bigg]
+2\mathbb{E}\bigg[\int_{D}w(\bm{m}_n(t))d\bm{x}\bigg]\nonumber\\
&\quad +\mathbb{E}\bigg[\int_{\mathbb{R}^3}|\bm{H}_{sn}(t)|^2d\bm{x}\bigg]
+\frac{k}{2}\mathbb{E}\bigg[\int_{D}(|\bm{m}_n(t)|^2-1)^2d\bm{x}\bigg]
+\mathbb{E}\bigg[\int_D|G_n(\bm{m}_n(t))|^2d\bm{x}\bigg]\nonumber\\
&=\mathbb{E}\bigg[\int_0^t\int_{D}\partial_s\bm{m}_n\cdot\bm{s}_nd\bm{x}ds\bigg]
+\mathbb{E}\bigg[\int_{D}|\nabla\bm{m}_n(0)|^2d\bm{x}\bigg]
+2\mathbb{E}\bigg[\int_{D}w(\bm{m}_n(0))d\bm{x}\bigg]\nonumber\\
&\quad +\frac{k}{2}\mathbb{E}\bigg[\int_{D}(|\bm{m}_n(0)|^2-1)^2d\bm{x}\bigg]
+\mathbb{E}\bigg[\int_{\mathbb{R}^3}|\bm{H}_{sn}(0)|^2d\bm{x}\bigg]
+\mathbb{E}\bigg[\int_D|G_n(\bm{m}_n(0))|^2d\bm{x}\bigg].
\end{align}
\end{lemma}
\begin{proof}
Multiplying the left-hand side of \eqref{S3} by $\sum_{j=1}^{\infty}\sum_{h=1}^{n}\lambda_j\frac{d\bm{\alpha}_{jh}^{\dag}}{dt}$, we can obtain
\begin{align}\label{PP2}
&\bigg(\sum\limits_{j=1}^{\infty}\sum\limits_{h=1}^{n}\lambda_j\frac{d\bm{\alpha}_{jh}^{\dag}}{dt}\theta_h,i\partial_t\bm{\psi}_{jn}
+\frac{1}{2}\Delta\bm{\psi}_{jn}-V_n\bm{\psi}_{jn}+\frac{1}{2}\bm{m}_n\cdot\hat{\bm{\sigma}}\bm{\psi}_{jn}\bigg)=0.
\end{align}
Integrating with respect to time $t$, we get
\begin{align*}
i\sum\limits_{j=1}^{\infty}\lambda_j\int_0^t\int_{K}\sum\limits_{h=1}^{n}\frac{d\bm{\alpha}_{jh}^{\dag}}{ds}\theta_h
\partial_s\bm{\psi}_{jn}d\bm{x}ds
=i\sum\limits_{j=1}^{\infty}\lambda_j\int_0^t\int_{K}\partial_s\bm{\psi}^{\dag}_{jn}\partial_s\bm{\psi}_{jn}d\bm{x}ds.
\end{align*}
By \eqref{P2}, we have $\partial_t\bm{\alpha}^{\dag}_{jh}\cdot\bm{\alpha}_{jh}=\frac{1}{2}\partial_t|\bm{\alpha}_{jh}|^2+{\rm Im}(\partial_t\bm{\alpha}^{\dag}_{jh}\cdot\bm{\alpha}_{jh})$, which implies that
\begin{align*}
&\frac{1}{2}\sum\limits_{j=1}^{\infty}\lambda_j
\int_0^t\int_{K}\sum\limits_{h=1}^{n}\frac{d\bm{\alpha}_{jh}^{\dag}}{ds}\theta_h\Delta\bm{\psi}_{jn}d\bm{x}ds
=\frac{1}{2}\sum\limits_{j=1}^{\infty}\lambda_j
\int_0^t\int_{K}\sum\limits_{h=1}^{n}\frac{d\bm{\alpha}_{jh}^{\dag}}{ds}\bm{\alpha}_{jh}\theta_h\Delta\theta_hd\bm{x}ds\\
&=-\frac{1}{4}\sum\limits_{j=1}^{\infty}\lambda_j\int_{K}|\nabla\bm{\psi}_{jn}(t)|^2d\bm{x}
+\frac{1}{4}\sum\limits_{j=1}^{\infty}\lambda_j\int_{K}|\nabla\bm{\psi}_{jn}(0)|^2d\bm{x}\\
&\quad -\frac{i}{2}\sum\limits_{j=1}^{\infty}\lambda_j\int_0^t\int_{K}\sum\limits_{h=1}^{n}{\rm Im}
(\partial_s\bm{\alpha}^{\dag}_{jh}\cdot\bm{\alpha}_{jh})|\nabla\theta_h|^2d\bm{x}ds,
\end{align*}
and
\begin{align*}
&-\sum\limits_{j=1}^{\infty}\lambda_j
\int_0^t\int_{K}\sum\limits_{h=1}^{n}\frac{d\bm{\alpha}_{jh}^{\dag}}{ds}\theta_hV_n\bm{\psi}_{jn}d\bm{x}ds
=-\sum\limits_{j=1}^{\infty}\lambda_j
\int_0^t\int_{K}\sum\limits_{h=1}^{n}\frac{d\bm{\alpha}_{jh}^{\dag}}{ds}\bm{\alpha}_{jh}\theta_h^2V_nd\bm{x}ds\\
&=-\frac{1}{2}\int_0^t\frac{d}{ds}\int_{K}V_n\sum\limits_{j=1}^{\infty}\lambda_j|\bm{\psi}_{jn}|^2d\bm{x}ds
+\frac{1}{2}\int_0^t\int_{K}\sum\limits_{j=1}^{\infty}\lambda_j|\bm{\psi}_{jn}|^2\partial_sV_nd\bm{x}ds\\
&\quad -i\sum\limits_{j=1}^{\infty}\lambda_j\int_0^t\int_{K}\sum\limits_{h=1}^{n}{\rm Im}
(\partial_s\bm{\alpha}^{\dag}_{jh}\cdot\bm{\alpha}_{jh})\theta_h^2V_nd\bm{x}ds\\
&=-\frac{1}{4}\int_{K}|\nabla V_n(t)|^2d\bm{x}
+\frac{1}{4}\int_{K}|\nabla V_n(0)|^2d\bm{x}
-i\sum\limits_{j=1}^{\infty}\lambda_j\int_0^t\int_{K}\sum\limits_{h=1}^{n}{\rm Im}
(\partial_s\bm{\alpha}^{\dag}_{jh}\cdot\bm{\alpha}_{jh})\theta_h^2V_nd\bm{x}ds.
\end{align*}
Furthermore, since
\begin{align*}
\bm{s}_n=&\sum_{j=1}^{\infty}\lambda_{j}{\rm{Tr}}_{{\mathbb{C}}^2}(\hat{\bm{\sigma}}{\bm{\psi}_{jn}}{\bm{\psi}^{\dag}_{jn}})
=\sum_{j=1}^{\infty}\lambda_{j}({\rm{Tr}}\sigma_1{\bm{\psi}_{jn}}{\bm{\psi}^{\dag}_{jn}},
{\rm{Tr}}\sigma_2{\bm{\psi}_{jn}}{\bm{\psi}^{\dag}_{jn}},{\rm{Tr}}\sigma_3{\bm{\psi}_{jn}}{\bm{\psi}^{\dag}_{jn}})\\
=&\sum_{j=1}^{\infty}\lambda_{j}\sum\limits_{h=1}^{n}
(\bm{\alpha}_{jh,-}\bar{\bm{\alpha}}_{jh,+}+\bm{\alpha}_{jh,+}\bar{\bm{\alpha}}_{jh,-},
i(\bm{\alpha}_{jh,+}\bar{\bm{\alpha}}_{jh,-}-\bm{\alpha}_{jh,-}\bar{\bm{\alpha}}_{jh,+}),
|\bm{\alpha}_{jh,+}|^2-|\bm{\alpha}_{jh,-}|^2)\theta_h^2,
\end{align*}
it is enough to show that
\begin{align*}
&\frac{1}{2}\sum\limits_{j=1}^{\infty}\lambda_j
\int_0^t\int_{K}\sum\limits_{h=1}^{n}\frac{d\bm{\alpha}_{jh}^{\dag}}{ds}\theta_h\bm{m}_n\cdot\hat{\bm{\sigma}}
\bm{\psi}_{jn}d\bm{x}ds\nonumber\\
&=\frac{1}{2}\sum\limits_{j=1}^{\infty}\lambda_j
\int_0^t\int_{K}\sum\limits_{h=1}^{n}\big((\partial_s\bar{\bm{\alpha}}_{jh,-}\bm{\alpha}_{jh,+}
+\partial_s\bar{\bm{\alpha}}_{jh,+}\bm{\alpha}_{jh,-})m_{n,1}\\
&\quad +(\partial_s\bar{\bm{\alpha}}_{jh,+}\bm{\alpha}_{jh,+}
-\partial_s\bar{\bm{\alpha}}_{jh,-}\bm{\alpha}_{jh,-})m_{n,3}
+i(\partial_s\bar{\bm{\alpha}}_{jh,-}\bm{\alpha}_{jh,+}
-\partial_s\bar{\bm{\alpha}}_{jh,+}\bm{\alpha}_{jh,-})m_{n,2}\big)\theta_h^2d\bm{x}ds\\
&=\frac{1}{4}\int_0^t\int_{K}\partial_s\bm{s}_{n}\cdot\bm{m}_nd\bm{x}ds
+\frac{i}{2}{\rm{Im}}\int_0^t\int_{K}\sum\limits_{j=1}^{\infty}\lambda_j
\partial_s\bm{\psi}^{\dag}_{jn}\bm{m}_n\cdot\hat{\bm{\sigma}}\bm{\psi}_{jn}d\bm{x}ds.
\end{align*}
Using the fact that the real part is 0, we obtain
\begin{align}\label{1PP1}
&\int_{K}\sum\limits_{j=1}^{\infty}\lambda_j|\nabla\bm{\psi}_{jn}(t)|^2d\bm{x}
+\int_{K}|\nabla V_n(t)|^2d\bm{x}\nonumber\\
&=\int_0^t\int_{K}\partial_s\bm{s}_{n}\cdot\bm{m}_nd\bm{x}ds
+\int_{K}\sum\limits_{j=1}^{\infty}\lambda_j|\nabla\bm{\psi}_{jn}(0)|^2d\bm{x}
+\int_{K}|\nabla V_n(0)|^2d\bm{x}.
\end{align}
Thus taking the expectation on the both sides of \eqref{1PP1}, we justify \eqref{PP1}.
Multiplying the left-hand side of \eqref{SLLG4} by $\sum_{h=1}^{n}\frac{d\bm{\beta}_{h}}{dt}$, we obtain
\begin{align}\label{M02}
\alpha(\partial_t\bm{m}_n,\partial_t\bm{m}_n)
&=-(\partial_t\bm{m}_n,\bm{m}_n\times\partial_t\bm{m}_n)+(\partial_t\bm{m}_n,\bm{H}_n)
-(\partial_t\bm{m}_n,k(|\bm{m}_n|^2-1)\bm{m}_n)\nonumber\\
&\quad -\bigg(\partial_t\bm{m}_n,G_n(\bm{m}_n)\circ\frac{dW(t)}{dt}\bigg)
-\bigg(\partial_t\bm{m}_n,F_n(\bm{m}_n,l)\frac{dL(t)}{dt}\bigg).
\end{align}
Integrating on $[0,t]$, we get
\begin{align*}
\alpha\int_0^t(\partial_s\bm{m}_n,\partial_s\bm{m}_n)ds
=\alpha\int_0^t|\partial_s\bm{m}_n|_{\mathbb{L}^2}^2ds.
\end{align*}
The identity $\langle a,a\times b\rangle=0$ leads to the equality
\begin{align*}
-\int_0^t(\partial_s\bm{m}_n,\bm{m}_n\times\partial_s\bm{m}_n)ds=0.
\end{align*}
Letting $\bm{u}_n=\nabla N\ast\cdot\bm{m}_n$, we can show that $\bm{u}_n$ satisfies the equation $\Delta \bm{u}=div(\bm{m}\chi_D)$ in the sense of distribution. Consequently, we have
\begin{align*}
-\int_{D}\partial_t\bm{m}_n\cdot\bm{H}_{sn}d\bm{x}
=\frac{1}{2}\frac{d}{dt}\int_{\mathbb{R}^3}|\bm{H}_{sn}|^2d\bm{x},
\end{align*}
therefore, we obtain
\begin{align*}
&\int_0^t(\partial_s\bm{m}_n,\bm{H}_n)ds
=\int_0^t\big(\partial_s\bm{m}_n,\Delta\bm{m}_n-w'(\bm{m}_n)+\bm{H}_{sn}+\frac{1}{2}\bm{s}_n\big)ds\\
&=-\int_0^t\int_{D}\partial_s(\nabla\bm{m}_n)\cdot\nabla\bm{m}_nd\bm{x}ds
-\int_0^t\frac{d}{ds}\int_{D}w(\bm{m}_n)d\bm{x}ds\\
&\quad -\frac{1}{2}\int_0^t\frac{d}{ds}\int_{\mathbb{R}^3}|\bm{H}_{sn}|^2d\bm{x}ds
+\frac{1}{2}\int_0^t\int_{D}\partial_s\bm{m}_n\cdot\bm{s}_nd\bm{x}ds\\
&=-\frac{1}{2}\int_{D}|\nabla\bm{m}_n(t)|^2d\bm{x}
-\int_{D}w(\bm{m}_n(t))d\bm{x}
-\frac{1}{2}\int_{\mathbb{R}^3}|\bm{H}_{sn}(t)|^2d\bm{x}\\
&\quad +\frac{1}{2}\int_0^t\int_{D}\partial_s\bm{m}_n\cdot\bm{s}_nd\bm{x}ds
+\frac{1}{2}\int_{D}|\nabla\bm{m}_n(0)|^2d\bm{x}\\
&\quad +\int_{D}w(\bm{m}_n(0))d\bm{x}
+\frac{1}{2}\int_{\mathbb{R}^3}|\bm{H}_{sn}(0)|^2d\bm{x}.
\end{align*}
According to the Newton-Leibniz formula, we infer that
\begin{align*}
&\int_0^t(\partial_s\bm{m}_n,k(|\bm{m}_n|^2-1)\bm{m}_n)ds
=\frac{k}{4}\int_0^t\frac{d}{ds}\int_{D}(|\bm{m}_n|^2-1)^2d\bm{x}ds\\
&=\frac{k}{4}\int_{D}(|\bm{m}_n(t)|^2-1)^2d\bm{x}-\frac{k}{4}\int_{D}(|\bm{m}_n(0)|^2-1)^2d\bm{x}.
\end{align*}
Using the definition of the Stratonovitch integral, we deduce that
\begin{align*}
&\int_0^t\bigg(\partial_s\bm{m}_n,G_n(\bm{m}_n)\circ\frac{dW(s)}{ds}\bigg)ds\\
&=\frac{1}{2}\int_0^t\int_D\partial_s\bm{m}_n\cdot G'_n(\bm{m}_n)[G_n(\bm{m}_n)]d\bm{x}ds
+\int_0^t(\partial_s\bm{m}_n,G_n(\bm{m}_n))dW(s)\\
&=\frac{1}{2}\int_0^t\int_D\partial_sG_n(\bm{m}_n)[G_n(\bm{m}_n)]d\bm{x}ds
+\int_0^t(\partial_s\bm{m}_n,G_n(\bm{m}_n))dW(s)\\
&=\frac{1}{2}\int_D|G_n(\bm{m}_n(t))|^2d\bm{x}
-\frac{1}{2}\int_D|G_n(\bm{m}_n(0))|^2d\bm{x}
+\int_0^t(\partial_s\bm{m}_n,G_n(\bm{m}_n))dW(s).
\end{align*}
Thanks to
\begin{align*}
\int_0^t\bigg(\partial_s\bm{m}_n,F_n(\bm{m}_n,l)\frac{dL(s)}{ds}\bigg)ds
=\int_0^t\int_B(\partial_s\bm{m}_n,F_n(\bm{m}_n,l))\tilde{\eta}(ds,dl),
\end{align*}
it holds that
\begin{align}\label{1M01}
&2\alpha\int_0^t\int_{D}|\partial_s\bm{m}_n|^2d\bm{x}ds
+\int_{D}|\nabla\bm{m}_n(t)|^2d\bm{x}
+2\int_{D}w(\bm{m}_n(t))d\bm{x}\nonumber\\
&\quad +\int_{\mathbb{R}^3}|\bm{H}_{sn}(t)|^2d\bm{x}
+\frac{k}{2}\int_{D}(|\bm{m}_n(t)|^2-1)^2d\bm{x}
+\int_D|G_n(\bm{m}_n(t))|^2d\bm{x}\nonumber\\
&\quad +2\int_0^t(\partial_s\bm{m}_n,G_n(\bm{m}_n))dW(s)
+2\int_0^t\int_B(\partial_s\bm{m}_n,F_n(\bm{m}_n,l))\tilde{\eta}(ds,dl)\nonumber\\
&=\int_0^t\int_{D}\partial_s\bm{m}_n\cdot\bm{s}_nd\bm{x}ds
+\int_{D}|\nabla\bm{m}_n(0)|^2d\bm{x}
+2\int_{D}w(\bm{m}_n(0))d\bm{x}\nonumber\\
&\quad +\frac{k}{2}\int_{D}(|\bm{m}_n(0)|^2-1)^2d\bm{x}
+\int_{\mathbb{R}^3}|\bm{H}_{sn}(0)|^2d\bm{x}
+\int_D|G_n(\bm{m}_n(0))|^2d\bm{x}.
\end{align}
Hence taking the expectation on the both sides of \eqref{1M01} yields \eqref{M01}.
\end{proof}
\begin{lemma}\label{SPMpci}
Given a Wiener process $W$ and a Poisson process $\eta$, and assume that $(\bm{\Psi}_n, \bm{m}_n)$ is the solution of \eqref{S3} and \eqref{SLLG4} with $T\in(0,\infty)$. For $1\leq r<\infty$ and every $n=1,2,\cdots$, there exists a constant $C>0$ such that the following inequalities hold:
{\small\begin{align}\label{MMpci2}
({\rm{i}})~&\frac{\alpha^{2r}}{4^{r}}\mathbb{E}\big[|\bm{m}_n(t)|_{\mathbb{L}^2}^{4r}\big]
+\mathbb{E}\bigg[\bigg(\int_0^t\int_{D}|\nabla\bm{m}_n|^2d\bm{x}ds\bigg)^{2r}\bigg]
+\mathbb{E}\bigg[\bigg(\int_0^t\int_{\mathbb{R}^3}|\bm{H}_{sn}|^2d\bm{x}ds\bigg)^{2r}\bigg]\nonumber\\
&+\mathbb{E}\bigg[\bigg(k\int_0^t\int_{D}|\bm{m}_n|^4d\bm{x}ds\bigg)^{2r}\bigg]
+\mathbb{E}\bigg[\bigg(\frac{\alpha}{2}\sup_{0\leq t\leq T}|\bm{m}_n(t)|_{\mathbb{L}^2}^2\bigg)^{2r}\bigg]
+\mathbb{E}\bigg[\bigg(k\sup_{0\leq t\leq T}\int_0^t\int_{D}|\bm{m}_n|^4d\bm{x}ds\bigg)^{2r}\bigg]\nonumber\\
&+\mathbb{E}\bigg[\bigg(\sup_{0\leq t\leq T}\int_0^t\int_{D}|\nabla\bm{m}_n|^2d\bm{x}ds\bigg)^{2r}\bigg]
+\mathbb{E}\bigg[\bigg(\sup_{0\leq t\leq T}\int_0^t\int_{\mathbb{R}^3}|\bm{H}_{sn}|^2d\bm{x}ds\bigg)^{2r}\bigg]\nonumber\\
&\leq C\mathbb{E}\big[|\bm{m}_n(0)|_{\mathbb{L}^2}^{4r}\big]
+C\mathbb{E}\bigg[\int_0^T|\bm{m}_n|_{\mathbb{H}^1}^{4r}dt\bigg]
+C\mathbb{E}\bigg[\int_0^T\bigg(\sum\limits_{j=1}^{\infty}\lambda_j
|\nabla\bm{\psi}_{jn}|_{\mathbb{L}^{2}}^2\bigg)^{2r}dt\bigg]+C,
\end{align}}
{\small
\begin{align}\label{PMpci2}
({\rm{ii}})~&\mathbb{E}\bigg[\bigg(\int_{K}\sum\limits_{j=1}^{\infty}\lambda_j
|\nabla\bm{\psi}_{jn}|^2d\bm{x}\bigg)^{2r}\bigg]
+\mathbb{E}\bigg[\bigg(\frac{k}{2}\int_{D}(|\bm{m}_n|^2-1)^2d\bm{x}\bigg)^{2r}\bigg]
+\mathbb{E}\bigg[\bigg(\int_{0}^{t}\int_{D}|\partial_s\bm{m}_n|^2d\bm{x}ds\bigg)^{2r}\bigg]\nonumber\\
&+\mathbb{E}\big[|\bm{m}_n|_{\mathbb{H}^1}^{4r}\big]
+\mathbb{E}\bigg[\bigg(\int_{\mathbb{R}^3}|\bm{H}_{sn}|^2d\bm{x}\bigg)^{2r}\bigg]
+\mathbb{E}\bigg[\bigg(\int_{D}w(\bm{m}_n)d\bm{x}\bigg)^{2r}\bigg]
+\mathbb{E}\bigg[\bigg(\int_{D}|G_n(\bm{m}_{n})|^2d\bm{x}\bigg)^{2r}\bigg]\nonumber\\
&+\mathbb{E}\bigg[\bigg(\int_{K}|\nabla V_{n}|^2d\bm{x}\bigg)^{2r}\bigg]
+\mathbb{E}\bigg[\bigg(\sup_{0\leq t\leq T}
\int_{K}\sum\limits_{j=1}^{\infty}\lambda_j|\nabla\bm{\psi}_{jn}|^2d\bm{x}\bigg)^{2r}\bigg]
+\mathbb{E}\bigg[\bigg(\sup_{0\leq t\leq T}|\bm{m}_n|_{\mathbb{H}^1}^2\bigg)^{2r}\bigg]\nonumber\\
&+\mathbb{E}\bigg[\bigg(\sup_{0\leq t\leq T}\int_{K}|\nabla V_{n}|^2d\bm{x}\bigg)^{2r}\bigg]
+\mathbb{E}\bigg[\bigg(\sup_{0\leq t\leq T}\int_{D}w(\bm{m}_n)d\bm{x}\bigg)^{2r}\bigg]
+\mathbb{E}\bigg[\bigg(\sup_{0\leq t\leq T}\int_{\mathbb{R}^3}|\bm{H}_{sn}|^2d\bm{x}\bigg)^{2r}\bigg]\nonumber\\
&+\mathbb{E}\bigg[\bigg(\frac{k}{2}\sup_{0\leq t\leq T}\int_{D}(|\bm{m}_n|^2-1)^2d\bm{x}\bigg)^{2r}\bigg]
+\mathbb{E}\bigg[\bigg(\sup_{0\leq t\leq T}\int_{D}|G_n(\bm{m}_{n})|^2d\bm{x}\bigg)^{2r}\bigg]
\leq C, ~0\leq t\leq T.
\end{align}}
\begin{align*}
({\rm{iii}})~&\mathbb{E}\bigg[\sup_{0\leq t\leq T}|\rho_n|_{\mathbb{L}^s(K)}^{2r}\bigg]\leq C,~
\mathbb{E}\bigg[\sup_{0\leq t\leq T}|\bm{s}_n|_{\mathbb{L}^s(K)}^{2r}\bigg]\leq C,~1\leq s\leq3,\\
&\mathbb{E}\bigg[\sup_{0\leq t\leq T}|\partial_t\bm{\psi}_{jn}|_{\mathbb{H}^{-1}(K)}\bigg]\leq C,~
\mathbb{E}\bigg[\sup_{0\leq t\leq T}\bigg(\int_{D}|w'(\bm{m}_n)|^sd\bm{x}\bigg)^{2r}\bigg]\leq C,~1\leq s\leq2.
\end{align*}
\end{lemma}
\begin{proof}
Multiplying the left-hand side of \eqref{SLLG4} by $\sum_{h=1}^{n}\bm{\beta}_{h}$, we arrive at the expression
\begin{align}\label{MM03}
\alpha(\bm{m}_n,\partial_t\bm{m}_n)
&=-(\bm{m}_n,\bm{m}_n\times\partial_t\bm{m}_n)+(\bm{m}_n,\bm{H}_n)
-(\bm{m}_n,k(|\bm{m}_n|^2-1)\bm{m}_n)\nonumber\\
&\quad -\bigg(\bm{m}_n,G_n(\bm{m}_n)\circ\frac{dW}{dt}\bigg)
-\bigg(\bm{m}_n,F_n(\bm{m}_n,l)\frac{dL}{dt}\bigg).
\end{align}
Letting $\bm{u}_n=\nabla N\ast\cdot\bm{m}_n$, we find that $\bm{u}_n$ satisfies the equation $\Delta \bm{u}=div(\bm{m}\chi_D)$ in the
sense of distribution. From this we infer that
\begin{align*}
-\int_{D}\bm{m}_n\cdot\bm{H}_{sn}d\bm{x}
=\int_{\mathbb{R}^3}|\bm{H}_{sn}|^2d\bm{x}.
\end{align*}
Integrating \eqref{MM03} with respect to time $t$, we get
\begin{align}\label{MM030}
&\frac{\alpha}{2}|\bm{m}_n(t)|_{\mathbb{L}^2}^2+\int_0^t\int_{D}|\nabla\bm{m}_n|^2d\bm{x}ds
+\int_0^t\int_{\mathbb{R}^3}|\bm{H}_{sn}|^2d\bm{x}ds
+k\int_0^t\int_{D}(|\bm{m}_n|^2-1)|\bm{m}_n|^2d\bm{x}ds\nonumber\\
&=\frac{\alpha}{2}|\bm{m}_n(0)|_{\mathbb{L}^2}^2
-\int_0^t\int_{D}\bm{m}_n\cdot w'(\bm{m}_n)d\bm{x}ds
-\frac{1}{2}\int_0^t\int_D\bm{m}_n\cdot G'_n(\bm{m}_n)[G_n(\bm{m}_n)]d\bm{x}ds\\
&\quad +\frac{1}{2}\int_0^t\int_{D}\bm{m}_n\cdot\bm{s}_nd\bm{x}ds
-\int_0^t\int_D\bm{m}_n\cdot G_n(\bm{m}_n)d\bm{x}dW(s)
-\int_0^t\int_B(\bm{m}_n,F_n(\bm{m}_n,l))\tilde{\eta}(ds,dl).\nonumber
\end{align}
Notice that
\begin{align}\label{MM13}
|\bm{s}_n|
\leq&\sum_{j=1}^{\infty}\lambda_{j}\big|{\rm{Tr}}_{{\mathbb{C}}^2}\big(\hat{\bm{\sigma}}{\bm{\psi}_{jn}}
{\bm{\psi}^{\dag}_{jn}}\big)\big|
=\sum\limits_{j=1}^{\infty}\lambda_j(|\bm{\psi}_{jn,+}|^2+|\bm{\psi}_{jn,-}|^2)
=\sum\limits_{j=1}^{\infty}\lambda_j|\bm{\psi}_{jn}|^2,
\end{align}
by the Gagliardo-Nirenberg-Sobolev (GNS) inequality, the Young inequality with $\varepsilon$, and Lemma \ref{P}, we get
\begin{align}\label{MMpci3}
&\mathbb{E}\bigg[\bigg(\int_0^t\int_{D}|\bm{m}_n\cdot\bm{s}_n|d\bm{x}ds\bigg)^{2r}\bigg]
\leq\mathbb{E}\bigg[\bigg(\int_0^tC|\bm{m}_n|_{\mathbb{H}^1(D)}|\bm{s}_n|_{\mathbb{L}^{1}(\mathbb{R}^3)}^{\frac{3}{4}}
|\bm{s}_n|_{\mathbb{L}^{3}(\mathbb{R}^3)}^{\frac{1}{4}}ds\bigg)^{2r}\bigg]\nonumber\\
&\leq\mathbb{E}\bigg[\bigg(\int_0^t\varepsilon|\bm{m}_n|_{\mathbb{H}^1(D)}^2
+\varepsilon|\bm{s}_n|_{\mathbb{L}^{3}(\mathbb{R}^3)}
+C|\bm{s}_n|_{\mathbb{L}^{1}(\mathbb{R}^3)}^3ds\bigg)^{2r}\bigg]\nonumber\\
&\leq3^{2r-1}\varepsilon^{2r}\mathbb{E}\bigg[\bigg(\int_0^t|\bm{m}_n|_{\mathbb{H}^1}^2ds\bigg)^{2r}\bigg]
+3^{2r-1}\varepsilon^{2r}\mathbb{E}\bigg[\bigg(\int_0^t\sum\limits_{j=1}^{\infty}\lambda_j|\nabla\bm{\psi}_{jn}|
_{\mathbb{L}^{2}}^2ds\bigg)^{2r}\bigg]\nonumber\\
&\quad +C\mathbb{E}\bigg[\bigg(\int_0^t\bigg(\sum\limits_{j=1}^{\infty}\lambda_j
|\bm{\psi}_{jn}|_{\mathbb{L}^{2}}^2\bigg)^3ds\bigg)^{2r}\bigg]\nonumber\\
&\leq3^{2r-1}\varepsilon^{2r}\mathbb{E}\bigg[\int_0^t|\bm{m}_n|_{\mathbb{H}^1}^{4r}ds\bigg]
+3^{2r-1}\varepsilon^{2r}\mathbb{E}\bigg[\int_0^t\bigg(\sum\limits_{j=1}^{\infty}\lambda_j|\nabla\bm{\psi}_{jn}|
_{\mathbb{L}^{2}}^2\bigg)^{2r}ds\bigg]+C.
\end{align}
Moreover, applying the Burkholder-Davis-Gundy (BDG) inequality, the Young inequality, the H\"{o}lder inequality, and Assumption \ref{ass1}, we have
{\small\begin{align}\label{PMpci331}
&\mathbb{E}\bigg[\bigg|\int_0^t\int_D\bm{m}_n\cdot G_n(\bm{m}_n)d\bm{x}dW(s)\bigg|^{2r}\bigg]
\leq\mathbb{E}\bigg[\bigg(\sup_{0\leq t\leq T}\bigg|\int_0^t\int_D
\bm{m}_n\cdot G_n(\bm{m}_n)d\bm{x}dW(s)\bigg|\bigg)^{2r}\bigg]\nonumber\\
&=\mathbb{E}\bigg[\bigg(\sup_{0\leq t\leq T}\bigg|\int_{0}^{t}\sum\limits_{i\geq1}\int_D\bm{m}_n\cdot G_n(\bm{m}_n)d\bm{x}\tilde{e}_idW_i(s)\bigg|\bigg)^{2r}\bigg]
\leq C\mathbb{E}\bigg[\bigg|\int_0^T\bigg|\sum\limits_{i\geq1}\int_D\bm{m}_n\cdot G_n(\bm{m}_n)d\bm{x}\tilde{e}_i\bigg|^2dt\bigg|^r\bigg]\nonumber\\
&\leq C\mathbb{E}\bigg[\bigg(\int_0^T\big|\bm{m}_n\big|_{\mathbb{L}^2}^2
\bigg(\sum\limits_{i\geq1}\big|G_{in}(\bm{m}_n)\big|_{\mathbb{L}^2}\bigg)^2dt\bigg)^{r}\bigg]
\leq C\mathbb{E}\bigg[\int_0^T|\bm{m}_n|_{\mathbb{L}^2}^{4r}dt\bigg]+C,
\end{align}}
and
\begin{align}\label{PMpci3331}
\mathbb{E}\bigg[\bigg|\int_0^t\int_B\int_{D}\bm{m}_n\cdot F_n(\bm{m}_n,l)d\bm{x}\tilde{\eta}(ds,dl)\bigg|^{2r}\bigg]
\leq C\mathbb{E}\bigg[\int_0^T|\bm{m}_n|_{\mathbb{L}^2}^{4r}dt\bigg]+C.
\end{align}
Choosing $\varepsilon=\frac{1}{3}$ in \eqref{MMpci3}, using the properties of the martingale, we obtain that
{\small\begin{align*}
&\frac{\alpha^{2r}}{4^{r}}\mathbb{E}\big[|\bm{m}_n(t)|_{\mathbb{L}^2}^{4r}\big]
+\mathbb{E}\bigg[\bigg(\int_0^t\int_{D}|\nabla\bm{m}_n|^2d\bm{x}ds\bigg)^{2r}\bigg]
+\mathbb{E}\bigg[\bigg(\int_0^t\int_{\mathbb{R}^3}|\bm{H}_{sn}|^2d\bm{x}ds\bigg)^{2r}\bigg]\nonumber\\
&\quad +\mathbb{E}\bigg[\bigg(k\int_0^t\int_{D}|\bm{m}_n|^4d\bm{x}ds\bigg)^{2r}\bigg]
+\mathbb{E}\bigg[\bigg(\frac{\alpha}{2}\sup_{0\leq t\leq T}|\bm{m}_n(t)|_{\mathbb{L}^2}^2\bigg)^{2r}\bigg]
+\mathbb{E}\bigg[\bigg(k\sup_{0\leq t\leq T}\int_0^t\int_{D}|\bm{m}_n|^4d\bm{x}ds\bigg)^{2r}\bigg]\nonumber\\
&\quad +\mathbb{E}\bigg[\bigg(\sup_{0\leq t\leq T}\int_0^t\int_{D}|\nabla\bm{m}_n|^2d\bm{x}ds\bigg)^{2r}\bigg]
+\mathbb{E}\bigg[\bigg(\sup_{0\leq t\leq T}\int_0^t\int_{\mathbb{R}^3}|\bm{H}_{sn}|^2d\bm{x}ds\bigg)^{2r}\bigg]\\
&\leq\mathbb{E}\bigg[\bigg(\frac{\alpha}{2}|\bm{m}_n(0)|_{\mathbb{L}^2}^2
+\int_0^t\int_{D}|\bm{m}_n\cdot w'(\bm{m}_n)|d\bm{x}ds
+\frac{1}{2}\int_0^t\int_{D}|\bm{m}_n\cdot\bm{s}_n|d\bm{x}ds\\
&\quad +\frac{1}{2}\int_0^t\int_D|\bm{m}_n\cdot G'_n(\bm{m}_n)[G_n(\bm{m}_n)]|d\bm{x}ds
+\bigg|\int_0^t\int_D\bm{m}_n\cdot G_n(\bm{m}_n)d\bm{x}dW(s)\bigg|\\
&\quad +\bigg|\int_0^t\int_B(\bm{m}_n,F_n(\bm{m}_n,l))\tilde{\eta}(ds,dl)\bigg|
+k\int_0^t\int_{D}|\bm{m}_n|^2d\bm{x}ds\bigg)^{2r}\bigg]\\
&\quad +\mathbb{E}\bigg[\bigg(2\alpha|\bm{m}_n(0)|_{\mathbb{L}^2}^2
+4\sup_{0\leq t\leq T}\int_0^t\int_{D}|\bm{m}_n\cdot w'(\bm{m}_n)|d\bm{x}ds
+2\sup_{0\leq t\leq T}\int_0^t\int_{D}|\bm{m}_n\cdot\bm{s}_n|d\bm{x}ds\\
&\quad +4\sup_{0\leq t\leq T}\bigg|\int_0^t\int_D\bm{m}_n\cdot G_n(\bm{m}_n)d\bm{x}dW(s)\bigg|
+4\sup_{0\leq t\leq T}\bigg|\int_0^t\int_B(\bm{m}_n,F_n(\bm{m}_n,l))\tilde{\eta}(ds,dl)\bigg|\\
&\quad +2\sup_{0\leq t\leq T}\int_0^t\int_D|\bm{m}_n\cdot G'_n(\bm{m}_n)[G_n(\bm{m}_n)]|d\bm{x}ds
+4k\sup_{0\leq t\leq T}\int_0^t\int_{D}|\bm{m}_n|^2d\bm{x}ds\bigg)^{2r}\bigg]\\
&\leq C\mathbb{E}\big[|\bm{m}_n(0)|_{\mathbb{L}^2}^{4r}\big]
+C\mathbb{E}\bigg[\int_0^T|\bm{m}_n|_{\mathbb{H}^1}^{4r}dt\bigg]
+C\mathbb{E}\bigg[\int_0^T\bigg(\sum\limits_{j=1}^{\infty}\lambda_j
|\nabla\bm{\psi}_{jn}|_{\mathbb{L}^{2}}\bigg)^{2r}dt\bigg]+C,
\end{align*}}
which leads to  \eqref{MMpci2}.
 Combining \eqref{1PP1} and \eqref{1M01} implies that
\begin{align}\label{PM3}
&\int_{K}\sum\limits_{j=1}^{\infty}\lambda_j|\nabla\bm{\psi}_{jn}|^2d\bm{x}
+\int_{K}|\nabla V_{n}|^2d\bm{x}+2\alpha\int_0^t\int_{D}|\partial_s\bm{m}_n|^2d\bm{x}ds+|\bm{m}_n|_{\mathbb{H}^1}^2\nonumber\\
&\quad +\int_{\mathbb{R}^3}|\bm{H}_{sn}|^2d\bm{x}
+2\int_{D}w(\bm{m}_n)d\bm{x}+\frac{k}{2}\int_{D}(|\bm{m}_n|^2-1)^2d\bm{x}+\int_{D}|G_n(\bm{m}_{n})|^2d\bm{x}\nonumber\\
&\quad +2\int_0^t\int_D\partial_s\bm{m}_n\cdot G_n(\bm{m}_n)d\bm{x}dW(s)
+2\int_0^t\int_B(\partial_s\bm{m}_n,F_n(\bm{m}_n,l))\tilde{\eta}(ds,dl)\nonumber\\
&=\int_{D}\bm{m}_n\cdot\bm{s}_nd\bm{x}+\int_{D}|\bm{m}_n|^2d\bm{x}+I_n,
\end{align}
where
\begin{align*}
I_n=&\int_{K}\sum\limits_{j=1}^{\infty}\lambda_j|\nabla\bm{\psi}_{jn}(\bm{x},0)|^2d\bm{x}
+\int_{K}|\nabla V_{n}(\bm{x},0)|^2d\bm{x}
+\int_{D}|\nabla\bm{m}_n(\bm{x},0)|^2d\bm{x}
+\int_{\mathbb{R}^3}|\bm{H}_{sn}(\bm{x},0)|^2d\bm{x}\\
&+2\int_{D}w(\bm{m}_n(\bm{x},0))d\bm{x}
+\frac{k}{2}\int_{D}(|\bm{m}_n(\bm{x},0)|^2-1)^2d\bm{x}
+\int_{D}|G_n(\bm{m}_{n}(\bm{x},0))|^2d\bm{x}\\
&-\int_{D}\bm{m}_n(\bm{x},0)\cdot\bm{s}_n(\bm{x},0)d\bm{x}.
\end{align*}
By the GNS inequality, the Young inequality with $\varepsilon$, and Lemma \ref{P}, we have
\begin{align}\label{PMpci3}
&\mathbb{E}\bigg[\bigg(\int_{D}|\bm{m}_n\cdot\bm{s}_n|d\bm{x}\bigg)^{2r}\bigg]
\leq3^{2r-1}\varepsilon^{2r}\mathbb{E}\bigg[|\bm{m}_n|_{\mathbb{H}^1}^{4r}
+\bigg(\sum\limits_{j=1}^{\infty}\lambda_j|\nabla\bm{\psi}_{jn}|_{\mathbb{L}^{2}}^2\bigg)^{2r}\bigg]+C.
\end{align}
Employing the BDG inequality, the Young inequality with $\varepsilon$, the H\"{o}lder inequality, and Assumption \ref{ass1}, we can derive that
\begin{align}\label{PMpci33}
&\mathbb{E}\bigg[\bigg|\int_0^t\int_D\partial_s\bm{m}_n\cdot G_n(\bm{m}_n)d\bm{x}dW(s)\bigg|^{2r}\bigg]
\leq\mathbb{E}\bigg[\bigg(\sup_{0\leq t\leq T}\bigg|\int_{0}^{t}\sum\limits_{i\geq1}\int_D\partial_s\bm{m}_n\cdot G_n(\bm{m}_n)d\bm{x}\tilde{e}_idW_i(s)\bigg|\bigg)^{2r}\bigg]\nonumber\\
&\leq C\mathbb{E}\bigg[\bigg|\int_0^T\bigg|\sum\limits_{i\geq1}\int_D\partial_s\bm{m}_n\cdot G_n(\bm{m}_n)d\bm{x}\tilde{e}_i\bigg|^2dt\bigg|^r\bigg]
\leq C\mathbb{E}\bigg[\bigg(\int_0^T\big|\partial_t\bm{m}_n\big|_{\mathbb{L}^2}^2
\bigg(\sum\limits_{i\geq1}\big|G_{in}(\bm{m}_n)\big|_{\mathbb{L}^2}\bigg)^2dt\bigg)^{r}\bigg]\nonumber\\
&\leq \varepsilon_1\mathbb{E}\bigg[\bigg(\int_0^T|\partial_t\bm{m}_n|_{\mathbb{L}^2}^2dt\bigg)^{2r}\bigg]
+C\mathbb{E}\bigg[\bigg(\sup_{0\leq t\leq T}|\bm{m}_n|_{\mathbb{L}^2}^2\bigg)^{2r}\bigg]+C,
\end{align}
and
\begin{align}\label{PMpci333}
&\mathbb{E}\bigg[\bigg|\int_0^t\int_B\int_{D}\partial_s\bm{m}_n\cdot F_n(\bm{m}_n,l)d\bm{x}\tilde{\eta}(ds,dl)\bigg|^{2r}\bigg]\nonumber\\
&\leq\varepsilon_1\mathbb{E}\bigg[\bigg(\int_0^T|\partial_t\bm{m}_n|_{\mathbb{L}^2}^2dt\bigg)^{2r}\bigg]
+C\mathbb{E}\bigg[\bigg(\sup_{0\leq t\leq T}|\bm{m}_n|_{\mathbb{L}^2}^2\bigg)^{2r}\bigg]+C.
\end{align}
Taking $\varepsilon^{2r}=2^{-1}\cdot15^{1-2r}$ in \eqref{PMpci3}, and using \eqref{PM3}, we observe that for any $t\in[0,T]$,
{\small\begin{align*}
&\mathbb{E}\bigg[\bigg(\int_{K}\sum\limits_{j=1}^{\infty}\lambda_j|\nabla\bm{\psi}_{jn}|^2d\bm{x}\bigg)^{2r}\bigg]
+\mathbb{E}\bigg[\bigg(\int_{K}|\nabla V_{n}|^2d\bm{x}\bigg)^{2r}\bigg]
+\mathbb{E}\bigg[\bigg(2\alpha\int_{0}^{t}\int_{D}|\partial_s\bm{m}_n|^2d\bm{x}ds\bigg)^{2r}\bigg]\\
&\quad +\mathbb{E}\big[|\bm{m}_n|_{\mathbb{H}^1}^{4r}\big]
+\mathbb{E}\bigg[\bigg(\int_{\mathbb{R}^3}|\bm{H}_{sn}|^2d\bm{x}\bigg)^{2r}\bigg]
+\mathbb{E}\bigg[\bigg(2\int_{D}w(\bm{m}_n)d\bm{x}\bigg)^{2r}\bigg]
+\mathbb{E}\bigg[\bigg(\frac{k}{2}\int_{D}(|\bm{m}_n|^2-1)^2d\bm{x}\bigg)^{2r}\bigg]\\
&\quad +\mathbb{E}\bigg[\bigg(\int_{D}|G_n(\bm{m}_{n})|^2d\bm{x}\bigg)^{2r}\bigg]
+\mathbb{E}\bigg[\bigg(\sup_{0\leq t\leq T}
\int_{K}\sum\limits_{j=1}^{\infty}\lambda_j|\nabla\bm{\psi}_{jn}|^2d\bm{x}\bigg)^{2r}\bigg]
+\mathbb{E}\bigg[\bigg(\sup_{0\leq t\leq T}\int_{K}|\nabla V_{n}|^2d\bm{x}\bigg)^{2r}\bigg]\\
&\quad +\mathbb{E}\bigg[\bigg(\sup_{0\leq t\leq T}|\bm{m}_n|_{\mathbb{H}^1}^2\bigg)^{2r}\bigg]
+\mathbb{E}\bigg[\bigg(2\alpha\int_0^T\int_{D}|\partial_t\bm{m}_n|^2d\bm{x}dt\bigg)^{2r}\bigg]
+\mathbb{E}\bigg[\bigg(\frac{k}{2}\sup_{0\leq t\leq T}\int_{D}(|\bm{m}_n|^2-1)^2d\bm{x}\bigg)^{2r}\bigg]\\
&\quad +\mathbb{E}\bigg[\bigg(2\sup_{0\leq t\leq T}\int_{D}w(\bm{m}_n)d\bm{x}\bigg)^{2r}\bigg]
+\mathbb{E}\bigg[\bigg(\sup_{0\leq t\leq T}\int_{\mathbb{R}^3}|\bm{H}_{sn}|^2d\bm{x}\bigg)^{2r}\bigg]
+\mathbb{E}\bigg[\bigg(\sup_{0\leq t\leq T}\int_{D}|G_n(\bm{m}_{n})|^2d\bm{x}\bigg)^{2r}\bigg]\\
&\leq\mathbb{E}\bigg[\bigg(I_n+\int_{D}|\bm{m}_n\cdot\bm{s}_n|d\bm{x}
+\int_{D}|\bm{m}_n|^2d\bm{x}
+2\bigg|\int_0^t\int_B(\partial_s\bm{m}_n,F_n(\bm{m}_n,l))\tilde{\eta}(ds,dl)\bigg|\\
&\quad +2\bigg|\int_0^t\int_D\partial_s\bm{m}_n\cdot G_n(\bm{m}_n)d\bm{x}dW(s)\bigg|\bigg)^{2r}\bigg]
+C\mathbb{E}\bigg[\bigg(I_n+\sup_{0\leq t\leq T}\bigg|\int_0^t\int_D\partial_s\bm{m}_n\cdot G_n(\bm{m}_n)d\bm{x}dW(s)\bigg|\\
&\quad +\sup_{0\leq t\leq T}\bigg|\int_0^t\int_B(\partial_s\bm{m}_n,F_n(\bm{m}_n,l))\tilde{\eta}(ds,dl)\bigg|
+\sup_{0\leq t\leq T}\int_{D}|\bm{m}_n\cdot\bm{s}_n|d\bm{x}
+\sup_{0\leq t\leq T}\int_{D}|\bm{m}_n|^2d\bm{x}\bigg)^{2r}\bigg]\\
&\leq\frac{1}{2}\mathbb{E}\bigg[\bigg(\int_{K}\sum\limits_{j=1}^{\infty}\lambda_j|\nabla\bm{\psi}_{jn}|^2d\bm{x}\bigg)^{2r}\bigg]
+\frac{1}{2}\mathbb{E}\big[|\bm{m}_n|_{\mathbb{H}^1}^{4r}\big]
+C\mathbb{E}\bigg[\bigg(\int_{D}|\bm{m}_n|^2d\bm{x}\bigg)^{2r}\bigg]
+\frac{1}{2}\mathbb{E}\bigg[\bigg(\sup_{0\leq t\leq T}|\bm{m}_n|_{\mathbb{H}^1}^2\bigg)^{2r}\bigg]\\
&\quad +\frac{1}{2}\mathbb{E}\bigg[\bigg(\sup_{0\leq t\leq T}\sum\limits_{j=1}^{\infty}\lambda_j
|\nabla\bm{\psi}_{jn}|_{\mathbb{L}^{2}}^2\bigg)^{2r}\bigg]
+4\varepsilon_1\mathbb{E}\bigg[\bigg(\int_0^T|\partial_t\bm{m}_n|_{\mathbb{L}^2}^2dt\bigg)^{2r}\bigg]
+4C\mathbb{E}\bigg[\bigg(\sup_{0\leq t\leq T}|\bm{m}_n|_{\mathbb{L}^2}^2\bigg)^{2r}\bigg]\\
&\quad +C\mathbb{E}\bigg[\bigg(\sup_{0\leq t\leq T}\int_{D}|\bm{m}_n|^2d\bm{x}\bigg)^{2r}\bigg]
+C\mathbb{E}\big[(I_n)^{2r}\big]+C.
\end{align*}}
\!\!Choosing $0<\varepsilon_1=4^{r-1}\alpha^{2r}-(4C_1)^{-1}$ in \eqref{PMpci33} and \eqref{PMpci333}, where $C_1>0$, and employing \eqref{MMpci2}, we deduce that for any $t\in[0,T]$,
{\small\begin{align}\label{PMpci4}
&\frac{1}{2}\mathbb{E}\bigg[\bigg(\int_{K}\sum\limits_{j=1}^{\infty}\lambda_j|\nabla\bm{\psi}_{jn}|^2d\bm{x}\bigg)^{2r}\bigg]
+\mathbb{E}\bigg[\bigg(\int_{K}|\nabla V_{n}|^2d\bm{x}\bigg)^{2r}\bigg]
+\mathbb{E}\bigg[\bigg(2\alpha\int_{0}^{t}\int_{D}|\partial_s\bm{m}_n|^2d\bm{x}ds\bigg)^{2r}\bigg]\nonumber\\
&\quad +\frac{1}{2}\mathbb{E}\big[|\bm{m}_n|_{\mathbb{H}^1}^{4r}\big]
+\mathbb{E}\bigg[\bigg(\int_{\mathbb{R}^3}|\bm{H}_{sn}|^2d\bm{x}\bigg)^{2r}\bigg]
+\mathbb{E}\bigg[\bigg(2\int_{D}w(\bm{m}_n)d\bm{x}\bigg)^{2r}\bigg]
+\mathbb{E}\bigg[\bigg(\frac{k}{2}\int_{D}(|\bm{m}_n|^2-1)^2d\bm{x}\bigg)^{2r}\bigg]\nonumber\\
&\quad +\mathbb{E}\bigg[\bigg(\int_{D}|G_n(\bm{m}_{n})|^2d\bm{x}\bigg)^{2r}\bigg]
+\frac{1}{2}\mathbb{E}\bigg[\bigg(\sup_{0\leq t\leq T}
\int_{K}\sum\limits_{j=1}^{\infty}\lambda_j|\nabla\bm{\psi}_{jn}|^2d\bm{x}\bigg)^{2r}\bigg]
+\mathbb{E}\bigg[\bigg(\sup_{0\leq t\leq T}\int_{K}|\nabla V_{n}|^2d\bm{x}\bigg)^{2r}\bigg]\nonumber\\
&\quad +\frac{1}{2}\mathbb{E}\bigg[\bigg(\sup_{0\leq t\leq T}|\bm{m}_n|_{\mathbb{H}^1}^2\bigg)^{2r}\bigg]
+\frac{1}{C_1}\mathbb{E}\bigg[\bigg(\int_0^T\int_{D}|\partial_t\bm{m}_n|^2d\bm{x}dt\bigg)^{2r}\bigg]
+\mathbb{E}\bigg[\bigg(\frac{k}{2}\sup_{0\leq t\leq T}\int_{D}(|\bm{m}_n|^2-1)^2d\bm{x}\bigg)^{2r}\bigg]\nonumber\\
&\quad +\mathbb{E}\bigg[\bigg(2\sup_{0\leq t\leq T}\int_{D}w(\bm{m}_n)d\bm{x}\bigg)^{2r}\bigg]
+\mathbb{E}\bigg[\bigg(\sup_{0\leq t\leq T}\int_{\mathbb{R}^3}|\bm{H}_{sn}|^2d\bm{x}\bigg)^{2r}\bigg]
+\mathbb{E}\bigg[\bigg(\sup_{0\leq t\leq T}\int_{D}|G_n(\bm{m}_{n})|^2d\bm{x}\bigg)^{2r}\bigg]\nonumber\\
&\leq C\mathbb{E}\big[|\bm{m}_n(0)|_{\mathbb{L}^2}^{4r}\big]
+C\mathbb{E}\big[(I_n)^{2r}\big]
+C\mathbb{E}\bigg[\int_0^T|\bm{m}_n|_{\mathbb{H}^1}^{4r}dt\bigg]
+C\mathbb{E}\bigg[\int_0^T\bigg(\sum\limits_{j=1}^{\infty}\lambda_j
|\nabla\bm{\psi}_{jn}|_{\mathbb{L}^{2}}^2\bigg)^{2r}dt\bigg]+C.
\end{align}}
Employing the Gronwall inequality, we obtain
\begin{align}\label{PMpci5}
&\mathbb{E}\bigg[\bigg(\int_{K}\sum\limits_{j=1}^{\infty}\lambda_j|\nabla\bm{\psi}_{jn}|^2d\bm{x}\bigg)^{2r}\bigg]
+\mathbb{E}\big[|\bm{m}_n|_{\mathbb{H}^1}^{4r}\big]
\leq C,
\end{align}
where $C$ depends on the initial data and $T$. Substituting \eqref{PMpci5} into \eqref{PMpci4} leads to \eqref{PMpci2}.

As a consequence of $w'(\bm{m}_n)=(0,2m_{n,2},2m_{n,3})$ and $\mathbb{L}^{2}(D)\hookrightarrow\mathbb{L}^{s}(D)(1\leq s\leq2)$, we obtain
\begin{align*}
\mathbb{E}\bigg[\sup_{0\leq t\leq T}\bigg(\int_{D}|w'(\bm{m}_n)|^sd\bm{x}\bigg)^{2r}\bigg]\leq C,~ 1\leq s\leq2.
\end{align*}
According to the GNS inequality and the definitions of $\bm{s}_n$ and $\rho_n$, we deduce that for $1\leq s\leq3$,
\begin{align*}
&\mathbb{E}\bigg[\sup_{0\leq t\leq T}|\rho_n|_{\mathbb{L}^s(K)}^{2r}\bigg]
\leq\mathbb{E}\bigg[\sup_{0\leq t\leq T}\bigg(
\int_{K}\sum_{j=1}^{\infty}\lambda_j|\nabla\bm{\psi}_{jn}|^{2}d\bm{x}\bigg)^{2r}\bigg]
\leq C,\\
&\mathbb{E}\bigg[\sup_{0\leq t\leq T}|\bm{s}_n|_{\mathbb{L}^s(K)}^{2r}\bigg]
\leq\mathbb{E}\bigg[\sup_{0\leq t\leq T}\bigg(
\int_{K}\sum_{j=1}^{\infty}\lambda_j|\nabla\bm{\psi}_{jn}|^{2}d\bm{x}\bigg)^{2r}\bigg]
\leq C.
\end{align*}
Now, we observe that, by \eqref{S3},
\begin{align*}
\int_{K}\left(i\partial_t\bm{\psi}_{jn}+\frac{1}{2}\Delta\bm{\psi}_{jn}-V_n\bm{\psi}_{jn}
+\frac{1}{2}\bm{m}_n\cdot\hat{\bm{\sigma}}\bm{\psi}_{jn}\right)\theta_hd\bm{x}=0,
\end{align*}
which leads to
\begin{align*}
&\mathbb{E}\bigg[\sup_{0\leq t\leq T}\sup_{|\theta_h|}\bigg|\int_{K}i\partial_t\bm{\psi}_{jn}\theta_hd\bm{x}\bigg|\bigg]
\leq\frac{1}{2}\mathbb{E}\bigg[\sup_{0\leq t\leq T}\sup_{|\theta_h|}
\bigg|\int_{K}\Delta\bm{\psi}_{jn}\theta_hd\bm{x}\bigg|\bigg]\\
&\quad +\mathbb{E}\bigg[\sup_{0\leq t\leq T}\sup_{|\theta_h|}\bigg|\int_{K}V_n\bm{\psi}_{jn}\theta_hd\bm{x}\bigg|\bigg]
+\frac{1}{2}\mathbb{E}\bigg[\sup_{0\leq t\leq T}\sup_{|\theta_h|}\bigg|
\int_{K}\bm{m}_n\cdot\hat{\bm{\sigma}}\bm{\psi}_{jn}\theta_hd\bm{x}\bigg|\bigg].
\end{align*}
Integrating by parts, we know that
\begin{align*}
&\frac{1}{2}\mathbb{E}\bigg[\sup_{0\leq t\leq T}\sup_{|\theta_h|}\bigg|
\int_{K}\Delta\bm{\psi}_{jn}\theta_hd\bm{x}\bigg|\bigg]
\leq\frac{1}{2}\mathbb{E}\bigg[\sup_{0\leq t\leq T}\sup_{|\theta_h|}|\nabla\bm{\psi}_{jn}|_{\mathbb{L}^{2}(K)}\cdot
|\nabla\theta_h|_{\mathbb{L}^{2}(K)}\bigg]\\
&\leq\frac{1}{4}\mathbb{E}\bigg[\sup_{|\theta_h|}\sup_{0\leq t\leq T}|\nabla\bm{\psi}_{jn}|_{\mathbb{L}^{2}(K)}^2
+\sup_{0\leq t\leq T}\sup_{|\theta_h|}|\nabla\theta_h|_{\mathbb{L}^{2}(K)}^2\bigg]\\
&\leq C+\frac{T}{4}\mathbb{E}\bigg[\sup_{|\theta_h|}|\nabla\theta_h|_{\mathbb{L}^{2}(K)}^2\bigg].
\end{align*}
Using the Poincar\'{e} inequality, we have
\begin{align*}
&\mathbb{E}\bigg[\sup_{0\leq t\leq T}\sup_{|\theta_h|}\bigg|\int_{K}V_n\bm{\psi}_{jn}\theta_hd\bm{x}\bigg|\bigg]
\leq C\mathbb{E}\bigg[\sup_{0\leq t\leq T}\sup_{|\theta_h|}\big(|\nabla V_n|_{\mathbb{L}^{2}(K)}\cdot
|\nabla\bm{\psi}_{jn}|_{\mathbb{L}^{2}(K)}\cdot|\theta_h|_{\mathbb{L}^{2}(K)}\big)\bigg]\\
&\leq\frac{C}{4}\mathbb{E}\bigg[\sup_{|\theta_h|}\sup_{0\leq t\leq T}|\nabla V_n|_{\mathbb{L}^{2}(K)}^4\bigg]
+\frac{C}{4}\mathbb{E}\bigg[\sup_{|\theta_h|}\sup_{0\leq t\leq T}|\nabla\bm{\psi}_{jn}|_{\mathbb{L}^{2}(K)}^4\bigg]
+\frac{CT}{2}\mathbb{E}\bigg[\sup_{|\theta_h|}|\theta_h|_{\mathbb{L}^{2}(K)}^2\bigg]\\
&\leq C+\frac{CT}{2}\mathbb{E}\bigg[\sup_{|\theta_h|}|\theta_h|_{\mathbb{L}^{2}(K)}^2\bigg],
\end{align*}
and
\begin{align*}
&\frac{1}{2}\mathbb{E}\bigg[\sup_{0\leq t\leq T}\sup_{|\theta_h|}
\bigg|\int_{K}\bm{m}_n\cdot\hat{\bm{\sigma}}\bm{\psi}_{jn}\theta_hd\bm{x}\bigg|\bigg]
\leq\frac{C}{8}\mathbb{E}\bigg[\sup_{|\theta_h|}\sup_{0\leq t\leq T}|\bm{m}_n|_{\mathbb{H}^{1}(K)}^4\bigg]\\
&\quad +\frac{C}{8}\mathbb{E}\bigg[\sup_{|\theta_h|}\sup_{0\leq t\leq T}|\nabla\bm{\psi}_{jn}|_{\mathbb{L}^{2}(K)}^4\bigg]
+\frac{CT}{4}\mathbb{E}\bigg[\sup_{|\theta_h|}|\theta_h|_{\mathbb{L}^{2}(K)}^2\bigg]\\
&\leq C+\frac{CT}{4}\mathbb{E}\bigg[\sup_{|\theta_h|}|\theta_h|_{\mathbb{L}^{2}(K)}^2\bigg].
\end{align*}
Hence
\begin{align*}
\mathbb{E}\bigg[\sup_{0\leq t\leq T}\sup_{|\theta_h|}\bigg|\int_{K}i\partial_t\bm{\psi}_{jn}\theta_hd\bm{x}\bigg|\bigg]
\leq& C_1+C_2\mathbb{E}\bigg[\sup_{|\theta_h|}|\theta_h|_{\mathbb{H}^{1}(K)}^2\bigg], ~h=1,2,\cdots,n.
\end{align*}
Since $|\partial_t\bm{\psi}_{jn}|_{\mathbb{H}^{-1}(K)}=\sup\limits_{|v|_{\mathbb{H}_{0}^{1}(K)}\leq 1}
\big|\int_{K}\partial_t\bm{\psi}_{jn}\cdot vd\bm{x}\big|$, $v\in\mathbb{H}_{0}^{1}(K)$, and $\theta_k$ is an orthogonal basis for $\mathbb{H}_{0}^{1}(K)$, we deduce that $\mathbb{E}\big[\sup_{0\leq t\leq T}|\partial_t\bm{\psi}_{jn}|_{\mathbb{H}^{-1}(K)}\big]\leq C$.
\end{proof}
\section{Tightness of the first-layer approximation sequence}
Let $\mathcal{L}(\bm{m}_n,\bm{\psi}_{jn},W_n,\eta_n)
=\mathcal{L}(\bm{m}_n)\times\mathcal{L}(\bm{\psi}_{jn})\times\mathcal{L}(W_n)\times\mathcal{L}(\eta_n)$, where $\mathcal{L}(\bm{m}_n)$ is the law of $\bm{m}_n$ on $\mathscr{Z}_{T}:=L_w^2(0,T;\mathbb{H}^1)\cap\mathbb{D}([0,T];X^{-\beta})\cap L^p(0,T;\mathbb{L}^q)\cap\mathbb{D}([0,T];\mathbb{H}_w^1)$. $\mathcal{L}(\bm{\psi}_{jn})$ is the law of $\bm{\psi}_{jn}$ on $\mathscr{Z}_{T}:=L_w^2(0,T;\mathbb{H}^1)\cap \mathbb{D}([0,T];X^{-\beta})\cap L^p(0,T;\mathbb{L}^q)\cap
\mathbb{D}([0,T];\mathbb{H}_w^1)$. $\mathcal{L}(W_n)$ is the law of
$W_n$ on $\mathscr{Z}_{W_n,T}:=C([0,T];\mathbb{R}^{N})$.
$\mathcal{L}(\eta_n)$ is the law of $\eta_n$ on $\mathscr{Z}_{\eta_n,T}:=\mathbb{M}_{\bar{\mathbb{N}}}([0,T]\times B)$.
Here $\mathbb{M}_{\bar{\mathbb{N}}}(S)$ represents the set of all measurable measures taking values in $\mathbb{N}\cup\infty$ on the measurable space $(S,\mathscr{S})$.
\begin{lemma}\label{taijinxing7}
The set $\{\mathcal{L}(\bm{m}_n),n\in\mathbb{N}\}$ is tight on the space $(\mathscr{Z}_{T},\mathcal{J})$.
\end{lemma}
\begin{proof}
By Lemma \ref{SPMpci}, we infer that $\sup\limits_{n\in\mathbb{N}}\mathbb{E}\big[|\bm{m}_n|_{L^{\infty}(0,T;\mathbb{H}^1)}^2\big]\leq C$. Next, we need to prove that the sequence satisfies the Aldous condition in $X^{-\beta}$. Let $\{\tau_n\}_{n\in\mathbb{N}}$ be a stopping time sequence such that $0\leq\tau_n+\theta\leq T$, and according to \eqref{SLLG4}, we get that
\begin{align*}
\alpha\bm{m}_n(t)=&\alpha\bm{m}_n(0)-\int_{0}^{t}\left(\bm{m}_n(s)\times\partial_s\bm{m}_n(s)
-\bm{H}_n(s)+k(|\bm{m}_n(s)|^2-1)\bm{m}_n(s)\right)ds\\
&-\frac{1}{2}\int_{0}^{t}G'_{n}(\bm{m}_n(s))[G_{n}(\bm{m}_n(s))]ds
-\int_{0}^{t}G_{n}(\bm{m}_n(s))dW(s)\\
&-\int_{0}^{t}\int_BF_{n}(\bm{m}_n(s-),l)\tilde{\eta}(ds,dl)\\
=:&\sum\limits_{i=0}^{6}J_{n}^{i}(t),
\end{align*}
$\mathbb{P}$-a.s. in $\mathbb{H}_n$, for all $t\in[0,T]$. Let $\theta>0$. To prove that $\bm{m}_n$ satisfies the Aldous condition, it is sufficient to prove that each term $J_{n}^{i}, i=0,1,\cdots,6$ satisfies the same condition. Obviously, $J_{n}^{0}$ satisfies \eqref {taijinxing2.1}.
Using the embedding $\mathbb{L}^{\frac{3}{2}}\hookrightarrow X^{-\beta}$, the H\"{o}lder inequality and Lemma \ref{SPMpci}, we can show that
\begin{align*}
&\mathbb{E}\big[|J_{n}^{1}(\tau_n+\theta)-J_{n}^{1}(\tau_n)|_{X^{-\beta}}\big]
=\mathbb{E}\bigg[\bigg|\int_{\tau_n}^{\tau_n+\theta}\bm{m}_n\times\partial_s\bm{m}_nds\bigg|_{X^{-\beta}}\bigg]\\
&\leq\mathbb{E}\bigg[\int_{\tau_n}^{\tau_n+\theta}|\bm{m}_n\times\partial_s\bm{m}_n|_{\mathbb{L}^{3/2}}ds\bigg]
\leq\mathbb{E}\bigg[\int_{\tau_n}^{\tau_n+\theta}|\bm{m}_n|_{\mathbb{L}^6}|\partial_s\bm{m}_n|_{\mathbb{L}^2}ds\bigg]\\
&\leq C\theta^{\frac{1}{2}}\mathbb{E}\big[|\bm{m}_n|_{L^{\infty}(0,T;\mathbb{H}^1)}
|\partial_s\bm{m}_n|_{L^{2}(0,T;\mathbb{L}^2)}\big]
\leq C\theta^{\frac{1}{2}}.
\end{align*}
From \eqref{bzh2}, \eqref{n} and Lemma \ref{SPMpci}, we obtain
\begin{align*}
&\mathbb{E}\big[|J_{n}^{2}(\tau_n+\theta)-J_{n}^{2}(\tau_n)|_{X^{-\beta}}\big]
=\mathbb{E}\bigg[\bigg|\int_{\tau_n}^{\tau_n+\theta}\bm{H}_nds\bigg|_{X^{-\beta}}\bigg]
\leq\mathbb{E}\bigg[\int_{\tau_n}^{\tau_n+\theta}|\bm{H}_n|_{\mathbb{L}^2}ds\bigg]\\
&\leq C\mathbb{E}\bigg[\int_{\tau_n}^{\tau_n+\theta}|\bm{m}_n|_{\mathbb{L}^2}+|w'(\bm{m}_n)|_{\mathbb{L}^2}
+|\bm{H}_{sn}|_{\mathbb{L}^2}+\frac{1}{2}|\bm{s}_n|_{\mathbb{L}^2}ds\bigg]\\
&\leq C\theta^{\frac{1}{2}}\mathbb{E}\big[|\bm{m}_n|_{L^{\infty}(0,T;\mathbb{L}^2)}
+|w'(\bm{m}_n)|_{L^{\infty}(0,T;\mathbb{L}^2)}
+|\bm{H}_{sn}|_{L^{\infty}(0,T;\mathbb{L}^2)}
+\frac{1}{2}|\bm{s}_n|_{L^{\infty}(0,T;\mathbb{L}^2)}\big]\\
&\leq C\theta^{\frac{1}{2}}.
\end{align*}
We can apply the embedding $\mathbb{L}^{3/2}\hookrightarrow X^{-\beta}$, and Lemma \ref{SPMpci} to see that
\begin{align*}
&\mathbb{E}\big[|J_{n}^{3}(\tau_n+\theta)-J_{n}^{3}(\tau_n)|_{X^{-\beta}}\big]
=\mathbb{E}\bigg[\bigg|\int_{\tau_n}^{\tau_n+\theta}k(|\bm{m}_n|^2-1)\bm{m}_nds\bigg|_{X^{-\beta}}\bigg]\\
&\leq\mathbb{E}\bigg[\int_{\tau_n}^{\tau_n+\theta}k|(|\bm{m}_n|^2-1)\bm{m}_n|_{\mathbb{L}^{3/2}}ds\bigg]
\leq C\mathbb{E}\bigg[\int_{\tau_n}^{\tau_n+\theta}k||\bm{m}_n|^2-1|_{\mathbb{L}^2}|\bm{m}_n|_{\mathbb{L}^6}ds\bigg]\\
&\leq Ck\theta^{\frac{1}{2}}\mathbb{E}\big[||\bm{m}_n|^2-1|_{L^{\infty}(0,T;\mathbb{L}^2)}
|\bm{m}_n|_{L^{\infty}(0,T;\mathbb{H}^1)}\big]
\leq C\theta^{\frac{1}{2}}.
\end{align*}
Similarly,
\begin{align*}
&\bigg(\mathbb{E}\big[|J_{n}^{4}(\tau_n+\theta)-J_{n}^{4}(\tau_n)|_{X^{-\beta}}\big]\bigg)^2
\leq\mathbb{E}\bigg[\bigg|\int_{\tau_n}^{\tau_n+\theta}G'_n(\bm{m}_n)[G_n(\bm{m}_n)]ds\bigg|_{X^{-\beta}}^2\bigg]\\
&\leq C\mathbb{E}\bigg[\int_{\tau_n}^{\tau_n+\theta}|G_n'(\bm{m}_n)[G_n(\bm{m}_n)]|_{\mathbb{L}^2}^2ds\bigg]
\leq C\theta\mathbb{E}\big[1+|\bm{m}_n|_{L^{\infty}(0,T;\mathbb{L}^2)}^2\big]
\leq C\theta.
\end{align*}
In view of Assumption \ref {ass1}, the It\^{o}-L\'{e}vy isometry, the H\"{o}lder inequality and Lemma \ref{SPMpci},
we deduce that
\begin{align*}
&\bigg(\mathbb{E}\big[|J_{n}^{5}(\tau_n+\theta)-J_{n}^{5}(\tau_n)|_{X^{-\beta}}\big]\bigg)^2
\leq\mathbb{E}\bigg[\bigg|\int_{\tau_n}^{\tau_n+\theta}G_n(\bm{m}_n)dW(s)\bigg|_{X^{-\beta}}^2\bigg]\\
&\leq\mathbb{E}\bigg[\bigg|\int_{\tau_n}^{\tau_n+\theta}G_n(\bm{m}_n)dW(s)\bigg|_{\mathbb{L}^2}^2\bigg]
=\mathbb{E}\bigg[\bigg|\int_{\tau_n}^{\tau_n+\theta}\sum\limits_{i\geq1}G_n(\bm{m}_n)\tilde{e}_idW_i(s)\bigg|_{\mathbb{L}^2}^2\bigg]\\
&\leq C\mathbb{E}\bigg[\int_{\tau_n}^{\tau_n+\theta}\bigg|\sum\limits_{i\geq1}G_{in}(\bm{m}_n)\bigg|_{\mathbb{L}^2}^2ds\bigg]
\leq C\theta\mathbb{E}\big[1+|\bm{m}_n|_{L^{\infty}(0,T;\mathbb{L}^2)}^2\big]
\leq C\theta.
\end{align*}
Similarly, we have
\begin{align*}
&\bigg(\mathbb{E}\big[|J_{n}^{6}(\tau_n+\theta)-J_{n}^{6}(\tau_n)|_{X^{-\beta}}\big]\bigg)^2
\leq\mathbb{E}\bigg[\bigg|\int_{\tau_n}^{\tau_n+\theta}\int_BF_n(\bm{m}_n,l)\tilde{\eta}(ds,dl)\bigg|_{X^{-\beta}}^2\bigg]\\
&\leq\mathbb{E}\bigg[\bigg|\int_{\tau_n}^{\tau_n+\theta}\int_BF_n(\bm{m}_n,l)\tilde{\eta}(ds,dl)\bigg|_{\mathbb{L}^2}^2\bigg]
\leq C\mathbb{E}\bigg[\int_{\tau_n}^{\tau_n+\theta}\int_B|F_n(\bm{m}_n,l)|_{\mathbb{L}^2}^2\mu(dl)ds\bigg]\\
&\leq C\theta\mathbb{E}\big[1+|\bm{m}_n|_{L^{\infty}(0,T;\mathbb{L}^2)}^2\big]
\leq C\theta.
\end{align*}
Thus, when $\alpha=1$, $\gamma=\frac{1}{2}$, the terms $J_{n}^{0}$-$J_{n}^{6}$ satisfy \eqref{taijinxing2.1},
which yields the desired result of Lemma \ref{taijinxing7}.
\end{proof}
\begin{lemma}\label{taijinxing8}
The set $\{\mathcal{L}(\bm{\psi}_{jn}),n\in\mathbb{N}\}$ is tight on the space $(\mathscr{Z}_{T},\mathcal{J})$.
\end{lemma}
\begin{proof}
By Lemma \ref{SPMpci}, we deduce $\sup\limits_{n\in\mathbb{N}}\mathbb{E}\big[|\bm{\psi}_{jn}|_{L^{\infty}(0,T;\mathbb{H}^1)}^2\big]\leq C$. Then it is sufficient to
show that the sequence satisfies the Aldous condition in $X^{-\beta}$.
Let $\{\tau_n\}_{n\in\mathbb{N}}$ be a stopping time sequence such that $0\leq\tau_n+\theta\leq T$, and according to \eqref{S3}, we can write
\begin{align*}
i\bm{\psi}_{jn}(t)
&=i\bm{\psi}_{jn}(0)-\frac{1}{2}\int_{0}^{t}\big(\Delta\bm{\psi}_{jn}(s)-V_n(s)\bm{\psi}_{jn}(s)
+\frac{1}{2}\bm{m}_n(s)\cdot\hat{\bm{\sigma}}\bm{\psi}_{jn}(s)\big)ds\\
&=:\sum\limits_{i=0}^{3}I_{n}^{i}(t),
\end{align*}
$\mathbb{P}$-a.s. in $\mathbb{H}_n$, for all $t\in[0,T]$. Suppose that $\theta>0$. To prove that $\bm{\psi}_{jn}$ satisfies the Aldous condition,
it is enough to prove that each term $I_{n}^{i}, i=0,1,\cdots,3$ satisfies the same condition. Obviously, $I_{n}^{0}$ satisfies \eqref{taijinxing2.1}.
Using the embedding $\mathbb{L}^2\hookrightarrow X^{-\beta}$, the H\"{o}lder inequality and Lemma \ref{SPMpci}, we have
\begin{align*}
&\bigg(\mathbb{E}\big[|I_{n}^{1}(\tau_n+\theta)-I_{n}^{1}(\tau_n)|_{X^{-\beta}}\big]\bigg)^2
=\mathbb{E}\bigg[\bigg|\int_{\tau_n}^{\tau_n+\theta}\Delta\bm{\psi}_{jn}ds\bigg|_{X^{-\beta}}^2\bigg]\\
&\leq\mathbb{E}\bigg[\int_{\tau_n}^{\tau_n+\theta}|\Delta\bm{\psi}_{jn}|_{\mathbb{L}^2}^2ds\bigg]
\leq C\theta\mathbb{E}\big[|\bm{\psi}_{jn}|_{L^{\infty}(0,T;\mathbb{L}^2)}^2\big]
\leq C\theta.
\end{align*}
\begin{align*}
&\mathbb{E}\big[|I_{n}^{2}(\tau_n+\theta)-I_{n}^{2}(\tau_n)|_{X^{-\beta}}\big]
=\mathbb{E}\bigg[\bigg|\int_{\tau_n}^{\tau_n+\theta}V_n\bm{\psi}_{jn}ds\bigg|_{X^{-\beta}}\bigg]\\
&\leq\mathbb{E}\bigg[\int_{\tau_n}^{\tau_n+\theta}|V_n\bm{\psi}_{jn}|_{\mathbb{L}^2}ds\bigg]
\leq C\theta^{\frac{1}{2}}\mathbb{E}\big[|V_n|_{L^2(0,T;\mathbb{L}^6)}
|\bm{\psi}_{jn}|_{L^{\infty}(0,T;\mathbb{H}^1)}\big]
\leq C\theta^{\frac{1}{2}}.
\end{align*}
\begin{align*}
&\mathbb{E}\big[|I_{n}^{3}(\tau_n+\theta)-I_{n}^{3}(\tau_n)|_{X^{-\beta}}\big]
=\mathbb{E}\bigg[\bigg|\int_{\tau_n}^{\tau_n+\theta}\bm{m}_n\cdot\hat{\bm{\sigma}}\bm{\psi}_{jn}ds\bigg|_{X^{-\beta}}\bigg]\\
&\leq\mathbb{E}\bigg[\int_{\tau_n}^{\tau_n+\theta}|\bm{m}_n\bm{\psi}_{jn}|_{\mathbb{L}^2}ds\bigg]
\leq C\theta^{\frac{1}{2}}\mathbb{E}\big[|\bm{m}_n|_{L^{\infty}(0,T;\mathbb{H}^1)}
|\bm{\psi}_{jn}|_{L^{\infty}(0,T;\mathbb{H}^1)}\big]
\leq C\theta^{\frac{1}{2}}.
\end{align*}
Thus, when $\alpha=1$, $\gamma=\frac{1}{2}$, the terms $I_{n}^{0}$-$I_{n}^{3}$ all satisfy \eqref{taijinxing2.1}. We complete the proof of Lemma \ref{taijinxing8}.
\end{proof}
\section{Existence of weak martingale solutions of the SPLLG system \eqref{SLLG3}}
\subsection{Construction of new probability space and processes as $n\rightarrow\infty$}
From Lemma \ref{taijinxing7}, it follows that the sequence $\{\mathcal{L}(\bm{m}_n),n\in\mathbb{N}\}$ is tight on $(\mathscr{Z}_{T},\mathcal{J})$. From Lemma \ref{taijinxing8}, it follows that the sequence $\{\mathcal{L}(\bm{\psi}_{jn}),n\in\mathbb{N}\}$ is tight on $(\mathscr{Z}_{T},\mathcal{J})$. Let $W_n:=W$, then $\mathcal{L}(W_n)$ is tight on $\mathscr{Z}_{W_n,T}$. Let $\eta_n:=\eta$, then $\mathcal{L}(\eta_n)$ is tight on $\mathscr{Z}_{\eta_n,T}$.
Therefore, $\{\mathcal{L}(\bm{m}_n, \bm{\psi}_{jn}, W_n,\eta_n),n\in\mathbb{N}\}$ is tight in $\mathscr{Z}_T\times\mathscr{Z}_T\times\mathscr{Z}_{W_n,T}\times\mathscr{Z}_{\eta_n,T}$, as shown in \cite{BM1, M}.
Here, we denote $\bm{\Psi}_{n}:=\bm{\Psi}_{n}^{k}$ and $\bm{m}_n:=\bm{m}_n^k$.
Therefore, by Theorem \ref{2SKO}, there exist a subsequence $\{n_i\}_{i\in\mathbb{N}}$, a filtered probability space $(\bar{\Omega},\bar{\mathscr{F}},(\bar{\mathscr{F}}_t)_{t\geq 0},\bar{\mathbb{P}})$ and on this space, $\mathscr{Z}_T\times\mathscr{Z}_T\times\mathscr{Z}_{W_n,T}\times\mathscr{Z}_{\eta_n,T}$-valued random variables $(\bm{m}^k,\bm{\psi}_j^k,W^k,\eta^k)$,
$(\bar{\bm{m}}_n,\bar{\bm{\psi}}_{jn},\bar{W}_n,\bar{\eta}_n), n\in\mathbb{N}$, such that\\
{\rm(1)} $\mathcal{L}((\bar{\bm{m}}_n, \bar{\bm{\psi}}_{jn},\bar{W}_n, \bar{\eta}_n))
=\mathcal{L}((\bm{m}_{n_i}, \bm{\psi}_{jn_i},W_{n_i}, \eta_{n_i}))$, for all $i\in\mathbb{N}$;\\
{\rm(2)} $(\bar{\bm{m}}_n,\bar{\bm{\psi}}_{jn},\bar{W}_n, \bar{\eta}_n)\rightarrow(\bm{m}^k, \bm{\psi}_j^k,W^k, \eta^k)$ in $\mathscr{Z}_T\times\mathscr{Z}_T\times\mathscr{Z}_{W_n,T}\times\mathscr{Z}_{\eta_n,T}$ as $n\rightarrow\infty$, for $\bar{\mathbb{P}}$-a.s. $\bar{\omega}$;\\
{\rm(3)} $(\bar{W}_n(\bar{\omega}),\bar{\eta}_n(\bar{\omega}))=(W^k(\bar{\omega}),\eta^k(\bar{\omega}))$, for all $\bar{\omega}\in\bar{\Omega}$.

For convenience, we still use $(\bm{m}_n, \bm{\psi}_{jn}, W_n, \eta_n)_{n\in\mathbb{N}}$ and $(\bar{\bm{m}}_n, \bar{\bm{\psi}}_{jn}, \bar{W}_n, \bar{\eta}_n)_{n\in\mathbb{N}}$ to represent these sequences.
In addition, $\bar{\eta}_n, n\in\mathbb{N}$ and $\eta^k$ are time-homogeneous Poisson random measure with the intensity measure $\mu$ on $(B, \mathscr{B}(B))$.
For $p\in[2,\infty)$, $q\in[2,6)$ and $\beta>\frac{1}{4}$, we have
\begin{align}\label{sltj}
\bar{\bm{m}}_n\rightarrow\bm{m}^k~\text{in}~\mathscr{Z}_T,~\bar{\mathbb{P}}\text{-}a.s., \text{ and }
\bar{\bm{\psi}}_{jn}\rightarrow\bm{\psi}_j^k~\text{in}~\mathscr{Z}_T,~\bar{\mathbb{P}}\text{-}a.s.
\end{align}

Next, we specifically consider the case of $p=4$, $q=4$, and $\beta=\frac{1}{2}$. Note that $\mathbb{D}([0,T];\mathbb{H}_n)$ is a Polish space, then $L_w^2(0,T;\mathbb{H}^1)\cap\mathbb{D}([0,T];X^{-\frac{1}{2}})\cap L^4(0,T;\mathbb{L}^4)\cap
\mathbb{D}([0,T];\mathbb{H}_w^1)$ is a separable metric space.
According to Kuratowski theorem(\!\!\!\cite{VTC}), the Borel subsets of $\mathbb{D}([0,T];\mathbb{H}_n)$ are the Borel subsets of $L_w^2(0,T;\mathbb{H}^1)\cap
\mathbb{D}([0,T];X^{-\frac{1}{2}})\cap L^4(0,T;\mathbb{L}^4)\cap\mathbb{D}([0,T];\mathbb{H}_w^1)$, and satisfy $\mathbb{P}\{\bm{\psi}_{jn}\in\mathbb{D}([0,T];\mathbb{H}_n^{K})\}=1$ and $\mathbb{P}\{\bm{m}_n\in\mathbb{D}([0,T];\mathbb{H}_n^{D})\}=1$. Therefore, we assume that $\bar{\bm{\psi}}_{jn}$, $\bar{\bm{m}}_n$ take values in $\mathbb{H}_n^{K}$ and $\mathbb{H}_n^D$, respectively, and the laws of $(\bm{m}_n,\bm{\psi}_{jn})$ and $(\bar{\bm{m}}_n,\bar{\bm{\psi}}_{jn})$ are same in $\mathbb{D}([0,T];\mathbb{H}_n^D)\times\mathbb{D}([0,T];\mathbb{H}_n^{K})$.
\subsection{Properties of the constructed new process and limiting process}
Due to the fact that the law of the sequences $(\bm{m}_n,\bm{\psi}_{jn})_{n\in\mathbb{N}}$ and $(\bar{\bm{m}}_n,\bar{\bm{\psi}}_{jn})_{n\in\mathbb{N}}$ is same, the estimates of $(\bar{\bm{m}}_n, \bar{\bm{\psi}}_{jn})$ are similar to the sequences $(\bm{m}_n, \bm{\psi}_{jn})$, satisfying Lemmas \ref{P}--\ref{SPMpci}. As a result, according to the weak$^{\ast}$ compactness, we find that as $n\rightarrow\infty$,
\begin{align}\label{sjdyxtj}
&\partial_t\bar{\bm{\Psi}}_n\rightarrow\partial_t\bm{\Psi}^k~\text{weak}^{\ast}~\text{in}~
L^{2r}(\bar{\Omega};L^{\infty}(\mathbb{R}^{+};\mathcal{H}_{\lambda}^{-1})),~
\partial_t\bar{\bm{m}}_n\rightarrow\partial_t\bm{m}^k~\text{weakly in}~
L^{2r}(\bar{\Omega};L^{2}(\mathbb{R}^{+};\mathbb{L}^{2})),\nonumber\\
&G_n(\bar{\bm{m}}_n)\rightarrow G^k(\bm{m}^k)~\text{weak}^{\ast}~\text{in}~L^{2r}(\bar{\Omega};L^{\infty}(\mathbb{R}^{+};\mathbb{H}^{1})),
\bar{V}_n\rightarrow V^k~\text{weak}^{\ast}~\text{in}~L^{2r}(\bar{\Omega};L^{\infty}(\mathbb{R}^{+};\mathbb{L}^{6})),\nonumber\\
&|\bar{\bm{m}}_n|^2-1\rightarrow|\bm{m}^k|^2-1~\text{weak}^{\ast}~\text{in}~L^{2r}(\bar{\Omega};L^{\infty}(\mathbb{R}^{+};\mathbb{L}^{2})),
\nabla \bar{V}_n\rightarrow\nabla V^k~\text{weak}^{\ast}~\text{in}~L^{2r}(\bar{\Omega};L^{\infty}(\mathbb{R}^{+};\mathbb{L}^{2})),\nonumber\\
&\bar{\bm{s}}_n\rightarrow\bm{s}^k~\text{weak}^{\ast}~\text{in}~
L^{2r}(\bar{\Omega};L^{\infty}(\mathbb{R}^{+};\mathbb{L}^{s})),~
\bar{\rho}_n\rightarrow\rho^k~\text{weak}^{\ast}~\text{in}~
L^{2r}(\bar{\Omega};L^{\infty}(\mathbb{R}^{+};\mathbb{L}^{s})),~1\leq s\leq3,\nonumber\\
&w'(\bar{\bm{m}}_n)\rightarrow w'(\bm{m}^k)~\text{weak}^{\ast}~\text{in}~
L^{2r}(\bar{\Omega};L^{\infty}(\mathbb{R}^{+};\mathbb{L}^{s})),~1\leq s\leq2.
\end{align}
Furthermore, since
\begin{align*}
&\bar{\mathbb{E}}\big[|\bar{\bm{H}}_{sn}-\bm{H}_{s}^k|_{ L^2(0,T;\mathbb{L}^2(\mathbb{R}^3))}^{2r}\big]
=\bar{\mathbb{E}}\bigg[\bigg(\int_{0}^{T}\int_{\mathbb{R}^3}|-\nabla_{\bm{x}}(\nabla N\ast\cdot\bar{\bm{m}}_n)
+\nabla_{\bm{x}}(\nabla N\ast\cdot\bm{m}^k)|^2d\bm{x}dt\bigg)^r\bigg]\\
&=\bar{\mathbb{E}}\bigg[\bigg(\int_{0}^{T}\big|\nabla_{\bm{x}}^2N(\bm{x})\ast
(\bm{m}^k(\bm{x},t)-\bar{\bm{m}}_n(\bm{x},t))\big|_{\mathbb{L}^2(\mathbb{R}^3)}^2dt\bigg)^r\bigg]\\
&\leq|\nabla_{\bm{x}}^2N(\bm{x})|_{\mathbb{L}^1(\mathbb{R}^3)}^{2r}
\bar{\mathbb{E}}\big[|\bar{\bm{m}}_n-\bm{m}^k|_{L^2(0,T;\mathbb{L}^2)}^{2r}\big]
\rightarrow 0, ~n\rightarrow\infty,
\end{align*}
we have
\begin{align}\label{Hnk}
\bar{\bm{H}}_{sn}\rightarrow\bm{H}_{s}^k~\text{in}~L^{2r}(\bar{\Omega};L^2(0,T;\mathbb{L}^2(\mathbb{R}^3))).
\end{align}
\begin{lemma}\label{ydsl}
For any $\xi\in L^4(\bar{\Omega};L^4(0,T;\mathbb{L}^4))$, it holds that
\begin{align*}
\lim\limits_{n\rightarrow\infty}\bar{\mathbb{E}}\bigg[\int_{0}^{t}\langle\bar{\bm{m}}_n\times\partial_s\bar{\bm{m}}_n,\xi\rangle
_{\mathbb{L}^2}ds\bigg]
=\bar{\mathbb{E}}\bigg[\int_{0}^{t}\langle\bm{m}^k\times\partial_s\bm{m}^k,\xi\rangle_{\mathbb{L}^2}ds\bigg].
\end{align*}
\end{lemma}
\begin{proof}
Applying the H\"{o}lder inequality and Lemma \ref{SPMpci}, we infer that
\begin{align*}
\bar{\mathbb{E}}\bigg[\bigg|\int_{0}^{t}
\langle \bar{\bm{m}}_n\times\partial_s\bar{\bm{m}}_n,\xi\rangle_{\mathbb{L}^2}ds\bigg|^{2r}\bigg]
\leq C\bar{\mathbb{E}}\big[|\bar{\bm{m}}_n|_{L^4(0,T;\mathbb{L}^4)}^{2r}|\partial_s\bar{\bm{m}}_n|
_{L^{2}(0,T;\mathbb{L}^2)}^{2r}|\xi|_{L^4(0,T;\mathbb{L}^4)}^{2r}\big]
\leq C.
\end{align*}
We can use the facts that $\bar{\bm{m}}_n\rightarrow\bm{m}^k $ in $L^4(0,T;\mathbb{L}^4)$, $\bar{\mathbb{P}}$-a.s.,
$\partial_t\bar{\bm{m}}_n\rightarrow\partial_t\bm{m}^k$ weakly in $L^2(\bar{\Omega};L^{2}(\mathbb{R}^{+};\mathbb{L}^{2}))$ to obtain
\begin{align*}
&\bigg|\int_{0}^{t}\langle\bar{\bm{m}}_n\times\partial_s\bar{\bm{m}}_n-\bm{m}^k\times\partial_s\bm{m}^k,
\xi\rangle_{\mathbb{L}^2}ds\bigg|^2\\
&\leq\bigg|\int_{0}^{t}\langle\partial_s\bar{\bm{m}}_n-\partial_s\bm{m}^k,\xi\times\bar{\bm{m}}_n\rangle_{\mathbb{L}^2}ds\bigg|^2
+\bigg|\int_{0}^{t}\langle\bar{\bm{m}}_n-\bm{m}^k,\partial_s\bm{m}^k\times\xi\rangle_{\mathbb{L}^2}ds\bigg|^2\\
&\leq\bigg|\int_{0}^{t}\langle\partial_s\bar{\bm{m}}_n-\partial_s\bm{m}^k,\xi\times\bar{\bm{m}}_n\rangle_{\mathbb{L}^2}ds\bigg|^2
+|\bar{\bm{m}}_n-\bm{m}^k|_{L^4(0,T;\mathbb{L}^4)}^2|\partial_s\bm{m}^k|
_{L^{2}(0,T;\mathbb{L}^2)}^2|\xi|_{L^4(0,T;\mathbb{L}^4)}^2\\
&\rightarrow0,~\bar{\mathbb{P}}\text{-}a.s.
\end{align*}
for $\xi\times\bar{\bm{m}}_n\in L^2(\bar{\Omega};L^{2}(0,T;\mathbb{L}^{2}))$.
Next, we can get the desired result of Lemma \ref{ydsl} by employing the Vitali convergence theorem (for details see \cite{E}).
\end{proof}
\begin{lemma}\label{gdsl1}
For any $v\in L^4(\bar{\Omega};L^4(0,T;\mathbb{L}^2))$, and for all $t\in[0,T]$, we have
\begin{align*}
{\rm{(I)}}&\lim\limits_{n\rightarrow\infty}\bar{\mathbb{E}}\bigg[\bigg|\int_{0}^{t}
\langle G{^k}'(\bar{\bm{m}}_n)[G^k(\bar{\bm{m}}_n)]-G{^k}'(\bm{m}^k)[G^k(\bm{m}^k)],
v\rangle_{\mathbb{L}^2}ds\bigg|^2\bigg]=0,\\
{\rm{(II)}}&\lim\limits_{n\rightarrow\infty}\bar{\mathbb{E}}\bigg[\bigg|\int_{0}^{t}\langle G'_n(\bar{\bm{m}}_n)[G_n(\bar{\bm{m}}_n)]
-G{^k}'(\bm{m}^k)[G^k(\bm{m}^k)],v\rangle_{\mathbb{L}^2}ds\bigg|^2\bigg]=0,\\
{\rm{(III)}}&\lim\limits_{n\rightarrow\infty}\bar{\mathbb{E}}\bigg[\int_{0}^{t}\bigg|\sum\limits_{i\geq1}
\langle G_n(\bar{\bm{m}}_n)-G^k(\bm{m}^k),v\rangle_{\mathbb{L}^2}\tilde{e}_i\bigg|^2ds\bigg]=0,\\
{\rm{(IV)}}&\lim\limits_{n\rightarrow\infty}\bar{\mathbb{E}}\bigg[\int_{0}^{t}\int_B
\big|\langle F_n(l,\bar{\bm{m}}_n)-F^k(l,\bm{m}^k),v\rangle_{\mathbb{L}^2}\big|^2\mu(dl)ds\bigg]=0.
\end{align*}
\end{lemma}
\begin{proof}
According to Assumption \ref{ass1} and the H\"{o}lder inequality, we can infer that
\begin{align*}
&\bar{\mathbb{E}}\bigg[\bigg|\int_{0}^{t}
\langle G{^k}'(\bar{\bm{m}}_n)[G^k(\bar{\bm{m}}_n)],v\rangle_{\mathbb{L}^2}ds\bigg|^{2r}\bigg]
\leq C\bar{\mathbb{E}}\big[|v|_{L^4(0,T;\mathbb{L}^2)}^{2r}
|G{^k}'(\bar{\bm{m}}_n)[G^k(\bar{\bm{m}}_n)]|_{L^4(0,T;\mathbb{L}^2)}^{2r}\big]\\
&\leq C\big(\bar{\mathbb{E}}\big[|v|_{L^4(0,T;\mathbb{L}^2)}^{4}\big]\big)^{\frac{r}{2}}
\big(1+\bar{\mathbb{E}}\big[|\bar{\bm{m}}_n|_{L^4(0,T;\mathbb{L}^2)}^{\frac{4r}{2-r}}\big]\big)^{\frac{2-r}{2}}
\leq C
\end{align*}
for all $t\in[0,T]$, $r\in(1,2)$ and $n\in\mathbb{N}$, and then using \eqref{sltj}, Assumption \ref{ass1} and the H\"{o}lder inequality,
we proceed to get
\begin{align*}
&\bigg|\int_{0}^{t}\langle G{^k}'(\bar{\bm{m}}_n)[G^k(\bar{\bm{m}}_n)]
-G{^k}'(\bm{m}^k)[G^k(\bm{m}^k)],v\rangle_{\mathbb{L}^2}ds\bigg|^2\\
&\leq\int_{0}^{t}|v|_{\mathbb{L}^2}^{2}|G{^k}'(\bar{\bm{m}}_n)[G^k(\bar{\bm{m}}_n)]
-G{^k}'(\bm{m}^k)[G^k(\bm{m}^k)]|_{\mathbb{L}^2}^{2}ds\\
&\leq CK_1|v|_{L^4(0,T;\mathbb{L}^2)}^{2}|\bar{\bm{m}}_n-\bm{m}^k|_{L^4(0,T;\mathbb{L}^4)}^{2}
\rightarrow0,~\bar{\mathbb{P}}\text{-}a.s.
\end{align*}
Therefore, the proof of {\rm{(I)}} is complete by utilizing the Vitali convergence theorem.
{\rm{(II)}}, {\rm{(III)}} as well as {\rm{(IV)}} can be proved in a similar way.
\end{proof}
\subsection{Convergence of the new process to the corresponding limit process}
\begin{lemma}\label{psk}
The sequence $(\Omega, \mathscr{F}, (\mathscr{F}_t)_{t\geq0}, \mathbb{P}, \bm{\Psi}^k, \rho^k, \bm{s}^k, W^k,\eta^k)$ satisfies the first and second equation of \eqref{ss}, $\bar{\mathbb{P}}$-a.s.
\end{lemma}
\begin{proof}
Let $t\in[0,T]$, for $\zeta\in L^4(\bar{\Omega};L^4(0,T;\mathbb{L}^4))$, thanks to $\bar{\rho}_n\rightarrow\rho^k$ weak$^{\ast}$ in
$L^{2r}(\bar{\Omega};L^{\infty}(\mathbb{R}^{+};\mathbb{L}^{s})), 1\leq s\leq3$,  it follows that
\begin{align}\label{VVV1}
\lim\limits_{n\rightarrow\infty}\bar{\mathbb{E}}\bigg[\int_{0}^{T}
|\langle\bar{\rho}_n-\rho^k,\zeta\rangle_{\mathbb{L}^2}|^2dt\bigg]
=0.
\end{align}
By $\bar{\bm{\psi}}_{jn}\rightarrow\bm{\psi}_j^k$ in $L^4(0,T;\mathbb{L}^4)$, $\bar{\mathbb{P}}$-a.s., it follows from the H\"{o}lder inequality that
\begin{align*}
&\int_{0}^{T}\bigg|\sum\limits_{j=1}^{\infty}\lambda_j\langle|\bar{\bm{\psi}}_{jn}|^2-|\bm{\psi}_j^k|^2,
\zeta\rangle_{\mathbb{L}^2}\bigg|^2dt
\leq\int_0^T\bigg(\sum\limits_{j=1}^{\infty}\lambda_j\int_{K}|\bar{\bm{\psi}}_{jn}+\bm{\psi}_j^k|
|\bar{\bm{\psi}}_{jn}-\bm{\psi}_j^k||\zeta|d\bm{x}\bigg)^2dt\\
&\leq\int_0^T\sum\limits_{j=1}^{\infty}\lambda_j^2|\bar{\bm{\psi}}_{jn}+\bm{\psi}_j^k|_{\mathbb{L}^2}^2
|\bar{\bm{\psi}}_{jn}-\bm{\psi}_j^k|_{\mathbb{L}^4}^2|\zeta|_{\mathbb{L}^4}^2dt\\
&\leq\sum\limits_{j=1}^{\infty}\lambda_j^2|\bar{\bm{\psi}}_{jn}+\bm{\psi}_j^k|_{L^2(0,T;\mathbb{L}^2)}^2
|\bar{\bm{\psi}}_{jn}-\bm{\psi}_j^k|_{L^4(0,T;\mathbb{L}^4)}^2|\zeta|_{L^4(0,T;\mathbb{L}^4)}^2
\rightarrow 0,~\bar{\mathbb{P}}\text{-}a.s.
\end{align*}
Moreover, for all $t\in[0,T]$, $r\in(1,2)$, and all $n\in\mathbb{N}$, the following holds:
\begin{align*}
&\bar{\mathbb{E}}\bigg[\bigg|\int_{0}^{T}\bigg|\sum\limits_{j=1}^{\infty}\lambda_j\langle|\bar{\bm{\psi}}_{jn}|^2-|\bm{\psi}_j^k|^2,
\zeta\rangle_{\mathbb{L}^2}\bigg|^2dt\bigg|^r\bigg]\\
&\leq C\sum\limits_{j=1}^{\infty}\lambda_j^{2r}\bar{\mathbb{E}}
\big[|\bar{\bm{\psi}}_{jn}+\bm{\psi}_j^k|_{L^2(0,T;\mathbb{L}^2)}^{2r}
|\bar{\bm{\psi}}_{jn}-\bm{\psi}_j^k|_{L^4(0,T;\mathbb{L}^4)}^{2r}|\zeta|_{L^4(0,T;\mathbb{L}^4)}^{2r}\big]\\
&\leq C\big(\bar{\mathbb{E}}\big[|\zeta|_{L^4(0,T;\mathbb{L}^4)}^{4}\big]\big)^{\frac{r}{2}}
\big(\bar{\mathbb{E}}\big[|\bar{\bm{\psi}}_{jn}|_{L^4(0,T;\mathbb{L}^4)}^{\frac{8r}{2-r}}
+|\bm{\psi}_j^k|_{L^4(0,T;\mathbb{L}^4)}^{\frac{8r}{2-r}}\big]\big)^{\frac{2-r}{4}}
\leq C.
\end{align*}
Hence, using the Vitali convergence theorem, we deduce that
\begin{align}\label{VVVV1}
\lim\limits_{n\rightarrow\infty}\bar{\mathbb{E}}
\bigg[\int_{0}^{T}\bigg|\sum\limits_{j=1}^{\infty}\lambda_j\langle|\bar{\bm{\psi}}_{jn}|^2-|\bm{\psi}_j^k|^2,
\zeta\rangle_{\mathbb{L}^2}\bigg|^2dt\bigg]=0.
\end{align}
Recall that $\rho_n$ satisfies $\rho_n=\sum_{j=1}^{\infty}\lambda_j|\bm{\psi}_{jn}|^2$ for all $t\in[0,T]$. In particular,
\begin{align*}
\int_{0}^{T}\mathbb{E}\bigg[\bigg|\langle\rho_n,\zeta\rangle
-\sum_{j=1}^{\infty}\lambda_j\langle|\bm{\psi}_{jn}|^2,\zeta\rangle\bigg|^2\bigg]dt=0.
\end{align*}
From $\mathcal{L}(\bm{\psi}_{jn})=\mathcal{L}(\bar{\bm{\psi}}_{jn})$, we obtain that
\begin{align*}
\int_{0}^{T}\bar{\mathbb{E}}\bigg[\bigg|\langle\bar{\rho}_n,\zeta\rangle
-\sum_{j=1}^{\infty}\lambda_j\langle|\bar{\bm{\psi}}_{jn}|^2,\zeta\rangle\bigg|^2\bigg]dt=0.
\end{align*}
Thus, using \eqref{VVV1} and \eqref{VVVV1}, we get
\begin{align*}
\int_{0}^{T}\bar{\mathbb{E}}\bigg[\bigg|\langle\rho^k,\zeta\rangle
-\sum_{j=1}^{\infty}\lambda_j\langle|\bm{\psi}_{j}^k|^2,\zeta\rangle\bigg|^2\bigg]dt=0.
\end{align*}
Therefore, for m-almost all $t\in[0,T]$ and $\mathbb{P}$-almost $\omega\in\bar{\Omega}$, we have $\langle\rho^k,\zeta\rangle=\sum_{j=1}^{\infty}\lambda_j\langle|\bm{\psi}_{j}^k|^2,\zeta\rangle$. Since $\bm{\psi}_j^k$ is a $\mathscr{Z}_T$-valued random variable, $\bm{\psi}_j^k\in\mathbb{D}([0,T];\mathbb{H}_{w}^1)$ a.e., that is, $\bm{\psi}_j^k$ is weakly c\`{a}dl\`{a}g with respect to time. For all $t\in[0,T]$, notice that two c\`{a}dl\`{a}g functions are equal for m-almost all $t\in[0,T]$ must be equal for all $t\in[0,T]$, hence for all $t\in[0,T]$, $\bar{\mathbb{P}}$-a.s., and for each $\zeta\in L^4(\bar{\Omega};L^4(0,T;\mathbb{L}^4))$, $\langle\rho^k,\zeta\rangle=\sum_{j=1}^{\infty}\lambda_j\langle|\bm{\psi}_{j}^k|^2,\zeta\rangle$ holds. Thus for all $t\in[0,T]$, we have that $\rho^k=\rho[\bm{\Psi}^k]=\sum\limits_{j=1}^{\infty}\lambda_j|\bm{\psi}_j^k|^2$, $\bar{\mathbb{P}}$-a.s. Similarly, we can prove that for all $t\in[0,T]$, $\bm{s}^k=\sum\limits_{j=1}^{\infty}\lambda_j({\bm{\psi}_j^k}^{\dag}\hat{\bm{\sigma}}\bm{\psi}_j^k)$,
$\bar{\mathbb{P}}$-a.s. holds.
\end{proof}
\begin{lemma}\label{Vpk}
The sequence $(\Omega, \mathscr{F}, (\mathscr{F}_t)_{t\geq0}, \mathbb{P}, V^k, \rho^k, W^k,\eta^k)$ satisfies the third equation of \eqref{S2}, $\bar{\mathbb{P}}$-a.s.
\end{lemma}
\begin{proof}
For $\nabla\zeta\in L^2(\bar{\Omega};L^2(0,T;\mathbb{L}^2))$, according to $\nabla \bar{V}_n\rightarrow\nabla V^k$ weak$^{\ast}$ in $L^{2r}(\bar{\Omega};L^{\infty}(\mathbb{R}^{+};\mathbb{L}^{2}))$,
it remains to show that
\begin{align}\label{VVV}
\lim\limits_{n\rightarrow\infty}\bar{\mathbb{E}}\bigg[\int_{0}^{T}
|\langle\Delta\bar{V}_n-\Delta V^k,\zeta\rangle_{\mathbb{L}^2}|^2dt\bigg]
=\lim\limits_{n\rightarrow\infty}\bar{\mathbb{E}}\bigg[\int_{0}^{T}
|\langle\nabla\bar{V}_n-\nabla V^k,\nabla\zeta\rangle_{\mathbb{L}^2}|^2dt\bigg]
=0.
\end{align}
Using the fact that $\bar{\rho}_n\rightarrow\rho^k$ weak$^{\ast}$ in $L^{2r}(\bar{\Omega};L^{\infty}(\mathbb{R}^{+};\mathbb{L}^{s})),~1\leq s\leq3$, we get
\begin{align}\label{VVVV}
\lim\limits_{n\rightarrow\infty}\bar{\mathbb{E}}\bigg[\int_{0}^{T}
|\langle\bar{\rho}_n-\rho^k,\zeta\rangle_{\mathbb{L}^2}|^2dt\bigg]
=0.
\end{align}
Note that $V_n$ satisfies $-\Delta V_n=\rho_n$ for all $t\in[0,T]$, then we have
\begin{align*}
\int_{0}^{T}\mathbb{E}\big[\big|\langle-\Delta V_n,\zeta\rangle-\langle\rho_n,\zeta\rangle\big|^2\big]dt=0.
\end{align*}
From $\mathcal{L}(\bm{\psi}_{jn})=\mathcal{L}(\bar{\bm{\psi}}_{jn})$ we conclude that,
\begin{align*}
\int_{0}^{T}\bar{\mathbb{E}}\big[\big|\langle-\Delta\bar{V}_n,\zeta\rangle-\langle\bar{\rho}_n,\zeta\rangle\big|^2\big]dt=0.
\end{align*}
Thus, we can apply \eqref{VVV} and \eqref{VVVV} to deduce that
\begin{align*}
\int_{0}^{T}\bar{\mathbb{E}}\big[\big|\langle-\Delta V^k,\zeta\rangle-\langle\rho^k,\zeta\rangle\big|^2\big]dt=0.
\end{align*}
Hence, for m-almost all $t\in[0,T]$ and $\mathbb{P}$-almost all $\omega\in\bar{\Omega}$, we have $\langle-\Delta V^k,\zeta\rangle=\langle\rho^k,\zeta\rangle$. Since $\bm{\psi}_j^k$ is weakly c\`{a}dl\`{a}g with respect to time, for all $t\in[0,T]$, $\bar{\mathbb{P}}$-a.s., and for each $\nabla\zeta\in L^2(\bar{\Omega};L^2(0,T;\mathbb{L}^2))$, $\langle-\Delta V^k,\zeta\rangle=\langle\rho^k,\zeta\rangle$ holds.
According to Lemma \ref{SPMpci}, we can deduce  that $\bar{\mathbb{E}}\big[\sup_{0\leq t\leq T}|\bar{\rho}_n|_{\mathbb{L}^2(K)}\big]\leq C$, and since
\begin{align*}
|\bar{V}_n|_{\mathbb{H}^2(K)}^2=\int_{K}|\bar{V}_n|^2+|\nabla\bar{V}_n|^2+|\nabla^2\bar{V}_n|^2d\bm{x}\leq C,~
\bar{\mathbb{P}}\text{-}a.s.,
\end{align*}
it follows from the weak lower semi-continuity of norms that
\begin{align*}
|V^k|_{\mathbb{H}^2(K)}^2\leq\lim_{n\rightarrow\infty}\inf|V_n|_{\mathbb{H}^2(K)}^2<\infty,
~\forall t\in\mathbb{R}^{+}, ~\mathbb{P}\text{-}a.s.
\end{align*}
Then, $V^k$ satisfies the third equation of system \eqref{S2}.
\end{proof}
\begin{lemma}\label{dycsl}
The sequence $(\Omega, \mathscr{F}, (\mathscr{F}_t)_{t\geq0}, \mathbb{P}, \bm{\Psi}^k,\bm{m}^k, V^k, \rho^k,\bm{s}^k,\bm{H}_s^k,W^k,\eta^k)$ satisfies the equation \eqref{SLLG3} in the sense of distribution, $\bar{\mathbb{P}}$-a.s., with a constant $C>0$ such that
{\small\begin{align}\label{kdgj}
&\mathbb{E}\bigg[\bigg(\sup_{0\leq t\leq T}
\int_{K}\sum\limits_{j=1}^{\infty}\lambda_j|\nabla\bm{\psi}_{j}^k|^2d\bm{x}\bigg)^{2r}\bigg]
+\mathbb{E}\bigg[\bigg(\sup_{0\leq t\leq T}\int_{K}|\nabla V^k|^2d\bm{x}\bigg)^{2r}\bigg]
+\mathbb{E}\bigg[\bigg(2\alpha\int_0^T\int_{D}|\partial_t\bm{m}^k|^2d\bm{x}dt\bigg)^{2r}\bigg]\nonumber\\
&\quad +\mathbb{E}\bigg[\bigg(\sup_{0\leq t\leq T}|\bm{m}^k|_{\mathbb{H}^1}^2\bigg)^{2r}\bigg]
+\mathbb{E}\bigg[\bigg(\sup_{0\leq t\leq T}\int_{\mathbb{R}^3}|\bm{H}_{s}^k|^2d\bm{x}\bigg)^{2r}\bigg]
+\mathbb{E}\bigg[\bigg(2\sup_{0\leq t\leq T}\int_{D}w(\bm{m}^k)d\bm{x}\bigg)^{2r}\bigg]\nonumber\\
&\quad +\mathbb{E}\bigg[\bigg(\frac{k}{2}\sup_{0\leq t\leq T}\int_{D}(|\bm{m}^k|^2-1)^2d\bm{x}\bigg)^{2r}\bigg]
+\mathbb{E}\bigg[\bigg(\sup_{0\leq t\leq T}\int_{D}|G^k(\bm{m}^k)|^2d\bm{x}\bigg)^{2r}\bigg]
\leq C.
\end{align}}
\end{lemma}
\begin{proof}
Let $t\in[0,T]$, $v\in L^4(\bar{\Omega};L^4(0,T;X^{\beta}))$, where $\beta=\frac{1}{2}$. We write
\begin{align}\label{M1}
&M_n(\bar{\bm{m}}_n,\bar{\bm{\psi}}_{jn},\bar{W}_n,\bar{\eta}_n,v)(t):=\alpha_{X^{-\beta}}\langle\bar{\bm{m}}_n(0),v\rangle_{X^{\beta}}
-\int_{0}^{t}\langle\bar{\bm{m}}_n(s)\times\partial_s\bar{\bm{m}}_n(s),v\rangle_{\mathbb{L}^2}ds\nonumber\\
&\quad +\int_{0}^{t}\langle\bar{\bm{H}}_n(s),v\rangle ds
-\int_{0}^{t}\langle k(|\bar{\bm{m}}_n(s)|^2-1)\bar{\bm{m}}_n(s),v\rangle ds
-\int_{0}^{t}\langle G_n(\bar{\bm{m}}_n(s)),v\rangle d\bar{W}_n(s)\\
&\quad -\frac{1}{2}\int_{0}^{t}\langle G'_n(\bar{\bm{m}}_n(s))[G_n(\bar{\bm{m}}_n(s))],v\rangle ds
-\int_{0}^{t}\int_B\langle F_n(l,\bar{\bm{m}}_n(s-)),v\rangle\tilde{\bar{\eta}}_n(ds,dl),\nonumber
\end{align}
and
\begin{align}\label{M2}
&M(\bm{m}^k,\bm{\psi}_j^k,W^k,\eta^k,v)(t):=\alpha_{X^{-\beta}}\langle\bm{m}^k(0),v\rangle_{X^{\beta}}
-\int_{0}^{t}\langle\bm{m}^k(s)\times\partial_s\bm{m}^k(s),v\rangle_{\mathbb{L}^2}ds\nonumber\\
&\quad +\int_{0}^{t}\langle\bm{H}^k(s),v\rangle ds
-\int_{0}^{t}\langle k(|\bm{m}^k(s)|^2-1)\bm{m}^k(s),v\rangle ds
-\int_{0}^{t}\langle G^k(\bm{m}^k(s)),v\rangle dW^k(s)\\
&\quad -\frac{1}{2}\int_{0}^{t}\langle {G^k}'(\bm{m}^k(s))[G^k(\bm{m}^k(s))],v\rangle ds
-\int_{0}^{t}\int_B\langle F^k(l,\bm{m}^k(s-)),v\rangle\tilde{\eta}^k(ds,dl).\nonumber
\end{align}
Let $\xi\in L^4(\bar{\Omega};L^4(0,T;X^{\beta}))$, where $\beta=\frac{1}{2}$. Notice that
\begin{align}\label{L1}
L_n(\bar{\bm{\psi}}_{jn},\bar{\bm{m}}_n,\xi)(t)
:=&i_{X^{-\beta}}\langle\bar{\bm{\psi}}_{jn}(0),\xi\rangle_{X^{\beta}}
-\frac{1}{2}\int_{0}^{t}\langle\Delta\bar{\bm{\psi}}_{jn}(s),\xi\rangle_{\mathbb{L}^2}ds\nonumber\\
&+\int_{0}^{t}\langle\bar{V}_n\bar{\bm{\psi}}_{jn},\xi\rangle ds
-\frac{1}{2}\int_{0}^{t}\langle\bar{\bm{m}}_n\cdot\hat{\bm{\sigma}}\bar{\bm{\psi}}_{jn},\xi\rangle ds,
\end{align}
and
\begin{align}\label{L2}
L(\bm{\psi}_j^k,\bm{m}^k,\xi)(t)
:=&i_{X^{-\beta}}\langle\bm{\psi}_j^k(0),\xi\rangle_{X^{\beta}}
-\frac{1}{2}\int_{0}^{t}\langle\Delta\bm{\psi}_j^k(s),\xi\rangle_{\mathbb{L}^2}ds\nonumber\\
&+\int_{0}^{t}\langle V^k\bm{\psi}_j^k,\xi\rangle ds
-\frac{1}{2}\int_{0}^{t}\langle\bm{m}^k\cdot\hat{\bm{\sigma}}\bm{\psi}_j^k,\xi\rangle ds.
\end{align}
In order to demonstrate that $(\Omega, \mathscr{F}, (\mathscr{F}_t)_{t\geq0}, \mathbb{P}, \bm{\Psi}^k,\bm{m}^k, V^k, \rho^k,\bm{s}^k,\bm{H}_s^k,W^k,\eta^k)$ is a weak martingale solution of the system \eqref{SLLG3}, we divide the proof into the following four steps:

Step 1: We intend to prove the followings for all $v\in L^4(\bar{\Omega};L^4(0,T;X^{\beta}))$:
\begin{align}
&{\rm(a)}\ \lim\limits_{n\rightarrow\infty}\bar{\mathbb{E}}\bigg[\int_{0}^{T}|\langle\bar{\bm{m}}_n(t)-\bm{m}^k(t),v\rangle
_{\mathbb{L}^2}|^2dt\bigg]=0,\label{tj001}\\
&{\rm(b)}\ \lim\limits_{n\rightarrow\infty}\bar{\mathbb{E}}\bigg
[\int_{0}^{T}|M_n(\bar{\bm{m}}_n,\bar{\bm{\psi}}_{jn},\bar{W}_n,\bar{\eta}_n,v)(t)
-M(\bm{m}^k,\bm{\psi}_j^k,W^k,\eta^k,v)(t)|^2dt\bigg]=0.\label{tj002}
\end{align}
To prove \eqref{tj001}, it follows from \eqref{sltj} that $\bar{\bm{m}}_{n}\rightarrow\bm{m}^k$ in $L^4(0,T;\mathbb{L}^4)$, $\bar{\mathbb{P}}$-a.s., which implies
\begin{align*}
&\int_{0}^{T}|\langle\bar{\bm{m}}_n(t)-\bm{m}^k(t),v\rangle_{\mathbb{L}^2}|^2dt
\leq|v|_{L^4(0,T;\mathbb{L}^4)}^2|\bar{\bm{m}}_n(t)-\bm{m}^k(t)|_{L^4(0,T;\mathbb{L}^4)}^2
\rightarrow 0,~\bar{\mathbb{P}}\text{-}a.s.
\end{align*}
In addition, for all $t\in[0,T]$, $r\in(1,2)$ and all $n\in\mathbb{N}$, it holds that
\begin{align*}
&\bar{\mathbb{E}}\bigg[\bigg|\int_{0}^{T}|\langle\bar{\bm{m}}_n(t)-\bm{m}^k(t),v\rangle_{\mathbb{L}^2}|^2dt\bigg|^r\bigg]
\leq C\bar{\mathbb{E}}\big[|v|_{L^4(0,T;\mathbb{L}^2)}^{2r}|\bar{\bm{m}}_n(t)-\bm{m}^k(t)|_{L^4(0,T;\mathbb{L}^2)}^{2r}\big]\\
&\leq C\big(\bar{\mathbb{E}}\big[|v|_{L^4(0,T;\mathbb{L}^2)}^{4}\big]\big)^{\frac{r}{2}}
\big(\bar{\mathbb{E}}\big[|\bar{\bm{m}}_n(t)|_{L^4(0,T;\mathbb{L}^2)}^{\frac{4r}{2-r}}
+|\bm{m}^k(t)|_{L^4(0,T;\mathbb{L}^2)}^{\frac{4r}{2-r}}\big]\big)^{\frac{2-r}{2}}
\leq C.
\end{align*}
Therefore, by the Vitali convergence theorem, we have
\begin{align*}
\lim\limits_{n\rightarrow\infty}\bar{\mathbb{E}}
\bigg[\int_{0}^{T}|\langle\bar{\bm{m}}_n(t)-\bm{m}^k(t),v\rangle_{\mathbb{L}^2}|^2dt\bigg]=0.
\end{align*}
Next, we show the proof of \eqref{tj002}. According to the Fubini theorem, we infer that
\begin{align*}
&\bar{\mathbb{E}}\bigg[\int_{0}^{T}|M_n(\bar{\bm{m}}_n,\bar{\bm{\psi}}_{jn},\bar{W}_n,\bar{\eta}_n,v)(t)
-M(\bm{m}^k,\bm{\psi}_j^k,W^k,\eta^k,v)(t)|^2dt\bigg]\\
&=\int_{0}^{T}\bar{\mathbb{E}}\big[|M_n(\bar{\bm{m}}_n,\bar{\bm{\psi}}_{jn},\bar{W}_n,\bar{\eta}_n,v)(t)
-M(\bm{m}^k,\bm{\psi}_j^k,W^k,\eta^k,v)(t)|^2\big]dt.
\end{align*}
Hence, it suffices to show that each term on the right-hand side of \eqref{M1} converges to each term on the right-hand side of \eqref{M2} in $L^2(\bar{\Omega};L^2(0,T;\mathbb{L}^2))$.
By applying the fact that $\bar{\bm{m}}_n\rightarrow\bm{m}^k$ in $\mathbb{D}([0,T];\mathbb{H}_w^1)$, $\bar{\mathbb{P}}$-a.s., and the continuity of $\bm{m}^k$ at $t=0$, we have
\begin{align*}
\lim\limits_{n\rightarrow\infty}\bar{\mathbb{E}}[|_{X^{-\beta}}\langle\bar{\bm{m}}_n(0)-\bm{m}^k(0),v\rangle_{X^{\beta}}|^2]
=\lim\limits_{n\rightarrow\infty}\bar{\mathbb{E}}[|\langle\bar{\bm{m}}_n(0)-\bm{m}^k(0),v\rangle_{\mathbb{L}^2}|^2]=0.
\end{align*}
Using the facts that $\bar{\bm{m}}_n\rightarrow\bm{m}^k$ in $L_w^2(0,T;\mathbb{H}^1)$, $\bar{\mathbb{P}}$-a.s.,
$\bar{\bm{s}}_n\rightarrow\bm{s}^k$ weak$^{\ast}$ in $L^{2r}(\bar{\Omega};L^{\infty}(\mathbb{R}^{+};\mathbb{L}^{s}))$, $1\leq s\leq3$,
$w'(\bar{\bm{m}}_n)\rightarrow w'(\bm{m}^k)$ weak$^{\ast}$ in $L^{2r}(\bar{\Omega};L^{\infty}(\mathbb{R}^{+};\mathbb{L}^{s}))$, $1\leq s\leq2$, and $\bar{\bm{H}}_{sn}\rightarrow\bm{H}_{s}^k$ in $L^{2r}(\bar{\Omega};L^2(0,T;\mathbb{L}^2(\mathbb{R}^3)))$ yield that
\begin{align*}
\lim\limits_{n\rightarrow\infty}\bar{\mathbb{E}}\bigg[\bigg|\int_{0}^{t}
\langle\bar{\bm{H}}_n-\bm{H}^k,v\rangle_{\mathbb{L}^2}ds\bigg|^2\bigg]=0.
\end{align*}
Furthermore, we have
\begin{align*}
&\bar{\mathbb{E}}\bigg[\bigg|\int_{0}^{t}\langle\bar{\bm{H}}_n,v\rangle_{\mathbb{L}^2}ds\bigg|^2\bigg]
\leq\bar{\mathbb{E}}\bigg[\int_{0}^{t}|\bar{\bm{H}}_n|_{\mathbb{L}^2}^2|v|_{\mathbb{L}^2}^2ds\bigg]\\
&\leq\bar{\mathbb{E}}\bigg[\int_{0}^{t}\big(|\Delta\bar{\bm{m}}_n|_{\mathbb{L}^2}^2+|w'(\bar{\bm{m}}_n)|_{\mathbb{L}^2}^2
+|\bar{\bm{H}}_{sn}|_{\mathbb{L}^2}^2+\frac{1}{2}|\bar{\bm{s}}_n|_{\mathbb{L}^2}^2\big)
|v|_{\mathbb{L}^2}^2ds\bigg]\\
&\leq\bar{\mathbb{E}}\big[\big(|\bar{\bm{m}}_n|_{L^4(0,T;\mathbb{H}^1)}^2+|w'(\bar{\bm{m}}_n)|_{L^4(0,T;\mathbb{L}^2)}^2
+|\bar{\bm{H}}_{sn}|_{L^4(0,T;\mathbb{L}^2)}^2+\frac{1}{2}|\bar{\bm{s}}_n|_{L^4(0,T;\mathbb{L}^2)}^2\big)
|v|_{L^4(0,T;\mathbb{L}^2)}^2\big]\\
&\leq C,
\end{align*}
thus we can conclude from the Vitali convergence theorem that
\begin{align*}
\lim\limits_{n\rightarrow\infty}\int_{0}^{T}\bar{\mathbb{E}}\bigg[\bigg|\int_{0}^{t}
\langle\bar{\bm{H}}_n-\bm{H}^k,v\rangle_{\mathbb{L}^2}ds\bigg|^2\bigg]dt=0.
\end{align*}
Applying \eqref{sjdyxtj}, we can infer that $(|\bar{\bm{m}}_n|^2-1)\bar{\bm{m}}_n\rightarrow(|\bm{m}^k|^2-1)\bm{m}^k$ weak$^{\ast}$ in $L^2(\bar{\Omega},L^{\infty}(0,T;\mathbb{L}^2))$. Then,
\begin{align*}
\lim\limits_{n\rightarrow\infty}\bar{\mathbb{E}}\bigg[\bigg|\int_{0}^{t}
\langle k(|\bar{\bm{m}}_n|^2-1)\bar{\bm{m}}_n-k(|\bm{m}^k|^2-1)\bm{m}^k,v\rangle_{\mathbb{L}^2}ds\bigg|^2\bigg]=0.
\end{align*}
Since
\begin{align*}
&\bar{\mathbb{E}}\bigg[\bigg|\int_{0}^{t}\langle(|\bar{\bm{m}}_n|^2-1)\bar{\bm{m}}_n,v\rangle_{\mathbb{L}^2}ds\bigg|^2\bigg]
\leq C\bar{\mathbb{E}}\big[|\bar{\bm{m}}_n|_{L^4(0,T;\mathbb{L}^4)}^2||\bar{\bm{m}}_n|^2-1|
_{L^{2}(0,T;\mathbb{L}^2)}^2|v|_{L^4(0,T;\mathbb{L}^4)}^2\big]\leq C,
\end{align*}
applying the Vitali convergence theorem, we derive that
\begin{align*}
\lim\limits_{n\rightarrow\infty}\int_{0}^{T}\bar{\mathbb{E}}\bigg[\bigg|\int_{0}^{t}
\langle k(|\bar{\bm{m}}_n|^2-1)\bar{\bm{m}}_n-k(|\bm{m}^k|^2-1)\bm{m}^k,v\rangle_{\mathbb{L}^2}ds\bigg|^2\bigg]dt=0.
\end{align*}
Next, employing Lemma \ref{ydsl}, we get
\begin{align*}
\lim\limits_{n\rightarrow\infty}\bar{\mathbb{E}}\bigg[\bigg|\int_{0}^{t}
\langle\bar{\bm{m}}_n\times\partial_s\bar{\bm{m}}_n-\bm{m}^k\times\partial_s\bm{m}^k,v\rangle_{\mathbb{L}^2}ds
\bigg|^2\bigg]=0,
\end{align*}
and then we make use of the H\"{o}lder inequality and Lemma \ref{SPMpci} to obtain
\begin{align*}
&\bar{\mathbb{E}}\bigg[\bigg|\int_{0}^{t}\langle\bar{\bm{m}}_n\times\partial_s\bar{\bm{m}}_n,v\rangle_{\mathbb{L}^2}ds\bigg|^2\bigg]
\leq C\bar{\mathbb{E}}\big[|\bar{\bm{m}}_n|_{L^4(0,T;\mathbb{L}^4)}^2|\partial_s\bar{\bm{m}}_n|
_{L^{2}(0,T;\mathbb{L}^2)}^2|v|_{L^4(0,T;\mathbb{L}^4)}^2\big]\leq C.
\end{align*}
Thus, the Vitali convergence theorem implies that
\begin{align*}
\lim\limits_{n\rightarrow\infty}\int_{0}^{T}\bar{\mathbb{E}}\bigg[\bigg|\int_{0}^{t}
\langle\bar{\bm{m}}_n\times\partial_s\bar{\bm{m}}_n-\bm{m}^k\times\partial_s\bm{m}^k,v\rangle_{\mathbb{L}^2}ds
\bigg|^2\bigg]dt=0.
\end{align*}
According to Lemma \ref{gdsl1}, for any $v\in L^4(\bar{\Omega};L^4(0,T;X^{\beta}))\subset L^4(\bar{\Omega};L^4(0,T;\mathbb{L}^2))$, and $t\in[0,T]$, we have
\begin{align*}
\lim\limits_{n\rightarrow\infty}\bar{\mathbb{E}}\bigg[\bigg|\int_{0}^{t}
\langle G'_n(\bar{\bm{m}}_n)[G_n(\bar{\bm{m}}_n)]
-G{^k}'(\bm{m}^k)[G^k(\bm{m}^k)],v\rangle_{\mathbb{L}^2}ds\bigg|^2\bigg]=0.
\end{align*}
Using Assumption \ref{ass1}, the H\"{o}lder inequality and Lemma \ref{SPMpci}, we get
\begin{align*}
&\bar{\mathbb{E}}\bigg[\bigg|\int_{0}^{t}
\langle G'_n(\bar{\bm{m}}_n)[G_n(\bar{\bm{m}}_n)]-G{^k}'(\bm{m}^k)[G^k(\bm{m}^k)],
v\rangle_{\mathbb{L}^2}ds\bigg|^2\bigg]\\
&\leq\bar{\mathbb{E}}\bigg[\int_{0}^{t}|G'_n(\bar{\bm{m}}_n)[G_n(\bar{\bm{m}}_n)]-G{^k}'(\bm{m}^k)[G^k(\bm{m}^k)]|
_{\mathbb{L}^2}^2|v|_{\mathbb{L}^2}^2ds\bigg]\\
&\leq\frac{C}{2}\bar{\mathbb{E}}\bigg[\int_{0}^{t}|v|_{\mathbb{L}^2}^4ds
+K_2^2\int_{0}^{t}|\bar{\bm{m}}_n-\bm{m}^k|_{\mathbb{L}^2}^4ds\bigg]\\
&\leq\frac{C}{2}\bar{\mathbb{E}}[|v|_{L^4(0,T;\mathbb{L}^2)}^4]
+\frac{CK_2^2}{2}\bar{\mathbb{E}}[|\bar{\bm{m}}_n|_{L^4(0,T;\mathbb{L}^2)}^4+|\bm{m}^k|_{L^4(0,T;\mathbb{L}^2)}^4]\\
&\leq C.
\end{align*}
Thus, by the Vitali convergence theorem, we find that
\begin{align*}
\lim\limits_{n\rightarrow\infty}\int_{0}^{T}\bar{\mathbb{E}}\bigg[\bigg|\int_{0}^{t}
\langle G'_n(\bar{\bm{m}}_n)[G_n(\bar{\bm{m}}_n)]-G{^k}'(\bm{m}^k)[G^k(\bm{m}^k)],
v\rangle_{\mathbb{L}^2}ds\bigg|^2\bigg]dt=0.
\end{align*}
According to Lemma \ref{gdsl1} and the It\^{o}-L\'{e}vy isometry, we know that for any $v\in L^4(\bar{\Omega};L^4(0,T;X^{\beta})$ and $t\in[0,T]$,
\begin{align*}
&\lim\limits_{n\rightarrow\infty}\bar{\mathbb{E}}\bigg[\bigg|\int_{0}^{t}
\langle G_n(\bar{\bm{m}}_n)-G^k(\bm{m}^k),v\rangle_{\mathbb{L}^2}dW^{k}(s)\bigg|^2\bigg]\\
&=\lim\limits_{n\rightarrow\infty}\bar{\mathbb{E}}\bigg[\bigg|\int_{0}^{t}\sum\limits_{i\geq1}
\langle G_n(\bar{\bm{m}}_n)-G^k(\bm{m}^k),v\rangle_{\mathbb{L}^2}\tilde{e}_idW_i^{k}(s)\bigg|^2\bigg]\\
&=\lim\limits_{n\rightarrow\infty}\bar{\mathbb{E}}\bigg[\int_{0}^{t}\bigg|\sum\limits_{i\geq1}
\langle G_{in}(\bar{\bm{m}}_{n})-G_i^k(\bm{m}^k),v\rangle_{\mathbb{L}^2}\bigg|^2ds\bigg]
=0.
\end{align*}
We can also apply the It\^{o}-L\'{e}vy isometry, Assumption \ref{ass1}, the H\"{o}lder inequality, and Lemma \ref{SPMpci} to see that
{\small\begin{align*}
&\bar{\mathbb{E}}\bigg[\bigg|\int_{0}^{t}
\langle G_n(\bar{\bm{m}}_n)-G^k(\bm{m}^k),v\rangle_{\mathbb{L}^2}dW^k(s)\bigg|^2\bigg]
=\bar{\mathbb{E}}\bigg[\bigg|\int_{0}^{t}\sum\limits_{i\geq1}
\langle G_n(\bar{\bm{m}}_n)-G^k(\bm{m}^k),v\rangle_{\mathbb{L}^2}\tilde{e}_idW_i^k(s)\bigg|^2\bigg]\\
&\leq\bar{\mathbb{E}}\bigg[\int_{0}^{t}\bigg|\sum\limits_{i\geq1}\langle G_n(\bar{\bm{m}}_n)-G^k(\bm{m}^k),v\rangle_{\mathbb{L}^2}\tilde{e}_i\bigg|^2ds\bigg]
\leq\bar{\mathbb{E}}\bigg[\int_{0}^{t}\bigg(\sum\limits_{i\geq1}|G_{in}(\bar{\bm{m}}_n)-G_i^k(\bm{m}^k)|_{\mathbb{L}^2}\bigg)^2
|v|_{\mathbb{L}^2}^2ds\bigg]\\
&\leq\frac{C}{2}\bar{\mathbb{E}}\bigg[\int_{0}^{t}|v|_{\mathbb{L}^2}^4ds
+K_2^2\int_{0}^{t}|\bar{\bm{m}}_n-\bm{m}^k|_{\mathbb{L}^2}^4ds\bigg]
\leq\frac{C}{2}\bar{\mathbb{E}}[|v|_{L^4(0,T;\mathbb{L}^2)}^4]
+\frac{CK_2^2}{2}\bar{\mathbb{E}}[|\bar{\bm{m}}_n|_{L^4(0,T;\mathbb{L}^2)}^4
+|\bm{m}^k|_{L^4(0,T;\mathbb{L}^2)}^4]\\
&\leq C.
\end{align*}}
Similarly, we have
\begin{align*}
\lim\limits_{n\rightarrow\infty}\int_{0}^{T}\bar{\mathbb{E}}\bigg[\bigg|\int_{0}^{t}
\langle G_n(\bar{\bm{m}}_n)-G^k(\bm{m}^k),v\rangle_{\mathbb{L}^2}dW^k(s)\bigg|^2\bigg]dt=0.
\end{align*}
According to Lemma \ref{gdsl1} and the It\^{o}-L\'{e}vy isometry, for any $v\in L^4(\bar{\Omega};L^4(0,T;X^{\beta}))$ and $t\in[0,T]$, we arrive at
\begin{align*}
&\lim\limits_{n\rightarrow\infty}\bar{\mathbb{E}}\bigg[\bigg|\int_{0}^{t}\int_B
\langle F_n(l,\bar{\bm{m}}_n)-F^k(l,\bm{m}^k),v\rangle_{\mathbb{L}^2}\tilde{\eta}^k(ds,dl)\bigg|^2\bigg]\\
&=\lim\limits_{n\rightarrow\infty}\bar{\mathbb{E}}\bigg[\int_{0}^{t}\int_B
\big|\langle F_n(l,\bar{\bm{m}}_n)-F^k(l,\bm{m}^k),v\rangle_{\mathbb{L}^2}\big|^2\mu(dl)ds\bigg]=0.
\end{align*}
We use the It\^{o}-L\'{e}vy isometry, Assumption \ref{ass1}, the H\"{o}lder inequality, as well as Lemma \ref{SPMpci} to
deduce that
\begin{align*}
&\bar{\mathbb{E}}\bigg[\bigg|\int_{0}^{t}\int_B
\langle F_n(l,\bar{\bm{m}}_n)-F^k(l,\bm{m}^k),v\rangle_{\mathbb{L}^2}\tilde{\eta}^k(ds,dl)\bigg|^2\bigg]\\
&=\bar{\mathbb{E}}\bigg[\int_{0}^{t}\int_B\big|\langle F_n(l,\bar{\bm{m}}_n)-F^k(l,\bm{m}^k),
v\rangle_{\mathbb{L}^2}\big|^2\mu(dl)ds\bigg]\\
&\leq\frac{C}{2}\bar{\mathbb{E}}\bigg[\int_{0}^{t}|v|_{\mathbb{L}^2}^4ds
+K_2^2\int_{0}^{t}|\bar{\bm{m}}_n-\bm{m}^k|_{\mathbb{L}^2}^4ds\bigg]\\
&\leq\frac{C}{2}\bar{\mathbb{E}}[|v|_{L^4(0,T;\mathbb{L}^2)}^4]
+\frac{CK_2^2}{2}\bar{\mathbb{E}}[|\bar{\bm{m}}_n|_{L^4(0,T;\mathbb{L}^2)}^4
+|\bm{m}^k|_{L^4(0,T;\mathbb{L}^2)}^4]\leq C.
\end{align*}
Now, by the Vitali convergence theorem, we observe that
\begin{align*}
\lim\limits_{n\rightarrow\infty}\int_{0}^{T}\bar{\mathbb{E}}\bigg[\bigg|\int_{0}^{t}
\langle F_n(l,\bar{\bm{m}}_n)-F^k(l,\bm{m}^k),v\rangle_{\mathbb{L}^2}\tilde{\eta}^k(ds,dl)\bigg|^2\bigg]dt=0.
\end{align*}
Hence the proof of \eqref{tj002} is complete.

Step 2: we prove that for all $\xi\in L^4(\bar{\Omega};L^4(0,T;X^{\beta}))$, the followings hold:
\begin{align}
&{\rm(a)}\ \lim\limits_{n\rightarrow\infty}\bar{\mathbb{E}}\bigg[\int_{0}^{T}|\langle\bar{\bm{\psi}}_{jn}(t)
-\bm{\psi}_j^k(t),\xi\rangle
_{\mathbb{L}^2}|^2dt\bigg]=0,\label{11tj001}\\
&{\rm(b)}\ \lim\limits_{n\rightarrow\infty}\bar{\mathbb{E}}\bigg
[\int_{0}^{T}|L_n(\bar{\bm{\psi}}_{jn},\bar{\bm{m}}_n,\xi)(t)-L(\bm{\psi}_j^k,\bm{m}^k,\xi)(t)|^2dt\bigg]=0.\label{11tj002}
\end{align}
To prove \eqref{11tj001},  we observe that, according to \eqref{sltj}, the fact that $\bar{\bm{\psi}}_{jn}\rightarrow\bm{\psi}_j^k$ in $L^4(0,T;\mathbb{L}^4)$, $\bar{\mathbb{P}}$-a.s., there holds that
\begin{align*}
&\int_{0}^{T}|\langle\bar{\bm{\psi}}_{jn}(t)-\bm{\psi}_j^k(t),\xi\rangle_{\mathbb{L}^2}|^2dt
\leq|\xi|_{L^4(0,T;\mathbb{L}^4)}^2|\bar{\bm{\psi}}_{jn}(t)-\bm{\psi}_j^k(t)|_{L^4(0,T;\mathbb{L}^4)}^2
\rightarrow 0,~\bar{\mathbb{P}}\text{-}a.s.,
\end{align*}
Moreover, for all $t\in[0,T]$, $r\in(1,2)$ and all $n\in\mathbb{N}$,
\begin{align*}
&\bar{\mathbb{E}}\bigg[\bigg|\int_{0}^{T}|\langle\bar{\bm{\psi}}_{jn}(t)-\bm{\psi}_j^k(t),
\xi\rangle_{\mathbb{L}^2}|^2dt\bigg|^r\bigg]
\leq C\bar{\mathbb{E}}\big[|\xi|_{L^4(0,T;\mathbb{L}^2)}^{2r}|\bar{\bm{\psi}}_{jn}(t)-\bm{\psi}_j^k(t)|
_{L^4(0,T;\mathbb{L}^2)}^{2r}\big]\\
&\leq C\big(\bar{\mathbb{E}}\big[|\xi|_{L^4(0,T;\mathbb{L}^2)}^{4}\big]\big)^{\frac{r}{2}}
\bigg(\bar{\mathbb{E}}\bigg[|\bar{\bm{\psi}}_{jn}(t)|_{L^4(0,T;\mathbb{L}^2)}^{\frac{4r}{2-r}}
+|\bm{\psi}_j^k(t)|_{L^4(0,T;\mathbb{L}^2)}^{\frac{4r}{2-r}}\bigg]\bigg)^{\frac{2-r}{2}}
\leq C,
\end{align*}
thus by the Vitali convergence theorem, we deduce that
\begin{align*}
\lim\limits_{n\rightarrow\infty}\bar{\mathbb{E}}
\bigg[\int_{0}^{T}|\langle\bar{\bm{\psi}}_{jn}(t)-\bm{\psi}_j^k(t),\xi\rangle_{\mathbb{L}^2}|^2dt\bigg]=0.
\end{align*}
Next, we intend to prove \eqref{11tj002}. Using the Fubini theorem, we can show that
\begin{align*}
&\bar{\mathbb{E}}\bigg[\int_{0}^{T}|L_n(\bar{\bm{\psi}}_{jn},\bar{\bm{m}}_n,\xi)(t)
-L(\bm{\psi}_j^k,\bm{m}^k,\xi)(t)|^2dt\bigg]\\
&=\int_{0}^{T}\bar{\mathbb{E}}\big[|L_n(\bar{\bm{\psi}}_{jn},\bar{\bm{m}}_n,\xi)(t)
-L(\bm{\psi}_j^k,\bm{m}^k,\xi)(t)|^2\big]dt.
\end{align*}
Then, it is enough to prove that each term on the right-hand side of \eqref{L1} converges to each term on the right-hand side of \eqref{L2} in $L^2(\bar{\Omega};L^2(0,T;\mathbb{L}^2))$.

Utilizing \eqref{sltj}, the fact that $\bar{\bm{\psi}}_{jn}\rightarrow\bm{\psi}_j^k$ in $\mathbb{D}([0,T];\mathbb{H}_w^1)$,
$\bar{\mathbb{P}}$-a.s., and the continuity of $\bm{\psi}_j^k$ at $t=0$, we infer that
\begin{align*}
\lim\limits_{n\rightarrow\infty}\bar{\mathbb{E}}[|_{X^{-\beta}}\langle\bar{\bm{\psi}}_{jn}(0)
-\bm{\psi}_j^k(0),\xi\rangle_{X^{\beta}}|^2]=0.
\end{align*}
Here, we can apply \eqref{sltj}, the fact that $\bar{\bm{\psi}}_{jn}\rightarrow\bm{\psi}_j^k$ in $L_w^2(0,T;\mathbb{H}^1)$, $\bar{\mathbb{P}}$-a.s. to find that for $\nabla\xi\in L^2(\bar{\Omega};L^2(0,T,\mathbb{L}^2))$,
\begin{align*}
\lim\limits_{n\rightarrow\infty}\bar{\mathbb{E}}\bigg[\bigg|\int_{0}^{t}
\langle\Delta\bar{\bm{\psi}}_{jn}(s)-\Delta\bm{\psi}_j^k(s),\xi\rangle_{\mathbb{L}^2}ds\bigg|^2\bigg]
=\lim\limits_{n\rightarrow\infty}\bar{\mathbb{E}}\bigg[\bigg|\int_{0}^{t}
\langle\nabla\bar{\bm{\psi}}_{jn}(s)-\nabla\bm{\psi}_j^k(s),\nabla\xi\rangle_{\mathbb{L}^2}ds\bigg|^2\bigg]
=0.
\end{align*}
In addition, since
\begin{align*}
\sup\limits_{n\in\mathbb{N}}\bar{\mathbb{E}}\bigg[\int_{0}^{t}|\Delta\bar{\bm{\psi}}_{jn}|_{\mathbb{L}^2}^2ds\bigg]
\leq C\sup\limits_{n\in\mathbb{N}}\bar{\mathbb{E}}\bigg[\int_{0}^{t}|\bar{\bm{\psi}}_{jn}|_{\mathbb{L}^2}^2ds\bigg]\leq C,
\end{align*}
we can use the Vitali convergence theorem to obtain
\begin{align*}
\lim\limits_{n\rightarrow\infty}\int_{0}^{T}\bar{\mathbb{E}}\bigg[\bigg|\int_{0}^{t}
\langle\Delta\bar{\bm{\psi}}_{jn}(s)-\Delta\bm{\psi}_j^k(s),\xi\rangle_{\mathbb{L}^2}ds
\bigg|^2\bigg]dt=0.
\end{align*}
It follows from \eqref{sltj}, the fact that $\bar{\bm{\psi}}_{jn}\rightarrow\bm{\psi}_j^k$ in $L^4(0,T;\mathbb{L}^4)$, $\bar{\mathbb{P}}$-a.s.,
and $\bar{V}_n\rightarrow V^k$ weak$^{\ast}$ in $L^2(\bar{\Omega};L^{\infty}(\mathbb{R}^{+};\mathbb{L}^{6}))$ that for $\xi\in L^4(\bar{\Omega};L^4(0,T,X^{\beta}))$,
\begin{align*}
&\lim\limits_{n\rightarrow\infty}\bar{\mathbb{E}}\bigg[\bigg|\int_{0}^{t}
\langle\bar{V}_n\bar{\bm{\psi}}_{jn}(s)-V^k\bm{\psi}_{j}^k(s),\xi\rangle_{\mathbb{L}^2}ds\bigg|^2\bigg]\\
&\leq\lim\limits_{n\rightarrow\infty}\bar{\mathbb{E}}\big[|\bar{\bm{\psi}}_{jn}-\bm{\psi}_{j}^k|_{L^2(0,T;\mathbb{L}^2)}^2
|\bar{V}_n|_{L^4(0,T;\mathbb{L}^6)}^2|\xi|_{L^4(0,T;\mathbb{L}^3)}^2\big]\\
&\quad +\lim\limits_{n\rightarrow\infty}\bar{\mathbb{E}}\bigg[\bigg|\int_{0}^{t}\int_{K}(\bar{V}_n-V^k)\bm{\psi}_{j}^k\xi
d\bm{x}ds\bigg|^2\bigg]=0.
\end{align*}
The H\"{o}lder inequality and Lemma \ref{SPMpci} imply that
\begin{align*}
\bar{\mathbb{E}}\bigg[\bigg|\int_{0}^{t}\langle\bar{V}_n\bar{\bm{\psi}}_{jn},\xi\rangle_{\mathbb{L}^2}ds\bigg|^2\bigg]
\leq C\bar{\mathbb{E}}\big[|\bar{\bm{\psi}}_{jn}|_{L^4(0,T;\mathbb{L}^4)}^2|\bar{V}_n|
_{L^2(0,T;\mathbb{L}^2)}^2|\xi|_{L^4(0,T;\mathbb{L}^4)}^2\big]\leq C,
\end{align*}
therefore, we have
\begin{align*}
\lim\limits_{n\rightarrow\infty}\int_{0}^{T}\bar{\mathbb{E}}\bigg[\bigg|\int_{0}^{t}
\langle\bar{V}_n\bar{\bm{\psi}}_{jn}(s)-V^k\bm{\psi}_{j}^k(s),\xi\rangle_{\mathbb{L}^2}ds\bigg|^2\bigg]dt=0.
\end{align*}
Employing \eqref{sltj}, the facts that $\bar{\bm{m}}_n\rightarrow\bm{m}^k$ in $L^4(0,T;\mathbb{L}^4)$, $\bar{\mathbb{P}}$-a.s., and
 $\bar{\bm{\psi}}_{jn}\rightarrow\bm{\psi}_j^k$ in $L^4(0,T;\mathbb{L}^4)$, $\bar{\mathbb{P}}$-a.s. yield that
\begin{align*}
&\lim\limits_{n\rightarrow\infty}\bar{\mathbb{E}}\bigg[\bigg|\int_{0}^{t}
\langle\bar{\bm{m}}_n\cdot\hat{\bm{\sigma}}\bar{\bm{\psi}}_{jn}(s)
-\bm{m}^k\cdot\hat{\bm{\sigma}}\bm{\psi}_{j}^k(s),\xi\rangle_{\mathbb{L}^2}ds\bigg|^2\bigg]\\
&\leq\lim\limits_{n\rightarrow\infty}\bar{\mathbb{E}}\big[|\bar{\bm{\psi}}_{jn}-\bm{\psi}_{j}^k|_{L^4(0,T;\mathbb{L}^4)}^2
|\bar{\bm{m}}_n|_{L^2(0,T;\mathbb{L}^2)}^2
|\xi|_{L^4(0,T;\mathbb{L}^4)}^2\big]\\
&+\lim\limits_{n\rightarrow\infty}\bar{\mathbb{E}}\big[|\bar{\bm{m}}_n-\bm{m}^k|_{L^4(0,T;\mathbb{L}^4)}^2
|\bm{\psi}_{j}^k|_{L^2(0,T;\mathbb{L}^2)}^2
|\xi|_{L^4(0,T;\mathbb{L}^4)}^2\big]=0.
\end{align*}
Utilizing the H\"{o}lder inequality and Lemma \ref{SPMpci}, we have
\begin{align*}
\bar{\mathbb{E}}\bigg[\bigg|\int_{0}^{t}\langle\bar{\bm{m}}_n\cdot\hat{\bm{\sigma}}\bar{\bm{\psi}}_{jn},
\xi\rangle_{\mathbb{L}^2}ds\bigg|^2\bigg]
\leq C\bar{\mathbb{E}}\big[|\bar{\bm{\psi}}_{jn}|_{L^4(0,T;\mathbb{L}^4)}^2|\bar{\bm{m}}_n|
_{L^2(0,T;\mathbb{L}^2)}^2|\xi|_{L^4(0,T;\mathbb{L}^4)}^2\big]\leq C,
\end{align*}
thus, the Vitali convergence theorem leads to
\begin{align*}
\lim\limits_{n\rightarrow\infty}\int_{0}^{T}\bar{\mathbb{E}}\bigg[\bigg|\int_{0}^{t}
\langle\bar{\bm{m}}_n\cdot\hat{\bm{\sigma}}\bar{\bm{\psi}}_{jn}(s)
-\bm{m}^k\cdot\hat{\bm{\sigma}}\bm{\psi}_{j}^k(s),\xi\rangle_{\mathbb{L}^2}ds
\bigg|^2\bigg]dt=0.
\end{align*}
Hence, the proof of \eqref{11tj002} is completed.

Step 3: Since $(\bm{m}_n,\bm{\psi}_{jn})$ is a Galerkin approximation solution, for all $t\in[0,T]$, we have
\begin{align*}
\alpha{_{X^{-\beta}}}\langle\bm{m}_n(t),v\rangle_{X^{\beta}}=M_n(\bm{m}_n,\bm{\psi}_{jn},W_n,\eta_n,v)(t),~\mathbb{P}\text{-}a.s.,
\end{align*}
in particular,
\begin{align*}
\int_{0}^{T}\mathbb{E}\big[\big|\alpha{_{X^{-\beta}}}\langle\bm{m}_n(t),v\rangle_{X^{\beta}}
-M_n(\bm{m}_n,\bm{\psi}_{jn},W_n,\eta_n,v)(t)\big|^2\big]dt=0.
\end{align*}
From $\mathcal{L}(\bm{m}_n,\bm{\psi}_{jn},W_n,\eta_n)=\mathcal{L}(\bar{\bm{m}}_n,\bar{\bm{\psi}}_{jn},\bar{W}_n,\bar{\eta}_n)$, we can infer that
\begin{align*}
\int_{0}^{T}\bar{\mathbb{E}}\big[\big|\alpha{_{X^{-\beta}}}\langle\bar{\bm{m}}_n(t),v\rangle_{X^{\beta}}
-M_n(\bar{\bm{m}}_n,\bar{\bm{\psi}}_{jn},\bar{W}_n,\bar{\eta}_n,v)(t)\big|^2\big]dt=0.
\end{align*}
Therefore, from \eqref{tj001} and \eqref{tj002}, we obtain that
\begin{align*}
\int_{0}^{T}\bar{\mathbb{E}}\big[\big|\alpha{_{X^{-\beta}}}\langle\bm{m}^k(t),v\rangle_{X^{\beta}}
-M(\bm{m}^k,\bm{\psi}_j^k,W^k,\eta^k,v)(t)\big|^2\big]dt=0.
\end{align*}
Thus, for m-almost all $t\in[0,T]$ and $\bar{\mathbb{P}}$-almost all $\omega\in\bar{\Omega}$, it holds that
\begin{align}\label{MM2}
&\alpha{_{X^{-\beta}}}\langle\bm{m}^k(t),v\rangle_{X^{\beta}}-\alpha_{X^{-\beta}}\langle\bm{m}^k(0),v\rangle_{X^{\beta}}
+\int_{0}^{t}\langle\bm{m}^k(s)\times\partial_s\bm{m}^k(s),v\rangle_{\mathbb{L}^2}ds\nonumber\\
&~-\int_{0}^{t}\langle\bm{H}^k(s),v\rangle ds
+\int_{0}^{t}\langle k(|\bm{m}^k(s)|^2-1)\bm{m}^k(s),v\rangle ds
+\int_{0}^{t}\langle G^k(\bm{m}^k(s)),v\rangle dW^k(s)\\
&~+\frac{1}{2}\int_{0}^{t}\langle G{^k}'(\bm{m}^k(s))[G^k(\bm{m}^k(s))],v\rangle ds
+\int_{0}^{t}\int_B\langle F^k(\bm{m}^k(s-),l),v\rangle\tilde{\eta}^k(ds,dl)=0.\nonumber
\end{align}
Note that $\bm{m}^k$ is a $\mathscr{Z}_T$-valued random variable, $\bm{m}^k\in\mathbb{D}([0,T];\mathbb{H}_{w}^1)$ a.e., then
$\bm{m}^k$ is weakly c\`{a}dl\`{a}g with respect to time. Thus, the function on the left-hand side of \eqref{MM2} is c\`{a}dl\`{a}g in time. Therefore, for all $t\in[0,T]$ and $v\in L^4(\bar{\Omega};L^4(0,T;X^{\beta}))$, \eqref{MM2} holds.

Step 4: since $(\bm{m}_n,\bm{\psi}_{jn})$ is the Galerkin approximation solution, then for all $t\in[0,T]$, we have
\begin{align*}
i{_{X^{-\beta}}}\langle\bm{\psi}_{jn}(t),\xi\rangle_{X^{\beta}}=L_n(\bm{\psi}_{jn},\bm{m}_n,\xi)(t),~\mathbb{P}\text{-}a.s.
\end{align*}
In particular,
\begin{align*}
\int_{0}^{T}\mathbb{E}\big[\big|i{_{X^{-\beta}}}\langle\bm{\psi}_{jn}(t),\xi\rangle_{X^{\beta}}
-L_n(\bm{\psi}_{jn},\bm{m}_n,\xi)(t)\big|^2\big]dt=0.
\end{align*}
We further obtain using $\mathcal{L}(\bm{\psi}_{jn},\bm{m}_n)=\mathcal{L}(\bar{\bm{\psi}}_{jn},\bar{\bm{m}}_n)$ that
\begin{align*}
\int_{0}^{T}\bar{\mathbb{E}}\big[\big|i{_{X^{-\beta}}}\langle\bar{\bm{\psi}}_{jn}(t),\xi\rangle_{X^{\beta}}
-L_n(\bar{\bm{\psi}}_{jn},\bar{\bm{m}}_n,\xi)(t)\big|^2\big]dt=0.
\end{align*}
Note that, from \eqref{11tj001} and \eqref{11tj002}, we get
\begin{align*}
\int_{0}^{T}\bar{\mathbb{E}}\big[\big|\alpha{_{X^{-\beta}}}\langle\bm{\psi}_j^k(t),\xi\rangle_{X^{\beta}}
-L(\bm{\psi}_j^k,\bm{m}^k,\xi)(t)\big|^2\big]dt=0.
\end{align*}
Thus, for m-almost all $t\in[0,T]$ and $\bar{\mathbb{P}}$-almost all $\omega\in\bar{\Omega}$, we obtain
\begin{align}\label{LL2}
&i_{X^{-\beta}}\langle\bm{\psi}_j^k(t),\xi\rangle_{X^{\beta}}
-i_{X^{-\beta}}\langle\bm{\psi}_j^k(0),\xi\rangle_{X^{\beta}}
+\frac{1}{2}\int_{0}^{t}\langle\Delta\bm{\psi}_j^k(s),\xi\rangle_{\mathbb{L}^2}ds\nonumber\\
&-\int_{0}^{t}\langle V^k\bm{\psi}_j^k,\xi\rangle ds
+\frac{1}{2}\int_{0}^{t}\langle\bm{m}^k\cdot\hat{\bm{\sigma}}\bm{\psi}_j^k,\xi\rangle ds=0.
\end{align}
Since $\bm{\psi}_j^k$ is weakly c\`{a}dl\`{a}g respect to time, we deduce that the function on the left-hand side of \eqref{LL2} is c\`{a}dl\`{a}g in time. Therefore, for all $t\in[0,T]$ and $\xi\in L^4(\bar{\Omega};L^4(0,T;X^{\beta}))$, \eqref{LL2} holds. In conclusion, for all $t\in[0,T]$, $\bar{\mathbb{P}}$-a.s., and each $v\in L^4(\bar{\Omega};L^4(0,T;X^{\beta}))$ and each $\xi\in L^4(\bar{\Omega};L^4(0,T;X^{\beta}))$, $(\Omega, \mathscr{F}, (\mathscr{F}_t)_{t\geq0}, \mathbb{P}, \bm{\Psi}^k,\bm{m}^k, V^k, \rho^k,\bm{s}^k,\bm{H}_s^k,W^k,\eta^k)$ is a weak martingale solution of the system \eqref{SLLG3}.

By \eqref{sjdyxtj} and the weak lower semi-continuity of norms, we have
\begin{align*}
|\nabla\bm{\Psi}^k|_{L^{2r}(\bar{\Omega};L^{\infty}(\mathbb{R}^{+},\mathcal{L}_{\lambda}^{2}))}
\leq&\lim_{n\rightarrow\infty}\inf|\nabla\bar{\bm{\Psi}}_{n}|_{L^{2r}(\bar{\Omega};L^{\infty}(\mathbb{R}^{+},\mathcal{L}_{\lambda}^{2}))}
\leq C,\\
|\nabla V^k|_{L^{2r}(\bar{\Omega};L^{\infty}(\mathbb{R}^{+},\mathbb{L}^{2}))}
\leq&\lim_{n\rightarrow\infty}\inf|\nabla\bar{V}_{n}|_{L^{2r}(\bar{\Omega};L^{\infty}(\mathbb{R}^{+},\mathbb{L}^{2}))}
\leq C,\\
|\partial_t\bm{m}^k|_{L^{2r}(\bar{\Omega};L^2(\mathbb{R}^{+},\mathbb{L}^{2}))}
\leq&\lim_{n\rightarrow\infty}\inf|\partial_t\bm{m}_n|_{L^{2r}(\bar{\Omega};L^2(\mathbb{R}^{+},\mathbb{L}^{2}))}
\leq C,\\
|\bm{m}^k|_{L^{2r}(\bar{\Omega};L^{\infty}(\mathbb{R}^{+},\mathbb{H}^{1}))}
\leq&\lim_{n\rightarrow\infty}\inf|\nabla\bar{\bm{m}}_{n}|_{L^{2r}(\bar{\Omega};L^{\infty}(\mathbb{R}^{+},\mathbb{H}^{1}))}
\leq C,\\
|\bm{H}_{s}^k|_{L^{2r}(\bar{\Omega};L^{\infty}(\mathbb{R}^{+},\mathbb{L}^{2}))}
\leq&\lim_{n\rightarrow\infty}\inf|\bar{\bm{H}}_{sn}|_{L^{2r}(\bar{\Omega};L^{\infty}(\mathbb{R}^{+},\mathbb{L}^{2}))}
\leq C,\\
|w'(\bm{m}^k)|_{L^{2r}(\bar{\Omega};L^{\infty}(\mathbb{R}^{+},\mathbb{L}^{2}))}
\leq&\lim_{n\rightarrow\infty}\inf|w'(\bar{\bm{m}}_{n})|_{L^{2r}(\bar{\Omega};L^{\infty}(\mathbb{R}^{+},\mathbb{L}^{2}))}
\leq C,\\
||\bm{m}^k|^2-1|_{L^{2r}(\bar{\Omega};L^{\infty}(\mathbb{R}^{+},\mathbb{L}^{2}))}
\leq&\lim_{n\rightarrow\infty}\inf||\bar{\bm{m}}_n|^2-1|_{L^{2r}(\bar{\Omega};L^{\infty}(\mathbb{R}^{+},\mathbb{L}^{2}))}
\leq C,\\
|G^k(\bm{m}^k)|_{L^{2r}(\bar{\Omega};L^{\infty}(\mathbb{R}^{+},\mathbb{L}^{2}))}
\leq&\lim_{n\rightarrow\infty}\inf|G_n(\bar{\bm{m}}_n)|_{L^{2r}(\bar{\Omega};L^{\infty}(\mathbb{R}^{+},\mathbb{L}^{2}))}
\leq C
\end{align*}
which conclude \eqref{kdgj}.
\end{proof}
\section{Tightness of the second-layer approximation sequence}
Define the probability measure $\mathcal{L}(\bm{m}^k,\bm{\psi}_j^k,W^k,\eta^k)
=\mathcal{L}(\bm{m}^k)\times\mathcal{L}(\bm{\psi}_j^k)\times\mathcal{L}(W^k)\times\mathcal{L}(\eta^k)$,
where
$\mathcal{L}(\bm{m}^k)$ is the law of $\bm{m}^k$ on $\mathscr{Z}_{T}$,
$\mathcal{L}(\bm{\psi}_j^k)$ is the law of $\bm{\psi}_j^k$ on $\mathscr{Z}_{T}$,
$\mathcal{L}(W^k)$ is the law of $W^k$ on $\mathscr{Z}_{W^k,T}$,
$\mathcal{L}(\eta^k)$ is the law of $\eta^k$ on $\mathscr{Z}_{\eta^k,T}$.
\begin{lemma}\label{taijinxing07}
The sequence $\{\mathcal{L}(\bm{m}^k),k>0\}$ is tight on the space $(\mathscr{Z}_{T},\mathcal{J})$.
\end{lemma}
\begin{proof}
According to \eqref{kdgj}, we can conclude that $\sup\limits_{k>0}\mathbb{E}\big[|\bm{m}^k|_{L^{\infty}(0,T;\mathbb{H}^1)}^2\big]\leq C$.
Next, we prove that the sequence satisfies the Aldous condition in $X^{-\beta}$.
Let $\{\tau_k\}_{k>0}$ be a stopping time sequence such that $0\leq\tau_k+\theta\leq T$. By Lemma \ref{dycsl}, we find that
\begin{align*}
\alpha\bm{m}^k(t)=&\alpha\bm{m}^k(0)-\int_{0}^{t}\left(\bm{m}^k(s)\times\partial_s\bm{m}^k(s)
-\bm{H}^k(s)+k(|\bm{m}^k(s)|^2-1)\bm{m}^k(s)\right)ds\\
&-\frac{1}{2}\int_{0}^{t}G{^k}'(\bm{m}^k(s))[G^k(\bm{m}^k(s))]ds
-\int_{0}^{t}G^k(\bm{m}^k(s))dW(s)\\
&-\int_{0}^{t}\int_BF^k(\bm{m}^k(s-),l)\tilde{\eta}(ds,dl)\\
=:&\sum\limits_{i=0}^{6}J_{k}^{i}(t),~\mathbb{P}\text{-}a.s., \text{for all}~t\in[0,T].
\end{align*}
Let $\theta>0$. In order to prove that $\bm{m}^k$ satisfies the Aldous condition,
it is sufficient to prove that each term $J_{k}^{i}, i=0,1,\cdots,6$ satisfies the same condition.
Obviously, $J_{k}^{0}$ satisfies \eqref{taijinxing2.1}. Using the embedding $\mathbb{L}^{\frac{3}{2}}\hookrightarrow X^{-\beta}$,
the H\"{o}lder inequality and \eqref{kdgj}, we deduce that
\begin{align*}
&\mathbb{E}\big[|J_k^{1}(\tau_k+\theta)-J_k^{1}(\tau_k)|_{X^{-\beta}}\big]
=\mathbb{E}\bigg[\bigg|\int_{\tau_k}^{\tau_k+\theta}\bm{m}^k\times\partial_s\bm{m}^kds\bigg|_{X^{-\beta}}\bigg]\\
&\leq\mathbb{E}\bigg[\int_{\tau_k}^{\tau_k+\theta}|\bm{m}^k\times\partial_s\bm{m}^k|_{\mathbb{L}^{3/2}}ds\bigg]
\leq\mathbb{E}\bigg[\int_{\tau_k}^{\tau_k+\theta}|\bm{m}^k|_{\mathbb{L}^6}|\partial_s\bm{m}^k|_{\mathbb{L}^2}ds\bigg]\\
&\leq C\theta^{\frac{1}{2}}\mathbb{E}\big[|\bm{m}^k|_{L^{\infty}(0,T;\mathbb{H}^1)}
|\partial_s\bm{m}^k|_{L^{2}(0,T;\mathbb{L}^2)}\big]
\leq C\theta^{\frac{1}{2}}.
\end{align*}
Then extending the domain $D$, taking $\beta=\frac{1}{2}$, using Definition \ref{def}, \eqref{kdgj} and the extension theorem,
we obtain that
\begin{align*}
&\mathbb{E}\bigg[\int_{\tau_k}^{\tau_k+\theta}\big|\Delta\bm{m}^k\big|_{X^{-\beta}(D)}ds\bigg]
=\mathbb{E}\bigg[\int_{\tau_k}^{\tau_k+\theta}\big|A_1^{-\beta}(\Delta\bm{m}^k)\big|_{\mathbb{L}^2(D)}ds\bigg]\\
&=\mathbb{E}\bigg[\int_{\tau_k}^{\tau_k+\theta}\big|(I-\Delta)^{-\frac{1}{2}}\Delta\bm{m}^k\big|_{\mathbb{L}^2(D)}ds\bigg]
\leq\mathbb{E}\bigg[\int_{\tau_k}^{\tau_k+\theta}\big|(-\Delta)^{-\frac{1}{2}}\Delta\bm{m}^k\big|_{\mathbb{L}^2(\mathbb{R}^3)}ds\bigg]\\
&\leq\mathbb{E}\bigg[\int_{\tau_k}^{\tau_k+\theta}\big|\nabla^{-1}\Delta\bm{m}^k\big|_{\mathbb{L}^2(\mathbb{R}^3)}ds\bigg]
\leq\mathbb{E}\bigg[\int_{\tau_k}^{\tau_k+\theta}\big|\nabla\bm{m}^k\big|_{\mathbb{L}^2(\mathbb{R}^3)}ds\bigg]\\
&\leq C\theta^{\frac{1}{2}}\mathbb{E}\big[|\bm{m}^k|_{L^2(0,T;\mathbb{H}^1)}\big]
\leq C\theta^{\frac{1}{2}}.
\end{align*}
According to the embedding $\mathbb{L}^2\hookrightarrow X^{-\beta}$, the H\"{o}lder inequality
and \eqref{kdgj}, it holds that
\begin{align*}
&\mathbb{E}\bigg[\int_{\tau_k}^{\tau_k+\theta}\big|w'(\bm{m}^k)+\bm{H}_{s}^k
+\frac{1}{2}\bm{s}^k\big|_{X^{-\beta}}ds\bigg]\\
&\leq C\mathbb{E}\bigg[\int_{\tau_k}^{\tau_k+\theta}|w'(\bm{m}^k)|_{\mathbb{L}^2}
+|\bm{H}_{s}^k|_{\mathbb{L}^2}+\frac{1}{2}|\bm{s}^k|_{\mathbb{L}^2}ds\bigg]\\
&\leq C\theta^{\frac{1}{2}}\mathbb{E}\big[
|w'(\bm{m}^k)|_{L^{\infty}(0,T;\mathbb{L}^2)}
+|\bm{H}_{s}^k|_{L^{\infty}(0,T;\mathbb{L}^2)}
+\frac{1}{2}|\bm{s}^k|_{L^{\infty}(0,T;\mathbb{L}^2)}\big]\\
&\leq C\theta^{\frac{1}{2}}.
\end{align*}
Thus, we have
\begin{align*}
&\mathbb{E}\big[|J_{k}^{2}(\tau_k+\theta)-J_{k}^{2}(\tau_k)|_{X^{-\beta}}\big]
=\mathbb{E}\bigg[\bigg|\int_{\tau_k}^{\tau_k+\theta}\bm{H}^kds\bigg|_{X^{-\beta}}\bigg]\\
&=\mathbb{E}\bigg[\bigg|\int_{\tau_k}^{\tau_k+\theta}\big(\Delta\bm{m}^k+w'(\bm{m}^k)+\bm{H}_{s}^k
+\frac{1}{2}\bm{s}^k\big)ds\bigg|_{X^{-\beta}}\bigg]\\
&\leq\mathbb{E}\bigg[\int_{\tau_k}^{\tau_k+\theta}\big|\Delta\bm{m}^k\big|_{X^{-\beta}}ds\bigg]
+\mathbb{E}\bigg[\int_{\tau_k}^{\tau_k+\theta}\big|w'(\bm{m}^k)+\bm{H}_{s}^k
+\frac{1}{2}\bm{s}^k\big|_{X^{-\beta}}ds\bigg]\\
&\leq C\theta^{\frac{1}{2}}.
\end{align*}
By the embedding $\mathbb{L}^{3/2}\hookrightarrow X^{-\beta}$, we conclude that
\begin{align*}
&\mathbb{E}\big[|J_{k}^{3}(\tau_k+\theta)-J_{k}^{3}(\tau_k)|_{X^{-\beta}}\big]
=\mathbb{E}\bigg[\bigg|\int_{\tau_k}^{\tau_k+\theta}k(|\bm{m}^k|^2-1)\bm{m}^kds\bigg|_{X^{-\beta}}\bigg]\\
&\leq\mathbb{E}\bigg[\int_{\tau_k}^{\tau_k+\theta}k|(|\bm{m}^k|^2-1)\bm{m}^k|_{\mathbb{L}^{3/2}}ds\bigg]
\leq C\mathbb{E}\bigg[\int_{\tau_k}^{\tau_k+\theta}k||\bm{m}^k|^2-1|_{\mathbb{L}^2}|\bm{m}^k|_{\mathbb{L}^6}ds\bigg]\\
&\leq Ck\theta^{\frac{1}{2}}\mathbb{E}\big[||\bm{m}^k|^2-1|_{L^{\infty}(0,T;\mathbb{L}^2)}
|\bm{m}^k|_{L^{\infty}(0,T;\mathbb{H}^1)}\big]
\leq C\theta^{\frac{1}{2}},
\end{align*}
Similarly,
\begin{align*}
&\bigg(\mathbb{E}\big[|J_{k}^{4}(\tau_k+\theta)-J_{k}^{4}(\tau_k)|_{X^{-\beta}}\big]\bigg)^2
\leq\mathbb{E}\bigg[\bigg|\int_{\tau_k}^{\tau_k+\theta}G{^k}'(\bm{m}^k)[G^k(\bm{m}^k)]ds\bigg|_{X^{-\beta}}^2\bigg]\\
&\leq C\mathbb{E}\bigg[\int_{\tau_k}^{\tau_k+\theta}|G{^k}'(\bm{m}^k)[G^k(\bm{m}^k)]|_{\mathbb{L}^2}^2ds\bigg]
\leq C\theta\mathbb{E}\big[1+|\bm{m}^k|_{L^{\infty}(0,T;\mathbb{L}^2)}^2\big]
\leq C\theta.
\end{align*}
Thanks to Assumption \ref{ass1}, the It\^{o}-L\'{e}vy isometry and the H\"{o}lder inequality, we have
\begin{align*}
&\bigg(\mathbb{E}\big[|J_{k}^{5}(\tau_k+\theta)-J_{k}^{5}(\tau_k)|_{X^{-\beta}}\big]\bigg)^2
\leq\mathbb{E}\bigg[\bigg|\int_{\tau_k}^{\tau_k+\theta}G^k(\bm{m}^k)dW(s)\bigg|_{X^{-\beta}}^2\bigg]\\
&\leq\mathbb{E}\bigg[\bigg|\int_{\tau_k}^{\tau_k+\theta}G^k(\bm{m}^k)dW(s)\bigg|_{\mathbb{L}^2}^2\bigg]
=\mathbb{E}\bigg[\bigg|\int_{\tau_n}^{\tau_n+\theta}\sum\limits_{i\geq1}G^k(\bm{m}^k)\tilde{e}_idW_i(s)\bigg|_{\mathbb{L}^2}^2\bigg]\\
&\leq C\mathbb{E}\bigg[\int_{\tau_k}^{\tau_k+\theta}\bigg|\sum\limits_{i\geq1}G_i^k(\bm{m}^k)\bigg|_{\mathbb{L}^2}^2ds\bigg]
\leq C\theta\mathbb{E}\big[1+|\bm{m}^k|_{L^{\infty}(0,T;\mathbb{L}^2)}^2\big]
\leq C\theta.
\end{align*}
Similarly, we get
\begin{align*}
&\bigg(\mathbb{E}\big[|J_{k}^{6}(\tau_k+\theta)-J_{k}^{6}(\tau_k)|_{X^{-\beta}}\big]\bigg)^2
\leq\mathbb{E}\bigg[\bigg|\int_{\tau_k}^{\tau_k+\theta}\int_BF^k(l,\bm{m}^k)\tilde{\eta}(ds,dl)\bigg|_{X^{-\beta}}^2\bigg]\\
&\leq\mathbb{E}\bigg[\bigg|\int_{\tau_k}^{\tau_k+\theta}\int_BF^k(l,\bm{m}^k)\tilde{\eta}(ds,dl)\bigg|_{\mathbb{L}^2}^2\bigg]
\leq C\mathbb{E}\bigg[\int_{\tau_k}^{\tau_k+\theta}\int_B|F^k(l,\bm{m}^k)|_{\mathbb{L}^2}^2\mu(dl)ds\bigg]\\
&\leq C\theta\mathbb{E}\big[1+|\bm{m}^k|_{L^{\infty}(0,T;\mathbb{L}^2)}^2\big]
\leq C\theta.
\end{align*}
Thus, when $\alpha=1$, $\gamma=\frac{1}{2}$, the terms $J_{k}^{0}$-$J_{k}^{6}$ satisfy \eqref{taijinxing2.1}, which yields the desired result of Lemma \ref{taijinxing07}.
\end{proof}
\begin{lemma}\label{taijinxing08}
The sequence $\{\mathcal{L}(\bm{\psi}_{j}^k),k>0\}$ is tight on the space $(\mathscr{Z}_{T},\mathcal{J})$.
\end{lemma}
\begin{proof}
According to \eqref{kdgj}, we know that $\sup\limits_{k>0}\mathbb{E}\big[|\bm{\psi}_j^k|_{L^{\infty}(0,T;\mathbb{H}^1)}^2\big]\leq C$.
Next, we will prove that the sequence satisfies Aldous condition in $X^{-\beta}$. To this end,
let $\{\tau_k\}_{k>0}$ be a stopping time sequence such that $0\leq\tau_k+\theta\leq T$, then we gain by Lemma \ref{dycsl} that
\begin{align*}
i\bm{\psi}_j^k(t)
&=i\bm{\psi}_j^k(0)-\frac{1}{2}\int_{0}^{t}\big(\Delta\bm{\psi}_j^k(s)-V^k(s)\bm{\psi}_j^k(s)
+\frac{1}{2}\bm{m}^k(s)\cdot\hat{\bm{\sigma}}\bm{\psi}_j^k(s)\big)ds\\
&=:\sum\limits_{i=0}^{3}I_{k}^{i}(t),~\mathbb{P}\text{-}a.s., \text{for all}~t\in[0,T].
\end{align*}
Let $\theta>0$, in order to prove that $\bm{\psi}_j^k$ satisfies the Aldous condition,
it is sufficient to prove that each term $I_{k}^{i}, i=0,1,\cdots,3$ satisfies the same condition.
Obviously, $I_{k}^{0}$ satisfies \eqref{taijinxing2.1}. Then extending the domain $K$ and setting $\beta=\frac{1}{2}$,
using Definition  \ref{def}, \eqref{kdgj} and the extension theorem, we deduce that
\begin{align*}
&\bigg(\mathbb{E}\big[|I_{k}^{1}(\tau_k+\theta)-I_{k}^{1}(\tau_k)|_{X^{-\beta}(K)}\big]\bigg)^2
=\mathbb{E}\bigg[\bigg|\int_{\tau_k}^{\tau_k+\theta}\Delta\bm{\psi}_j^kds\bigg|_{X^{-\beta}(K)}^2\bigg]\\
&\leq\mathbb{E}\bigg[\int_{\tau_k}^{\tau_k+\theta}|\nabla\bm{\psi}_j^k|_{\mathbb{L}^2(\mathbb{R}^3)}^2ds\bigg]
\leq C\theta\mathbb{E}\big[|\bm{\psi}_j^k|_{L^{\infty}(0,T;\mathbb{H}^1)}^2\big]
\leq C\theta.
\end{align*}
Hence, $\mathbb{E}[|I_{k}^{1}(\tau_k+\theta)-I_{k}^{1}(\tau_k)|_{X^{-\beta}}]\leq C\theta^{\frac{1}{2}}$.
By the embedding $\mathbb{L}^2\hookrightarrow X^{-\beta}$, the H\"{o}lder inequality and \eqref{kdgj}, we obtain
\begin{align*}
&\mathbb{E}\big[|I_{k}^{2}(\tau_k+\theta)-I_{k}^{2}(\tau_k)|_{X^{-\beta}}\big]
=\mathbb{E}\bigg[\bigg|\int_{\tau_k}^{\tau_k+\theta}V^k\bm{\psi}_j^kds\bigg|_{X^{-\beta}}\bigg]\\
&\leq\mathbb{E}\bigg[\int_{\tau_k}^{\tau_k+\theta}|V^k\bm{\psi}_j^k|_{\mathbb{L}^2}ds\bigg]
\leq C\theta^{\frac{1}{2}}\mathbb{E}\big[|V^k|_{L^2(0,T;\mathbb{L}^6)}
|\bm{\psi}_j^k|_{L^{\infty}(0,T;\mathbb{H}^1)}\big]
\leq C\theta^{\frac{1}{2}}.
\end{align*}
\begin{align*}
&\mathbb{E}\big[|I_{k}^{3}(\tau_k+\theta)-I_{k}^{3}(\tau_k)|_{X^{-\beta}}\big]
=\mathbb{E}\bigg[\bigg|\int_{\tau_k}^{\tau_k+\theta}\bm{m}^k\cdot\hat{\bm{\sigma}}\bm{\psi}_j^kds\bigg|_{X^{-\beta}}\bigg]\\
&\leq\mathbb{E}\bigg[\int_{\tau_k}^{\tau_k+\theta}|\bm{m}^k\bm{\psi}_j^k|_{\mathbb{L}^2}ds\bigg]
\leq C\theta^{\frac{1}{2}}\mathbb{E}\big[|\bm{m}^k|_{L^{\infty}(0,T;\mathbb{H}^1)}
|\bm{\psi}_j^k|_{L^{\infty}(0,T;\mathbb{H}^1)}\big]
\leq C\theta^{\frac{1}{2}}.
\end{align*}
Thus, when $\alpha=1$, $\gamma=\frac{1}{2}$, the terms $I_{k}^{0}$-$I_{k}^{3}$ satisfy \eqref{taijinxing2.1}. we complete the proof of Lemma \ref{taijinxing08}.
\end{proof}
\section{Existence of weak martingale solutions to the SPLLG system \eqref{S2}}
\subsection{Construction of new probability space and processes as $k\rightarrow\infty$}
By Lemma \ref{taijinxing07}, we know that the sequence $\{\mathcal{L}(\bm{m}^k),k>0\}$ is tight on $(\mathscr{Z}_{T},\mathcal{J})$.
By Lemma \ref{taijinxing08}, we know that the sequence $\{\mathcal{L}(\bm{\psi}_j^k),k>0\}$ is tight on $(\mathscr{Z}_{T},\mathcal{J})$. Let $W^k:=W$, then $\mathcal{L}(W^k)$ is tight on $\mathscr{Z}_{W^k,T}$. Let $\eta^k:=\eta$, then $\mathcal{L}(\eta^k)$ is tight on $\mathscr{Z}_{\eta^k,T}$. Therefore, $\{\mathcal{L}(\bm{m}^k, \bm{\psi}_j^k, W^k, \eta^k), k>0\}$ is tight on $\mathscr{Z}_T\times\mathscr{Z}_T\times\mathscr{Z}_{W^k,T}\times\mathscr{Z}_{\eta^k,T}$, by using the result in\cite{BM1, M}.

By Theorem \ref{2SKO}, we know that there exists a subsequence $\{k_i\}_{i\in\mathbb{N}}$, a filtered probability space $(\bar{\Omega},\bar{\mathscr{F}},(\bar{\mathscr{F}}_t)_{t\geq 0},\bar{\mathbb{P}})$ and, on this space, $\mathscr{Z}_T\times\mathscr{Z}_T\times\mathscr{Z}_{W^k,T}\times\mathscr{Z}_{\eta^k,T}$-valued random variable $(\bar{\bm{m}},\bar{\bm{\psi}}_j,\bar{W},\bar{\eta})$,
$(\bar{\bm{m}}^k,\bar{\bm{\psi}}_j^k,\bar{W}^k,\bar{\eta}^k)_{k>0}$ such that:\\
{\rm(1)} $\mathcal{L}((\bar{\bm{m}}^k, \bar{\bm{\psi}}_j^k,\bar{W}^k, \bar{\eta}^k))
=\mathcal{L}((\bm{m}^{k_i}, \bm{\psi}_j^{k_i},W^{k_i}, \eta^{k_i}))$, for all $i\in\mathbb{N}$;\\
{\rm(2)} $(\bar{\bm{m}}^k,\bar{\bm{\psi}}_j^k,\bar{W}^k, \bar{\eta}^k)\rightarrow
(\bar{\bm{m}}, \bar{\bm{\psi}}_j, \bar{W}, \bar{\eta})$ in $\mathscr{Z}_T\times\mathscr{Z}_T\times\mathscr{Z}_{W^k,T}\times\mathscr{Z}_{\eta^k,T}$ as $k\rightarrow\infty$, for $\bar{\mathbb{P}}$-a.s. $\bar{\omega}$;\\
{\rm(3)} $(\bar{W}^k(\bar{\omega}),\bar{\eta}^k(\bar{\omega}))=(\bar{W}(\bar{\omega}),\bar{\eta}(\bar{\omega}))$,
for all $\bar{\omega}\in\bar{\Omega}$.

For convenience, we still use $(\bm{m}^k, \bm{\psi}_j^k, W^k, \eta^k)_{k>0}$ and $(\bar{\bm{m}}^k, \bar{\bm{\psi}}_j^k, \bar{W}^k, \bar{\eta}^k)_{k>0}$ to represent these sequences. In addition, $\bar{\eta}^k$ and $\bar{\eta}$ are time-homogeneous Poisson random measures on $(B, \mathscr{B}(B))$ with intensity $\mu$. For $p\in[2,\infty)$, $q\in[2,6)$, and $\beta>\frac{1}{4}$, we get
\begin{align}\label{11sltj}
\bar{\bm{m}}^k\rightarrow\bar{\bm{m}}~\text{in}~\mathscr{Z}_T,~\bar{\mathbb{P}}\text{-}a.s., \text{ and }
\bar{\bm{\psi}}_j^k\rightarrow\bar{\bm{\psi}}_j~\text{in}~\mathscr{Z}_T,~\bar{\mathbb{P}}\text{-}a.s.
\end{align}
Next, we still consider the specific case of $p=4$, $q=4$, and $\beta=\frac{1}{2}$, Notice that $\mathbb{P}\{\bm{\psi}_j^k\in\mathscr{Z}_T\}=1$ and $\mathbb{P}\{\bm{m}^k\in\mathscr{Z}_T\}=1$, then the laws of $(\bm{m}^k,\bm{\psi}_j^k)$ and $(\bar{\bm{m}}^k,\bar{\bm{\psi}}_j^k)$ are equal on $\mathscr{Z}_T$.
\subsection{Properties of the constructed new process and limiting process}
Since the laws of $(\bm{m}^k,\bm{\psi}_j^k)$ and $(\bar{\bm{m}}^k,\bar{\bm{\psi}}_j^k)$ are equal, $(\bar{\bm{m}}^k, \bar{\bm{\psi}}_j^k)$ satisfies the same estimates of the sequence $(\bm{m}^k, \bm{\psi}_j^k)$, that is, they satisfy Lemma \ref{P}--Lemma \ref{SPMpci}. Furthermore, according to \eqref{kdgj}, we have $$\bar{\mathbb{E}}\bigg[\sup_{0\leq t\leq T}\int_{D}(|\bar{\bm{m}}^k|^2-1)^2d\bm{x}\bigg]\leq\frac{C}{k},$$ so as $k\rightarrow\infty$,
\begin{align*}
\bar{\bm{m}}^k\rightarrow\bar{\bm{m}}~p.a.e., \text{ and }~|\bar{\bm{m}}|=1,~\mathbb{P}\text{-}a.s.,
\end{align*}
where p.a.e. is the abbreviation for pointwise almost everywhere. Then using the weak${^\ast}$ compactness, as $k\rightarrow\infty$, we get
\begin{align}\label{11sjdyxtj}
&\partial_t\bar{\bm{\Psi}}^k\rightarrow\partial_t\bar{\bm{\Psi}}~\text{weak}^{\ast}~\text{in}~
L^{2r}(\bar{\Omega};L^{\infty}(\mathbb{R}^{+};\mathcal{H}_{\lambda}^{-1})),~
\partial_t\bar{\bm{m}}^k\rightarrow\partial_t\bar{\bm{m}}~\text{weakly in}~L^{2r}(\bar{\Omega};L^{2}(\mathbb{R}^{+};\mathbb{L}^{2})),\nonumber\\
&G^k(\bar{\bm{m}}^k)\rightarrow G(\bar{\bm{m}})~\text{weak}^{\ast}~
\text{in}~L^{2r}(\bar{\Omega};L^{\infty}(\mathbb{R}^{+};\mathbb{H}^{1})),~
\bar{\bm{H}}_s^k\rightarrow\bar{\bm{H}}_{s}~\text{in}~L^{2r}(\bar{\Omega};L^2(0,T;\mathbb{L}^2(\mathbb{R}^3))),\nonumber\\
&\bar{V}^k\rightarrow \bar{V}~\text{weak}^{\ast}~\text{in}~L^{2r}(\bar{\Omega};L^{\infty}(\mathbb{R}^{+};\mathbb{L}^{6})),~
\nabla\bar{V}^k\rightarrow\nabla \bar{V}~\text{weak}^{\ast}~\text{in}~L^{2r}(\bar{\Omega};L^{\infty}(\mathbb{R}^{+};\mathbb{L}^{2})),\nonumber\\
&\bar{\bm{s}}^k\rightarrow\bar{\bm{s}}~\text{weak}^{\ast}~\text{in}~L^{2r}(\bar{\Omega};L^{\infty}(\mathbb{R}^{+};\mathbb{L}^{s})),~
\bar{\rho}^k\rightarrow\bar{\rho}~\text{weak}^{\ast}~\text{in}~
L^{2r}(\bar{\Omega};L^{\infty}(\mathbb{R}^{+};\mathbb{L}^{s})),~ 1\leq s\leq3,\nonumber\\
&w'(\bar{\bm{m}}^k)\rightarrow w'(\bar{\bm{m}})~\text{weak}^{\ast}~
\text{in}~L^{2r}(\bar{\Omega};L^{\infty}(\mathbb{R}^{+};\mathbb{L}^{s})),~ 1\leq s\leq2,\nonumber\\
&|\bar{\bm{m}}^k|^2-1\rightarrow0~\text{weakly and a.e. in}~L^{2r}(\bar{\Omega};L^2(0,T;\mathbb{L}^{2})).
\end{align}
In summary, $\bar{\bm{m}}^k$ and $\bar{\bm{m}}$ still satisfy Lemma \ref{ydsl} and Lemma \ref{gdsl1}, and we will omit the proof.

\subsection{Convergence of the new process to the corresponding limit process} Now, we are in position to show
Theorem \ref{them1} and Theorem \ref{youjieyudejie}.
For $\chi\in L^4(\bar{\Omega};L^4(0,T;\mathbb{H}^1))$, and for all $t\in[0,T]$, we have
\begin{align*}
&\alpha\langle\bar{\bm{m}}^k(t),\chi\rangle=\alpha\langle\bar{\bm{m}}^k(0),\chi\rangle
-\int_{0}^{t}\langle\bar{\bm{m}}^k(s)\times\partial_s\bar{\bm{m}}^k(s),\chi\rangle ds
+\int_{0}^{t}\langle\bar{\bm{H}}^k(s),\chi\rangle ds\\
&\quad -\int_{0}^{t}\langle k(|\bar{\bm{m}}^k(s)|^2-1)\bar{\bm{m}}^k(s),\chi\rangle ds
-\frac{1}{2}\int_{0}^{t}\langle {G^k}'(\bar{\bm{m}}^k(s))[G^k(\bar{\bm{m}}^k(s))],\chi\rangle ds\\
&\quad -\int_{0}^{t}\langle G^k(\bar{\bm{m}}^k(s)),\chi\rangle d\bar{W}^k(s)
-\int_{0}^{t}\int_B\langle F^k(\bar{\bm{m}}^k(s-),l),\chi\rangle\tilde{\bar{\eta}}^k(ds,dl),~~\bar{\mathbb{P}}\text{-}a.s.
\end{align*}
Therefore,  letting $\xi\in L^8(\bar{\Omega};C^{\infty}([0,T]\times D))$ and $\chi=\bar{\bm{m}}^k\times\xi$, we obtain
\begin{align*}
&\int_{0}^{t}\int_{D}\alpha\partial_s\bar{\bm{m}}^k(\bar{\bm{m}}^k\times\xi)d\bm{x}ds
=-\int_{0}^{t}\int_{D}(\bar{\bm{m}}^k\times\partial_s\bar{\bm{m}}^k)(\bar{\bm{m}}^k\times\xi)d\bm{x}ds\\
&\quad -\int_{0}^{t}\int_{D}(-\bar{\bm{H}}_s^k-\frac{1}{2}\bar{\bm{s}}^k
+w'(\bar{\bm{m}}^k))(\bar{\bm{m}}^k\times\xi)d\bm{x}ds
-\int_{0}^{t}\int_{D}\nabla\bar{\bm{m}}^k\nabla(\bar{\bm{m}}^k\times\xi)d\bm{x}ds\\
&\quad -\int_{0}^{t}\int_{D}k(|\bar{\bm{m}}^k|^2-1)\bar{\bm{m}}^k(\bar{\bm{m}}^k\times\xi)d\bm{x}ds
-\int_{0}^{t}\int_{D}G^k(\bar{\bm{m}}^k)(\bar{\bm{m}}^k\times\xi)d\bm{x}d\bar{W}^k(s)\\
&\quad -\frac{1}{2}\int_{0}^{t}\int_{D}G{^k}'(\bar{\bm{m}}^k)[G^k(\bar{\bm{m}}^k)](\bar{\bm{m}}^k\times\xi)d\bm{x}ds\\
&\quad -\int_{0}^{t}\int_B\int_{D}F^k(\bar{\bm{m}}^k,l)(\bar{\bm{m}}^k\times\xi)d\bm{x}\tilde{\bar{\eta}}^k(ds,dl).
\end{align*}
From $\langle a\times b,a\times c\rangle=-\langle b,\langle a,c\rangle a-|a|^2c\rangle$, we have
\begin{align*}
-\int_{0}^{t}\int_{D}(\bar{\bm{m}}^k\times\partial_s\bar{\bm{m}}^k)(\bar{\bm{m}}^k\times\xi)d\bm{x}ds
=\int_{0}^{t}\int_{D}\partial_s\bar{\bm{m}}^k\bar{\bm{m}}^k\bar{\bm{m}}^k\xi d\bm{x}ds
-\int_{0}^{t}\int_{D}\partial_s\bar{\bm{m}}^k|\bar{\bm{m}}^k|^2\xi d\bm{x}ds.
\end{align*}
In virtue of $\langle a,a\times b\rangle=0$, we obtain
\begin{align*}
&-\int_{0}^{t}\int_{D}\nabla\bar{\bm{m}}^k\nabla(\bar{\bm{m}}^k\times\xi)d\bm{x}ds
=\int_{0}^{t}\int_{D}(\bar{\bm{m}}^k\times\nabla\bar{\bm{m}}^k)\nabla\xi d\bm{x}ds,\\
&-\int_{0}^{t}\int_{D}k(|\bar{\bm{m}}^k|^2-1)\bar{\bm{m}}^k(\bar{\bm{m}}^k\times\xi)d\bm{x}ds=0.
\end{align*}
Using the fact that $\langle a,b\times c\rangle=-\langle c,b\times a\rangle$, we derive
\begin{align*}
&\alpha\int_{0}^{t}\int_{D}\partial_s\bar{\bm{m}}^k(\bar{\bm{m}}^k\times\xi)d\bm{x}ds
=-\alpha\int_{0}^{t}\int_{D}(\bar{\bm{m}}^k\times\partial_s\bar{\bm{m}}^k)\xi d\bm{x}ds.\\
&-\int_{0}^{t}\int_{D}(-\bar{\bm{H}}_s^k-\frac{1}{2}\bar{\bm{s}}^k+w'(\bar{\bm{m}}^k))(\bar{\bm{m}}^k\times\xi)d\bm{x}ds
=-\int_{0}^{t}\int_{D}\bar{\bm{m}}^k\times(\bar{\bm{H}}_s^k+\frac{1}{2}\bar{\bm{s}}^k-w'(\bar{\bm{m}}^k))\xi d\bm{x}ds.\\
&-\frac{1}{2}\int_{0}^{t}\int_{D}(G{^k}'(\bar{\bm{m}}^k)[G^k(\bar{\bm{m}}^k)])(\bar{\bm{m}}^k\times\xi)d\bm{x}ds
=\frac{1}{2}\int_{0}^{t}\int_{D}(\bar{\bm{m}}^k\times (G{^k}'(\bar{\bm{m}}^k)[G^k(\bar{\bm{m}}^k)]))\xi d\bm{x}ds.\\
&-\int_{0}^{t}\int_{D}G^k(\bar{\bm{m}}^k)(\bar{\bm{m}}^k\times\xi)d\bm{x}d\bar{W}^k(s)
=\int_{0}^{t}\int_{D}(\bar{\bm{m}}^k\times G^k(\bar{\bm{m}}^k))\xi d\bm{x}d\bar{W}^k(s).\\
&-\int_{0}^{t}\int_B\int_{D}F^k(\bar{\bm{m}}^k,l)(\bar{\bm{m}}^k\times\xi)d\bm{x}\tilde{\bar{\eta}}^k(ds,dl)
=\int_{0}^{t}\int_B\int_{D}(\bar{\bm{m}}^k\times F^k(\bar{\bm{m}}^k,l))\xi d\bm{x}\tilde{\bar{\eta}}^k(ds,dl).
\end{align*}
Then, it is sufficient to show the following identity:
\begin{align*}
&-\int_{0}^{t}\int_{D}\alpha(\bar{\bm{m}}^k\times\partial_s\bar{\bm{m}}^k)\xi d\bm{x}ds
+\int_{0}^{t}\int_{D}|\bar{\bm{m}}^k|^2\partial_s\bar{\bm{m}}^k\xi d\bm{x}ds
-\int_{0}^{t}\int_{D}(\partial_s\bar{\bm{m}}^k\cdot\bar{\bm{m}}^k)\bar{\bm{m}}^k\xi d\bm{x}ds\nonumber\\
&=-\int_{0}^{t}\int_{D}\bar{\bm{m}}^k\times(\bar{\bm{H}}_s^k+\frac{1}{2}\bar{\bm{s}}^k-w'(\bar{\bm{m}}^k))\xi d\bm{x}ds
+\int_{0}^{t}\int_{D}(\bar{\bm{m}}^k\times\nabla\bar{\bm{m}}^k)\nabla\xi d\bm{x}ds\nonumber\\
&\quad +\frac{1}{2}\int_{0}^{t}\int_{D}(\bar{\bm{m}}^k\times (G{^k}'(\bar{\bm{m}}^k)[G^k(\bar{\bm{m}}^k)]))\xi d\bm{x}ds
+\int_{0}^{t}\int_{D}(\bar{\bm{m}}^k\times G^k(\bar{\bm{m}}^k))\xi d\bm{x}d\bar{W}^k(s)\nonumber\\
&\quad +\int_{0}^{t}\int_B\int_{D}(\bar{\bm{m}}^k\times F^k(\bar{\bm{m}}^k,l))\xi d\bm{x}\tilde{\bar{\eta}}^k(ds,dl).
\end{align*}
Next, we are going to get the convergence of each term in the above equation.
By the facts that $\bar{\bm{m}}^k\rightarrow\bar{\bm{m}}$ in $L^4(0,T;\mathbb{L}^4)$, $\bar{\mathbb{P}}$-a.s.,
and $\partial_t\bar{\bm{m}}^k\rightarrow\partial_t\bar{\bm{m}}$ weakly in $L^{2r}(\bar{\Omega};L^{2}(\mathbb{R}^{+};\mathbb{L}^{2}))$,
we can infer that
\begin{align*}
&\lim\limits_{k\rightarrow\infty}\bar{\mathbb{E}}\bigg[\bigg|\int_{0}^{t}\int_{D}(\bar{\bm{m}}^k\times\partial_s\bar{\bm{m}}^k
-\bar{\bm{m}}\times\partial_s\bar{\bm{m}})\xi d\bm{x}ds\bigg|^2\bigg]\\
&\leq\lim\limits_{k\rightarrow\infty}\bar{\mathbb{E}}\bigg[\bigg(\int_{0}^{t}\int_{D}|(\bar{\bm{m}}^k
-\bar{\bm{m}})(\partial_s\bar{\bm{m}}^k\times\xi)|d\bm{x}ds\bigg)^2\bigg]\\
&\quad +\lim\limits_{k\rightarrow\infty}\bar{\mathbb{E}}\bigg[\bigg|\int_{0}^{t}\int_{D}(\partial_s\bar{\bm{m}}^k
-\partial_s\bar{\bm{m}})(\bar{\bm{m}}\times\xi)d\bm{x}ds\bigg|^2\bigg]\\
&\leq\lim\limits_{k\rightarrow\infty}\bar{\mathbb{E}}\big[|\bar{\bm{m}}^k-\bar{\bm{m}}|_{L^4(0,T;\mathbb{L}^4)}^2
|\partial_s\bar{\bm{m}}^k|_{L^2(0,T;\mathbb{L}^2)}^2|\xi|_{L^4(0,T;\mathbb{L}^4)}^2\big]\\
&\quad +\lim\limits_{k\rightarrow\infty}\bar{\mathbb{E}}\bigg[\bigg|\int_{0}^{t}\int_{D}(\partial_s\bar{\bm{m}}^k
-\partial_s\bar{\bm{m}})(\bar{\bm{m}}\times\xi)d\bm{x}ds\bigg|^2\bigg]
=0.
\end{align*}
Applying the H\"{o}lder inequality and \eqref{kdgj}, we conclude that
\begin{align*}
&\bar{\mathbb{E}}\bigg[\bigg|\int_{0}^{t}\langle\bar{\bm{m}}^k\times\partial_s\bar{\bm{m}}^k,\xi\rangle_{\mathbb{L}^2}ds\bigg|^2\bigg]
\leq C\bar{\mathbb{E}}\bigg[\bigg(\int_{0}^{t}\int_D|(\bar{\bm{m}}^k\times\xi)\partial_s\bar{\bm{m}}^k|d\bm{x}ds\bigg)^2\bigg]\\
&\leq C\bar{\mathbb{E}}\big[|\bar{\bm{m}}^k|_{L^4(0,T;\mathbb{L}^4)}^2
|\partial_s\bar{\bm{m}}^k|_{L^2(0,T;\mathbb{L}^2)}^2|\xi|_{L^4(0,T;\mathbb{L}^4)}^2\big]
\leq C,
\end{align*}
thus, we use the Vitali convergence theorem to deduce that
\begin{align*}
\lim\limits_{k\rightarrow\infty}\int_{0}^{T}\bar{\mathbb{E}}\bigg[\bigg|\int_{0}^{t}
\langle\bar{\bm{m}}^k\times\partial_s\bar{\bm{m}}^k
-\bar{\bm{m}}\times\partial_s\bar{\bm{m}},\xi\rangle_{\mathbb{L}^2}ds\bigg|\bigg]dt=0.
\end{align*}
Since
\begin{align*}
\int_{0}^{t}\int_{D}|\bar{\bm{m}}^k|^2\partial_s\bar{\bm{m}}^k\xi d\bm{x}ds
=\int_{0}^{t}\int_{D}(|\bar{\bm{m}}^k|^2-1)\partial_s\bar{\bm{m}}^k\xi d\bm{x}ds
+\int_{0}^{t}\int_{D}\partial_s\bar{\bm{m}}^k\xi d\bm{x}ds,
\end{align*}
and $|\bar{\bm{m}}^k|^2-1\rightarrow0$ a.e. in $L^{2r}(\bar{\Omega};L^2(0,T;\mathbb{L}^{2}))$, we know that
\begin{align*}
&\lim\limits_{k\rightarrow\infty}\bar{\mathbb{E}}\bigg[\bigg|\int_{0}^{t}\int_{D}(|\bar{\bm{m}}^k|^2-1)
\partial_s\bar{\bm{m}}^k\xi d\bm{x}ds\bigg|^2\bigg]
\leq\lim\limits_{k\rightarrow\infty}\bar{\mathbb{E}}\bigg[\bigg(\int_{0}^{t}\int_{D}|(|\bar{\bm{m}}^k|^2-1)
\partial_s\bar{\bm{m}}^k\xi|d\bm{x}ds\bigg)^2\bigg]\\
&\leq\lim\limits_{k\rightarrow\infty}\bar{\mathbb{E}}\big[||\bar{\bm{m}}^k|^2-1|_{L^2(0,T;\mathbb{L}^2)}^2
|\partial_s\bar{\bm{m}}^k\xi|_{L^2(0,T;\mathbb{L}^2)}^2\big]=0.
\end{align*}
Using the fact that $\partial_t\bar{\bm{m}}^k\rightarrow\partial_t\bar{\bm{m}}$ weakly in
$L^{2r}(\bar{\Omega};L^{2}(\mathbb{R}^{+};\mathbb{L}^{2}))$, we obtain
\begin{align*}
\lim\limits_{k\rightarrow\infty}\bar{\mathbb{E}}\bigg[\bigg|\int_{0}^{t}\int_{D}(\partial_s\bar{\bm{m}}^k
-\partial_s\bar{\bm{m}})\xi d\bm{x}ds\bigg|^2\bigg]=0.
\end{align*}
Hence,
\begin{align*}
\lim\limits_{k\rightarrow\infty}\bar{\mathbb{E}}\bigg[\bigg|\int_{0}^{t}\int_{D}\big(|\bar{\bm{m}}^k|^2
\partial_s\bar{\bm{m}}^k-\partial_s\bar{\bm{m}}\big)\xi d\bm{x}ds\bigg|^2\bigg]
=0.
\end{align*}
By the H\"{o}lder inequality and \eqref{kdgj}, we have
\begin{align*}
&\bar{\mathbb{E}}\bigg[\bigg|\int_{0}^{t}\langle|\bar{\bm{m}}^k|^2\partial_s\bar{\bm{m}}^k,\xi\rangle_{\mathbb{L}^2}ds\bigg|^2\bigg]
\leq C\bar{\mathbb{E}}\bigg[\bigg(\int_{0}^{t}|\bar{\bm{m}}^k|_{\mathbb{L}^4}^2
|\partial_s\bar{\bm{m}}^k\xi|_{\mathbb{L}^2}ds\bigg)^2\bigg]\\
&\leq C\bar{\mathbb{E}}\big[|\bar{\bm{m}}^k|_{L^4(0,T;\mathbb{L}^4)}^4
|\partial_s\bar{\bm{m}}^k\xi|_{L^2(0,T;\mathbb{L}^2)}^2\big]
\leq C,
\end{align*}
therefore, applying the Vitali convergence theorem gives
\begin{align*}
\lim\limits_{k\rightarrow\infty}\int_{0}^{T}\bar{\mathbb{E}}\bigg[\bigg|\int_{0}^{t}
\langle|\bar{\bm{m}}^k|^2\partial_s\bar{\bm{m}}^k-\partial_s\bar{\bm{m}},\xi\rangle_{\mathbb{L}^2}ds\bigg|^2\bigg]dt=0.
\end{align*}
Using \eqref{11sltj}, the fact that $\bar{\bm{m}}^k\rightarrow\bar{\bm{m}}$ in $L^4(0,T;\mathbb{L}^4)$, $\bar{\mathbb{P}}$-a.s.,
and $\partial_t\bar{\bm{m}}^k\rightharpoonup\partial_t\bar{\bm{m}}$ in $L^{2r}(\bar{\Omega};L^{2}(\mathbb{R}^{+};\mathbb{L}^{2}))$,
we obtain
\begin{align*}
&\lim\limits_{k\rightarrow\infty}\bar{\mathbb{E}}\bigg[\bigg|\int_{0}^{t}
\int_{D}(\partial_s\bar{\bm{m}}^k\cdot\bar{\bm{m}}^k)\bar{\bm{m}}^k\cdot\xi
-(\partial_s\bar{\bm{m}}\cdot\bar{\bm{m}})\bar{\bm{m}}\cdot\xi d\bm{x}ds\bigg|^2\bigg]\\
&\leq\lim\limits_{k\rightarrow\infty}\bar{\mathbb{E}}\bigg[\bigg|\int_{0}^{t}
\int_{D}(\bar{\bm{m}}^k-\bar{\bm{m}})\partial_s\bar{\bm{m}}^k\cdot\bar{\bm{m}}^k\xi d\bm{x}ds\bigg|^2\bigg]\\
&\quad +\lim\limits_{k\rightarrow\infty}\bar{\mathbb{E}}\bigg[\bigg|\int_{0}^{t}
\int_{D}(\partial_s\bar{\bm{m}}^k-\partial_s\bar{\bm{m}})\bar{\bm{m}}\cdot\bar{\bm{m}}^k\xi d\bm{x}ds\bigg|^2\bigg]\\
&\quad +\lim\limits_{k\rightarrow\infty}\bar{\mathbb{E}}\bigg[\bigg|\int_{0}^{t}
\int_{D}(\bar{\bm{m}}^k-\bar{\bm{m}})\partial_s\bar{\bm{m}}\cdot\bar{\bm{m}}\xi d\bm{x}ds\bigg|^2\bigg]\\
&\leq\lim\limits_{k\rightarrow\infty}\bar{\mathbb{E}}\big[|\bar{\bm{m}}^k-\bar{\bm{m}}|_{L^4(0,T;\mathbb{L}^4)}^2
|\partial_s\bar{\bm{m}}^k\xi|_{L^2(0,T;\mathbb{L}^2)}^2|\bar{\bm{m}}^k|_{L^4(0,T;\mathbb{L}^4)}^2\big]\\
&\quad +\lim\limits_{k\rightarrow\infty}\bar{\mathbb{E}}\big[|\bar{\bm{m}}^k-\bar{\bm{m}}|_{L^4(0,T;\mathbb{L}^4)}^2
|\partial_s\bar{\bm{m}}\xi|_{L^2(0,T;\mathbb{L}^2)}^2|\bar{\bm{m}}|_{L^4(0,T;\mathbb{L}^4)}^2\big]\\
&\quad +\lim\limits_{k\rightarrow\infty}\bar{\mathbb{E}}\bigg[\bigg|\int_{0}^{t}
\int_{D}(\partial_s\bar{\bm{m}}^k-\partial_s\bar{\bm{m}})\bar{\bm{m}}\cdot\bar{\bm{m}}^k\xi d\bm{x}ds\bigg|^2\bigg]
=0.
\end{align*}
We use the H\"{o}lder inequality and \eqref{kdgj} to obtain
\begin{align*}
&\bar{\mathbb{E}}\bigg[\bigg|\int_{0}^{t}\langle(\partial_s\bar{\bm{m}}^k\cdot\bar{\bm{m}}^k)\bar{\bm{m}}^k,
\xi\rangle_{\mathbb{L}^2}ds\bigg|^2\bigg]
\leq C\bar{\mathbb{E}}\bigg[\bigg(\int_{0}^{t}|\bar{\bm{m}}^k|_{\mathbb{L}^4}^2
|\partial_s\bar{\bm{m}}^k\xi|_{\mathbb{L}^2}ds\bigg)^2\bigg]\\
&\leq C\bar{\mathbb{E}}\big[|\bar{\bm{m}}^k|_{L^4(0,T;\mathbb{L}^4)}^4|\partial_s\bar{\bm{m}}^k\xi|_{L^2(0,T;\mathbb{L}^2)}^2\big]
\leq C.
\end{align*}
Thus,
\begin{align*}
\lim\limits_{k\rightarrow\infty}\int_{0}^{T}\bar{\mathbb{E}}\bigg[\bigg|\int_{0}^{t}
\langle|\bar{\bm{m}}^k|^2\partial_s\bar{\bm{m}}^k-\partial_s\bar{\bm{m}},\xi\rangle_{\mathbb{L}^2}ds\bigg|^2\bigg]dt=0.
\end{align*}
Recalling \eqref{11sltj}, the facts that $\bar{\bm{m}}^k\rightarrow\bar{\bm{m}}$ in $L^4(0,T;\mathbb{L}^4)$, $\bar{\mathbb{P}}$-a.s.,
$\bar{\bm{H}}_s^k\rightarrow\bar{\bm{H}}_{s}$ in $L^{2r}(\bar{\Omega};L^2(0,T;\mathbb{L}^2(\mathbb{R}^3)))$,
$\bar{\bm{s}}^k\rightarrow\bar{\bm{s}}$ weak$^{\ast}$ in $L^{2r}(\bar{\Omega};L^{\infty}(\mathbb{R}^{+};\mathbb{L}^{s})), 1\leq s\leq3$,
$w'(\bar{\bm{m}}^k)\rightarrow w'(\bar{\bm{m}})$ weak$^{\ast}$ in $L^{2r}(\bar{\Omega};L^{\infty}(\mathbb{R}^{+};\mathbb{L}^{s})), 1\leq s\leq2$, we can see
\begin{align*}
&\lim\limits_{k\rightarrow\infty}\bar{\mathbb{E}}\bigg[\bigg|\int_{0}^{t}\int_{D}\bar{\bm{m}}^k\times(\bar{\bm{H}}_s^k
+\frac{1}{2}\bar{\bm{s}}^k-w'(\bar{\bm{m}}^k))\xi-\bar{\bm{m}}\times(\bar{\bm{H}}_s
+\frac{1}{2}\bar{\bm{s}}-w'(\bar{\bm{m}}))\xi d\bm{x}ds\bigg|^2\bigg]\\
&\leq\lim\limits_{k\rightarrow\infty}\bar{\mathbb{E}}\bigg[\bigg|\int_{0}^{t}
\int_{D}(\bar{\bm{H}}_s^k-\bar{\bm{H}}_s)\cdot\bar{\bm{m}}^k\times\xi d\bm{x}ds\bigg|^2\bigg]
+\lim\limits_{k\rightarrow\infty}\bar{\mathbb{E}}\bigg[\bigg|\int_{0}^{t}
\int_{D}(\bar{\bm{m}}^k-\bar{\bm{m}})\cdot\bar{\bm{H}}_s\times\xi d\bm{x}ds\bigg|^2\bigg]\\
&\quad +\frac{1}{2}\lim\limits_{k\rightarrow\infty}\bar{\mathbb{E}}\bigg[\bigg|\int_{0}^{t}
\int_{D}(\bar{\bm{s}}^k-\bar{\bm{s}})\cdot\bar{\bm{m}}^k\times\xi d\bm{x}ds\bigg|^2\bigg]
+\frac{1}{2}\lim\limits_{k\rightarrow\infty}\bar{\mathbb{E}}\bigg[\bigg|\int_{0}^{t}
\int_{D}(\bar{\bm{m}}^k-\bar{\bm{m}})\cdot\bar{\bm{s}}\times\xi d\bm{x}ds\bigg|^2\bigg]\\
&\quad +\lim\limits_{k\rightarrow\infty}\bar{\mathbb{E}}\bigg[\bigg|\int_{0}^{t}
\int_{D}(w'(\bar{\bm{m}}^k)-w'(\bar{\bm{m}}))\cdot\bar{\bm{m}}^k\times\xi d\bm{x}ds\bigg|^2\bigg]\\
&\quad +\lim\limits_{k\rightarrow\infty}\bar{\mathbb{E}}\bigg[\bigg|\int_{0}^{t}
\int_{D}(\bar{\bm{m}}^k-\bar{\bm{m}})\cdot w'(\bar{\bm{m}})\times \xi d\bm{x}ds\bigg|^2\bigg]\\
&\leq C\lim\limits_{k\rightarrow\infty}\bar{\mathbb{E}}\big[|\bar{\bm{H}}_s^k-\bar{\bm{H}}_s|_{L^2(0,T;\mathbb{L}^2)}^2
|\bar{\bm{m}}^k|_{L^4(0,T;\mathbb{L}^4)}^2|\xi|_{L^4(0,T;\mathbb{L}^4)}^2\big]\\
&\quad +C\lim\limits_{k\rightarrow\infty}\bar{\mathbb{E}}\big[|\bar{\bm{m}}^k-\bar{\bm{m}}|_{L^4(0,T;\mathbb{L}^4)}^2
|\bar{\bm{H}}_s|_{L^2(0,T;\mathbb{L}^2)}^2|\xi|_{L^4(0,T;\mathbb{L}^4)}^2\big]\\
&\quad +C\lim\limits_{k\rightarrow\infty}\bar{\mathbb{E}}\big[|\bar{\bm{m}}^k-\bar{\bm{m}}|_{L^4(0,T;\mathbb{L}^4)}^2
|\bar{\bm{s}}|_{L^2(0,T;\mathbb{L}^2)}^2|\xi|_{L^4(0,T;\mathbb{L}^4)}^2\big]\\
&\quad +C\lim\limits_{k\rightarrow\infty}\bar{\mathbb{E}}\big[|\bar{\bm{m}}^k-\bar{\bm{m}}|_{L^4(0,T;\mathbb{L}^4)}^2
|w'(\bar{\bm{m}})|_{L^2(0,T;\mathbb{L}^2)}^2|\xi|_{L^4(0,T;\mathbb{L}^4)}^2\big]\\
&\quad +\frac{1}{2}\lim\limits_{k\rightarrow\infty}\bar{\mathbb{E}}\bigg[\bigg|\int_{0}^{t}
\int_{D}(\bar{\bm{s}}^k-\bar{\bm{s}})\cdot\bar{\bm{m}}^k\times\xi d\bm{x}ds\bigg|^2\bigg]\\
&\quad +\lim\limits_{k\rightarrow\infty}\bar{\mathbb{E}}\bigg[\bigg|\int_{0}^{t}
\int_{D}(w'(\bar{\bm{m}}^k)-w'(\bar{\bm{m}}))\cdot\bar{\bm{m}}^k\times\xi d\bm{x}ds\bigg|^2\bigg]
=0.
\end{align*}
Using the H\"{o}lder inequality and \eqref{kdgj}, we obtain
\begin{align*}
&\bar{\mathbb{E}}\bigg[\bigg|\int_{0}^{t}\langle\bar{\bm{m}}^k\times(\bar{\bm{H}}_s^k
+\frac{1}{2}\bar{\bm{s}}^k-w'(\bar{\bm{m}}^k),\xi\rangle_{\mathbb{L}^2}ds\bigg|^2\bigg]\\
&\leq C\bar{\mathbb{E}}\bigg[\bigg(\int_{0}^{t}|\bar{\bm{H}}_s^k
+\frac{1}{2}\bar{\bm{s}}^k-w'(\bar{\bm{m}}^k)|_{\mathbb{L}^2}|\bar{\bm{m}}^k\times\xi|_{\mathbb{L}^2}ds\bigg)^2\bigg]\\
&\leq C\bar{\mathbb{E}}\big[\big(|\bar{\bm{H}}_s^k|_{L^2(0,T;\mathbb{L}^2)}^2
+\frac{1}{4}|\bar{\bm{s}}^k|_{L^{2}(0,T;\mathbb{L}^2)}^2
+|w'(\bar{\bm{m}}^k)|_{L^{2}(0,T;\mathbb{L}^2)}^2\big)\\
&\quad \cdot|\bar{\bm{m}}^k|_{L^4(0,T;\mathbb{L}^4))}^2|\xi|_{L^4(0,T;\mathbb{L}^4)}^2\big]
\leq C,
\end{align*}
which together with the Vitali convergence theorem yield that
\begin{align*}
\lim\limits_{k\rightarrow\infty}\int_{0}^{T}\bar{\mathbb{E}}\bigg[\bigg|\int_{0}^{t}\langle\bar{\bm{m}}^k\times(\bar{\bm{H}}_s^k
+\frac{1}{2}\bar{\bm{s}}^k-w'(\bar{\bm{m}}^k))-\bar{\bm{m}}\times(\bar{\bm{H}}_s
+\frac{1}{2}\bar{\bm{s}}-w'(\bar{\bm{m}})),\xi\rangle_{\mathbb{L}^2}ds\bigg|^2\bigg]dt=0.
\end{align*}
Let $\nabla\xi\in L^8(\bar{\Omega};L^4(0,T;\mathbb{L}^4))$.
By using \eqref{11sltj}, the facts that $\bar{\bm{m}}^k\rightarrow\bar{\bm{m}}$ in $L^4(0,T;\mathbb{L}^4)$, $\bar{\bm{m}}^k\rightarrow\bar{\bm{m}}$ in $L_w^2(0,T;\mathbb{H}^1)$, $\bar{\mathbb{P}}$-a.s., we derive that
\begin{align*}
&\lim\limits_{k\rightarrow\infty}\bar{\mathbb{E}}\bigg[\bigg|\int_{0}^{t}\int_{D}(\bar{\bm{m}}^k\times\nabla\bar{\bm{m}}^k)\nabla\xi
-(\bar{\bm{m}}\times\nabla\bar{\bm{m}})\nabla\xi d\bm{x}ds\bigg|^2\bigg]\\
&\leq\lim\limits_{k\rightarrow\infty}\bar{\mathbb{E}}\bigg[\bigg|\int_{0}^{t}\int_{D}\bar{\bm{m}}^k\times(\nabla\bar{\bm{m}}^k
-\nabla\bar{\bm{m}})\cdot\nabla\xi d\bm{x}ds\bigg|^2\bigg]\\
&\quad +\lim\limits_{k\rightarrow\infty}\bar{\mathbb{E}}\bigg[\bigg|\int_{0}^{t}
\int_{D}(\bar{\bm{m}}^k-\bar{\bm{m}})\times\nabla\bar{\bm{m}}\cdot\nabla\xi d\bm{x}ds\bigg|^2\bigg]\\
&\leq C\lim\limits_{k\rightarrow\infty}\bar{\mathbb{E}}\big[|\bar{\bm{m}}^k-\bar{\bm{m}}|_{L^4(0,T;\mathbb{L}^4)}^2
|\nabla\bar{\bm{m}}|_{L^2(0,T;\mathbb{L}^2)}^2|\nabla\xi|_{L^4(0,T;\mathbb{L}^4)}^2\big]\\
&\quad +\lim\limits_{k\rightarrow\infty}\bar{\mathbb{E}}\bigg[\bigg|\int_{0}^{t}\int_{D}(\nabla\bar{\bm{m}}^k-\nabla\bar{\bm{m}})
\cdot\bar{\bm{m}}^k\times\nabla\xi d\bm{x}ds\bigg|^2\bigg]
=0,
\end{align*}
and using the H\"{o}lder inequality and \eqref{kdgj}, we get
\begin{align*}
&\bar{\mathbb{E}}\bigg[\bigg|\int_{0}^{t}\langle\bar{\bm{m}}^k\times\nabla\bar{\bm{m}}^k,
\nabla\xi\rangle_{\mathbb{L}^2}ds\bigg|^2\bigg]
\leq \bar{\mathbb{E}}\bigg[\bigg(\int_{0}^{t}\int_D|\nabla\bar{\bm{m}}^k
(\bar{\bm{m}}^k\times\nabla\xi)|d\bm{x}ds\bigg)^2\bigg]\\
&\leq C\bar{\mathbb{E}}\big[|\nabla\bar{\bm{m}}^k|_{L^2(0,T;\mathbb{L}^2)}^2
|\bar{\bm{m}}^k|_{L^4(0,T;\mathbb{L}^4)}^2|\nabla\xi|_{L^4(0,T;\mathbb{L}^4)}^2\big]\leq C.
\end{align*}
By the Vitali convergence theorem, we infer that
\begin{align*}
\lim\limits_{k\rightarrow\infty}\int_{0}^{T}\bar{\mathbb{E}}\bigg[\bigg|\int_{0}^{t}
\langle\bar{\bm{m}}^k\times\nabla\bar{\bm{m}}^k
-\bar{\bm{m}}\times\nabla\bar{\bm{m}},\nabla\xi\rangle_{\mathbb{L}^2}ds\bigg|^2\bigg]dt=0.
\end{align*}
Let $\xi\in L^8(\bar{\Omega};L^4(0,T;\mathbb{L}^4))$; according to Assumption \ref{ass1} and \eqref{11sltj}, the fact that $\bar{\bm{m}}^k\rightarrow\bar{\bm{m}}$ in $L^4(0,T;\mathbb{L}^4)$, $\bar{\mathbb{P}}$-a.s., we deduce that
\begin{align*}
&\lim\limits_{k\rightarrow\infty}\bar{\mathbb{E}}\bigg[\bigg|\int_{0}^{t}\int_{D}\bar{\bm{m}}^k
\times({G^k}'(\bar{\bm{m}}^k)[G^k(\bar{\bm{m}}^k)])\xi
-\bar{\bm{m}}\times (G'(\bar{\bm{m}})[G(\bar{\bm{m}})])\xi d\bm{x}ds\bigg|^2\bigg]\\
&\leq\lim\limits_{k\rightarrow\infty}\bar{\mathbb{E}}\bigg[\bigg|\int_{0}^{t}\int_{D}({G^k}'(\bar{\bm{m}}^k)[G^k(\bar{\bm{m}}^k)]
-G'(\bar{\bm{m}})[G(\bar{\bm{m}})])\cdot\bar{\bm{m}}^k\times\xi d\bm{x}ds\bigg|^2\bigg]\\
&\quad +\lim\limits_{k\rightarrow\infty}\bar{\mathbb{E}}\bigg[\bigg|\int_{0}^{t}\int_{D}(\bar{\bm{m}}^k-\bar{\bm{m}})
\cdot(G'(\bar{\bm{m}})[G(\bar{\bm{m}})])\times\xi d\bm{x}ds\bigg|^2\bigg]\\
&\leq CK_1\lim\limits_{k\rightarrow\infty}\bar{\mathbb{E}}\big[|\bar{\bm{m}}^k-\bar{\bm{m}}|_{L^4(0,T;\mathbb{L}^4)}^2
|\bar{\bm{m}}^k|_{L^2(0,T;\mathbb{L}^2)}^2|\xi|_{L^4(0,T;\mathbb{L}^4)}^2\big]\\
&\quad +CK_2\lim\limits_{k\rightarrow\infty}\bar{\mathbb{E}}\big[|\bar{\bm{m}}^k-\bar{\bm{m}}|_{L^4(0,T;\mathbb{L}^4)}^2
(1+|\bar{\bm{m}}|_{L^2(0,T;\mathbb{L}^2)}^2)|\xi|_{L^4(0,T;\mathbb{L}^4)}^2\big]\\
&=0.
\end{align*}
Using the H\"{o}lder inequality and \eqref{kdgj}, we obtain
\begin{align*}
&\bar{\mathbb{E}}\bigg[\bigg|\int_{0}^{t}\langle\bar{\bm{m}}^k
\times({G^k}'(\bar{\bm{m}}^k)[G^k(\bar{\bm{m}}^k)]),\xi\rangle_{\mathbb{L}^2}ds\bigg|^2\bigg]\\
&\leq C\bar{\mathbb{E}}\bigg[\bigg(\int_{0}^{t}\int_D|{G^k}'(\bar{\bm{m}}^k)[G^k(\bar{\bm{m}}^k)]
(\bar{\bm{m}}^k\times\xi)|d\bm{x}ds\bigg)^2\bigg]\\
&\leq C\bar{\mathbb{E}}\big[\big(1+|\bar{\bm{m}}^k|_{L^{2}(0,T;\mathbb{L}^2)}^2\big)
|\bar{\bm{m}}^k|_{L^{4}(0,T;\mathbb{L}^4)}^2|\xi|_{L^{4}(0,T;\mathbb{L}^4)}^2\big]
\leq C,
\end{align*}
which together with the Vitali convergence theorem imply that
\begin{align*}
\lim\limits_{k\rightarrow\infty}\int_{0}^{T}\bar{\mathbb{E}}\bigg[\bigg|\int_{0}^{t}
\langle\bar{\bm{m}}^k\times({G^k}'(\bar{\bm{m}}^k)[G^k(\bar{\bm{m}}^k)])
-\bar{\bm{m}}\times(G'(\bar{\bm{m}})[G(\bar{\bm{m}})]),\xi\rangle_{\mathbb{L}^2}ds\bigg|^2\bigg]dt=0.
\end{align*}
Next, we use \eqref{11sltj}, the fact that $\bar{\bm{m}}^k\rightarrow\bar{\bm{m}}$ in $L^4(0,T;\mathbb{L}^4)$,
$\bar{\mathbb{P}}$-a.s., Assumption \ref{ass1} and the BDG inequality to achieve the following equation:
\begin{align*}
&\lim\limits_{k\rightarrow\infty}\bar{\mathbb{E}}\bigg[\bigg|\int_{0}^{t}
\langle\bar{\bm{m}}^k\times G^k(\bar{\bm{m}}^k)
-\bar{\bm{m}}\times G(\bar{\bm{m}}),\xi\rangle_{\mathbb{L}^2}d\bar{W}(s)\bigg|^2\bigg]\\
&=\lim\limits_{k\rightarrow\infty}\bar{\mathbb{E}}\bigg[\bigg|\int_{0}^{t}\sum\limits_{i\geq1}
\langle\bar{\bm{m}}^k\times G^k(\bar{\bm{m}}^k)
-\bar{\bm{m}}\times G(\bar{\bm{m}}),\xi\rangle_{\mathbb{L}^2}\tilde{e}_id\bar{W}_i(s)\bigg|^2\bigg]\\
&\leq C\lim\limits_{k\rightarrow\infty}\bar{\mathbb{E}}\bigg[\int_{0}^{T}\bigg|\sum\limits_{i\geq1}\int_{D}\bar{\bm{m}}^k\times G_i^k(\bar{\bm{m}}^k)\xi-\bar{\bm{m}}\times G_i(\bar{\bm{m}})\xi d\bm{x}\bigg|^2ds\bigg]\\
&\leq CK_1\lim\limits_{k\rightarrow\infty}\bar{\mathbb{E}}\big[|\bar{\bm{m}}^k-\bar{\bm{m}}|_{L^4(0,T;\mathbb{L}^4)}^2
|\bar{\bm{m}}^k|_{L^2(0,T;\mathbb{L}^2)}^2|\xi|_{L^4(0,T;\mathbb{L}^4)}^2\big]\\
&\quad +CK_2\lim\limits_{k\rightarrow\infty}\bar{\mathbb{E}}\big[|\bar{\bm{m}}^k-\bar{\bm{m}}|_{L^4(0,T;\mathbb{L}^4)}^2
(1+|\bar{\bm{m}}|_{L^2(0,T;\mathbb{L}^2)}^2)|\xi|_{L^4(0,T;\mathbb{L}^4)}^2\big]
=0,
\end{align*}
and using the BDG inequality, the H\"{o}lder inequality and \eqref{kdgj}, we get
\begin{align*}
&\bar{\mathbb{E}}\bigg[\bigg|\int_{0}^{t}\langle\bar{\bm{m}}^k
\times G^k(\bar{\bm{m}}^k),\xi\rangle_{\mathbb{L}^2}d\bar{W}(s)\bigg|^2\bigg]
=\bar{\mathbb{E}}\bigg[\bigg|\int_{0}^{t}\sum\limits_{i\geq1}\langle\bar{\bm{m}}^k
\times G^k(\bar{\bm{m}}^k),\xi\rangle_{\mathbb{L}^2}\tilde{e}_id\bar{W}_i(s)\bigg|^2\bigg]\\
&\leq C\bar{\mathbb{E}}\bigg[\int_{0}^{T}\bigg|\sum\limits_{i\geq1}\int_DG^k(\bar{\bm{m}}^k)(\bar{\bm{m}}^k\times\xi)
d\bm{x}\tilde{e}_i\bigg|^2dt\bigg]\\
&\leq C\bar{\mathbb{E}}\big[\big(1+|\bar{\bm{m}}^k|_{L^{2}(0,T;\mathbb{L}^2)}^2\big)
|\bar{\bm{m}}^k|_{L^4(0,T;\mathbb{L}^4)}^2|\xi|_{L^{4}(0,T;\mathbb{L}^4)}^2\big]
\leq C.
\end{align*}
Then by the Vitali convergence theorem, we derive that
\begin{align*}
\lim\limits_{k\rightarrow\infty}\int_{0}^{T}\bar{\mathbb{E}}\bigg[\bigg|\int_{0}^{t}\langle\bar{\bm{m}}^k\times G^k(\bar{\bm{m}}^k)
-\bar{\bm{m}}\times G(\bar{\bm{m}}),\xi\rangle_{\mathbb{L}^2}d\bar{W}(s)\bigg|^2\bigg]dt=0.
\end{align*}
Similarly, we obtain
\begin{align*}
&\lim\limits_{k\rightarrow\infty}\bar{\mathbb{E}}\bigg[\bigg|\int_{0}^{t}\int_B
\langle\bar{\bm{m}}^k\times F^k(l,\bar{\bm{m}}^k)-\bar{\bm{m}}\times F(l,\bar{\bm{m}}),\xi\rangle_{\mathbb{L}^2}
\tilde{\bar{\eta}}(ds,dl)\bigg|^2\bigg]\\
&\leq C\lim\limits_{k\rightarrow\infty}\bar{\mathbb{E}}\bigg[\int_{0}^{T}\int_B
\big|\langle\bar{\bm{m}}^k\times F^k(l,\bar{\bm{m}}^k)-\bar{\bm{m}}\times F(l,\bar{\bm{m}}),
\xi\rangle_{\mathbb{L}^2}\big|^2\mu(dl)dt\bigg]\\
&\leq CK_1\lim\limits_{k\rightarrow\infty}\bar{\mathbb{E}}\big[|\bar{\bm{m}}^k-\bar{\bm{m}}|_{L^4(0,T;\mathbb{L}^4)}^2
|\bar{\bm{m}}^k|_{L^2(0,T;\mathbb{L}^2)}^2|\xi|_{L^4(0,T;\mathbb{L}^4)}^2\big]\\
&\quad +CK_2\lim\limits_{k\rightarrow\infty}\bar{\mathbb{E}}\big[|\bar{\bm{m}}^k-\bar{\bm{m}}|_{L^4(0,T;\mathbb{L}^4)}^2
(1+|\bar{\bm{m}}|_{L^2(0,T;\mathbb{L}^2)}^2)|\xi|_{L^4(0,T;\mathbb{L}^4)}^2\big]
=0,
\end{align*}
and
\begin{align*}
&\bar{\mathbb{E}}\bigg[\bigg|\int_{0}^{t}\int_B\langle\bar{\bm{m}}^k\times F^k(\bar{\bm{m}}^k,l),\xi\rangle_{\mathbb{L}^2}
\tilde{\bar{\eta}}(ds,dl)\bigg|^2\bigg]\\
&\leq C\bar{\mathbb{E}}\bigg[\int_{0}^{T}\int_B\bigg|\int_DF^k(\bar{\bm{m}}^k,l)(\bar{\bm{m}}^k\times\xi)d\bm{x}\bigg|^2\mu(dl)dt\bigg]\\
&\leq C\bar{\mathbb{E}}\big[\big(1+|\bar{\bm{m}}^k|_{L^{2}(0,T;\mathbb{L}^2)}^2\big)
|\bar{\bm{m}}^k|_{L^4(0,T;\mathbb{L}^4)}^2|\xi|_{L^{4}(0,T;\mathbb{L}^4)}^2\big]
\leq C.
\end{align*}
So we can apply the Vitali convergence theorem to deduce that
\begin{align*}
\lim\limits_{k\rightarrow\infty}\int_{0}^{T}\bar{\mathbb{E}}\bigg[\bigg|\int_{0}^{t}\int_B
\langle\bar{\bm{m}}^k\times F^k(l,\bar{\bm{m}}^k))-\bar{\bm{m}}\times F(l,\bar{\bm{m}}),\xi\rangle_{\mathbb{L}^2}
\tilde{\bar{\eta}}(ds,dl)\bigg|^2\bigg]dt=0.
\end{align*}
Hence, we can conclude that
\begin{align*}
&\int_{0}^{t}\int_{D}\partial_s\bar{\bm{m}}\cdot\xi d\bm{x}ds
=\int_{0}^{t}\int_{D}\alpha(\bar{\bm{m}}\times\partial_s\bar{\bm{m}})\xi d\bm{x}ds
+\int_{0}^{t}\int_{D}(\bar{\bm{m}}\times\nabla\bar{\bm{m}})\nabla\xi d\bm{x}ds\\
&\quad -\int_{0}^{t}\int_{D}\bar{\bm{m}}\times(\bar{\bm{H}}_s+\frac{1}{2}\bar{\bm{s}}-w'(\bar{\bm{m}}))\xi d\bm{x}ds
+\frac{1}{2}\int_{0}^{t}\int_{D}(\bar{\bm{m}}\times G'(\bar{\bm{m}})[G(\bar{\bm{m}})])\xi d\bm{x}ds\\
&\quad +\int_{0}^{t}\int_{D}(\bar{\bm{m}}\times G(\bar{\bm{m}}))\xi d\bm{x}d\bar{W}(s)
+\int_{0}^{t}\int_B\int_{D}(\bar{\bm{m}}\times F(l,\bar{\bm{m}}))\xi d\bm{x}\tilde{\bar{\eta}}(ds,dl).
\end{align*}
Moreover,  Recalling that
\begin{align*}
i\int_{0}^{T}\int_{K}\partial_t\bm{\psi}_j^k\vartheta d\bm{x}dt
=&\frac{1}{2}\int_{0}^{T}\int_{K}\nabla\bm{\psi}_j^k\nabla\vartheta d\bm{x}dt
+\int_{0}^{T}\int_{K}V^k\bm{\psi}_j^k\vartheta d\bm{x}dt\\
&-\frac{1}{2}\int_{0}^{T}\int_{K}\bm{m}^k\cdot\hat{\bm{\sigma}}\bm{\psi}_j^k\vartheta d\bm{x}dt,
\end{align*}
letting $k\rightarrow\infty$ on the both sides of the above equation, and using the same argument of Lemma \ref{dycsl}, we have
\begin{align*}
i\int_{0}^{T}\int_{K}\partial_t\bar{\bm{\psi}}_j\vartheta d\bm{x}dt
=&\frac{1}{2}\int_{0}^{T}\int_{K}\nabla\bar{\bm{\psi}}_j\nabla\vartheta d\bm{x}dt
+\int_{0}^{T}\int_{K}\bar{V}\bar{\bm{\psi}}_j\vartheta d\bm{x}dt\\
&-\frac{1}{2}\int_{0}^{T}\int_{K}\bar{\bm{m}}\cdot\hat{\bm{\sigma}}\bar{\bm{\psi}}_j\vartheta d\bm{x}dt,
\end{align*}
for all $\vartheta\in L^4(\bar{\Omega};L^4(0,T;X^{\beta}))$. Thus, we complete the proof of Theorem \ref{youjieyudejie}.

\noindent\underline{\bf Proof of Theorem \ref{them1}.} In view of $K=\{\bm{x}\in\mathbb{R}^3,\, |\bm{x}|<R\}$, taking $R\rightarrow\infty$ (the third-layer approximation), we can get the global existence of \eqref{S1} in the whole space by applying the domain expansion method and the stochastic compactness, that is, for  all $\vartheta\in L^4(\bar{\Omega};L^4(0,T;X^{\beta}))$, it holds that
\begin{align*}
i\int_{0}^{T}\int_{\mathbb{R}^3}\partial_t\bar{\bm{\psi}}_j\vartheta d\bm{x}dt
=&\frac{1}{2}\int_{0}^{T}\int_{\mathbb{R}^3}\nabla\bar{\bm{\psi}}_j\nabla\vartheta d\bm{x}dt
+\int_{0}^{T}\int_{\mathbb{R}^3}\bar{V}\bar{\bm{\psi}}_j\vartheta d\bm{x}dt\\
&-\frac{1}{2}\int_{0}^{T}\int_{\mathbb{R}^3}\bar{\bm{m}}\cdot\hat{\bm{\sigma}}\bar{\bm{\psi}}_j\vartheta d\bm{x}dt.
\end{align*}
This completes the proof of Theorem \ref{them1}. In addition, the weak martingale solutions have the following properties:
\begin{proposition}\label{Prop}
Let $(\Omega, \mathscr{F}, (\mathscr{F}_t)_{t\geq0}, \mathbb{P},\bm{\Psi},\bm{m},V,\rho,\bm{s},\bm{H}_s,W,\eta)$ be a weak martingale solution of the system \eqref{S2}, then
\begin{align*}
|\bm{\psi}_j(t)|_{\mathbb{L}^2(K)}=|\bm{\varphi}_j|_{\mathbb{L}^2(K)},\quad
|\bm{m}(t)|_{\mathbb{L}^2(D)}=|\bm{m}_0|_{\mathbb{L}^2(D)}
\end{align*}
hold $\mathbb{P}$-a.s., and
\begin{align*}
&\frac{d}{dt}\int_{K}\sum\limits_{j=1}^{\infty}\lambda_j|\nabla\bm{\psi}_{j}|^2d\bm{x}
+\frac{d}{dt}\int_{K}|\nabla V|^2d\bm{x}
+2\alpha\int_{D}|\partial_t\bm{m}|^2d\bm{x}
+\frac{d}{dt}\int_{D}|\nabla\bm{m}|^2d\bm{x}\\
&+\frac{d}{dt}\int_{\mathbb{R}^3}|\bm{H}_{s}|^2d\bm{x}
+2\frac{d}{dt}\int_{D}w(\bm{m})d\bm{x}
+\frac{d}{dt}\int_{D}|G(\bm{m})|^2d\bm{x}
-\frac{d}{dt}\int_{D}\bm{m}\cdot\bm{s}d\bm{x}\\
&+2\int_D\partial_t\bm{m}\cdot G(\bm{m})\frac{dW(t)}{dt}d\bm{x}
+2\int_B(\partial_t\bm{m},F(\bm{m},l))\tilde{\eta}(ds,dl)=0.
\end{align*}
Moreover, there exists a constant $C$ such that for all $t>0$,
\begin{align*}
&\frac{1}{2}\mathbb{E}\bigg[\sup_{0\leq t\leq T}
\int_{K}\sum\limits_{j=1}^{\infty}\lambda_j|\nabla\bm{\psi}_{j}|^2d\bm{x}\bigg]
+\mathbb{E}\bigg[\sup_{0\leq t\leq T}\int_{K}|\nabla V|^2d\bm{x}\bigg]
+2\alpha\mathbb{E}\bigg[\int_0^T\int_{D}|\partial_t\bm{m}|^2d\bm{x}dt\bigg]\nonumber\\
&+\frac{1}{2}\mathbb{E}\bigg[\sup_{0\leq t\leq T}|\bm{m}|_{\mathbb{H}^1}^2\bigg]
+\mathbb{E}\bigg[\sup_{0\leq t\leq T}\int_{\mathbb{R}^3}|\bm{H}_{s}|^2d\bm{x}\bigg]
+2\mathbb{E}\bigg[\sup_{0\leq t\leq T}\int_{D}w(\bm{m})d\bm{x}\bigg]\\
&+\mathbb{E}\bigg[\sup_{0\leq t\leq T}\int_{D}|G(\bm{m})|^2d\bm{x}\bigg]
\leq C.\nonumber
\end{align*}
\end{proposition}
\section{Analytic tools}
Let $p\in[2,\infty), q\in[2,6), \beta>\frac{1}{4}$, and consider the following function space:
$\mathbb{D}([0,T];X^{-\beta})$ is the space of c\`{a}dl\`{a}g functions $\bm{m}:[0,T]\rightarrow X^{-\beta}$ with the topology $\mathcal{J}_1$ induced by the Skorokhod metric $\delta_{T,X^{-\beta}}$; $\mathbb{L}_w^2(0,T;\mathbb{H}^1)$ is the space of measurable functions $\bm{m}:[0,T]\rightarrow \mathbb{H}^1$ with the weak topology $\mathcal{J}_2$; $\mathbb{L}^p(0,T;\mathbb{L}^q)$ is the space of measurable functions $\bm{m}:[0,T]\rightarrow \mathbb{L}^q$ with the strong topology $\mathcal{J}_3$; $\mathbb{D}([0,T];\mathbb{H}_w^1)$ is the space of weak c\`{a}dl\`{a}g functions $\bm{m}:[0,T]\rightarrow X^{-\beta}$ with the topology $\mathcal{J}_4$, where $\mathcal{J}_4$ is the weakest topology making the map $\mathbb{D}([0,T];\mathbb{H}_w^1)\ni\bm{m}\mapsto\langle\bm{m}(\cdot),h\rangle_{\mathbb{H}^1}\in\mathbb{D}([0,T];\mathbb{R})$ continuous for all $h\in\mathbb{H}^1$; $\mathbb{D}([0,T];\mathbb{B}_{\mathbb{H}_w^1}(0,r)):=\{\bm{m}\in\mathbb{D}([0,T];\mathbb{H}_w^1):
\sup_{t\in[0,T]}|\bm{m}(t)|_{\mathbb{H}^1}\leq r\}$ is the space of weak c\`{a}dl\`{a}g functions $\bm{m}:[0,T]\rightarrow X^{-\beta}$ with the topology induced by $\mathcal{J}_4$.
Let $q_r$ be a metric on $\mathbb{B}_{\mathbb{H}^1}(0,r)$ that is uniformly equivalent to the weak topology. Define
\begin{align*}
\delta_{T,r}(\bm{m}_1,\bm{m}_2):=\inf_{\gamma\in\Gamma_T}\bigg[\sup_{t\in[0,T]}q_r(\bm{m}_1(t),\bm{m}_2(\gamma(t)))
+\sup_{t\in[0,T]}|t-\gamma(t)|
+\sup_{s<t}\bigg|{\rm{log}}\bigg(\frac{\gamma(t)-\gamma(s)}{t-s}\bigg)\bigg|\bigg],
\end{align*}
where $\Gamma_T$ is the set of all increasing homeomorphisms on $[0,T]$. By the Banach-Alaoglu theorem, the topological space $\mathbb{B}_{\mathbb{H}_w^1}(0,r)$ is compact, hence the space $\mathbb{D}([0,T];\mathbb{B}_{\mathbb{H}_w^1}(0,r),\delta_{T,r})$ is a compact metric space (for details, see\cite{BM1, M}).

Next, we will recall the compactness and tightness criteria, and detailed proofs can be found in \cite{BGJ, BHW, BL, BM1, M}.
\subsection{Compactness criterion}
\begin{proposition}\label{jinxing2}
Define $\mathscr{Z}_T:=L_w^2(0,T;\mathbb{H}^1)\cap\mathbb{D}([0,T];X^{-\beta})\cap L^p(0,T;\mathbb{L}^q)\cap
\mathbb{D}([0,T];\mathbb{H}_w^1)$, and let $\mathcal{J}$ be the supremum of the corresponding topology. For a set $\mathscr{K}\subset\mathscr{Z}_T$, and $r>0$, then $\mathscr{K}$ is relatively compact in $\mathscr{Z}_T$ if it satisfies the following conditions: \\
{\rm(a)} $\sup\limits_{\bm{z}\in \mathscr{K}}|\bm{z}|_{L^{\infty}(0,T;\mathbb{H}^1)}\leq r$;\\
{\rm(b)} $\lim\limits_{\delta\rightarrow0}\sup\limits_{\bm{z}\in \mathscr{K}}\sup\limits_{|t-s|\leq\delta}
|\bm{z}(t)-\bm{z}(s)|_{X^{-\beta}}=0$, i.e. $\mathscr{K}$ is equi-continuous in $\mathbb{D}([0,T];X^{-\beta})$.
\end{proposition}
\subsection{Aldous condition}
\begin{definition}\label{taijinxing1}
Let $\{X_n\}_{n\in\mathbb{N}}$ be a stochastic process sequence on the Banach space $E$. Assuming for each $\varepsilon>0$ and $\eta>0$, there exists $\delta>0$ such that for each $(\mathscr{F}_t)_{t\geq0}$ stopping time sequence $\{\tau_n\}_{n\in\mathbb{N}}$, and $\tau_n+\theta\leq T$, it holds that
\begin{align*}
\sup\limits_{n\in\mathbb{N}}\sup\limits_{0<\theta\leq\delta}\mathbb{P}\{|X_n(\tau_n+\theta)-X_n(\tau_n)|_{E}\geq\eta\}
\leq\varepsilon,
\end{align*}
then the sequence $\{X_n\}_{n\in\mathbb{N}}$ satisfies the Aldous condition.
\end{definition}
\begin{lemma}\label{taijinxing2}
Let $\{X_n\}_{n\in\mathbb{N}}$ be a stochastic process sequence on a separable Banach space $E$. Assuming for each $(\mathscr{F}_t)_{t\geq0}$ stopping time sequence $\{\tau_n\}_{n\in\mathbb{N}}$, and $\tau_n+\theta\leq T$, and for every $n\in\mathbb{N}$, $\theta\geq0$, it holds that
\begin{align}\label{taijinxing2.1}
\mathbb{E}\{|X_n(\tau_n+\theta)-X_n(\tau_n)|_{E}^{\alpha}\}\leq C\theta^{\gamma},
\end{align}
where $\alpha, \gamma>0$, constant $C>0$, then the sequence $\{X_n\}_{n\in\mathbb{N}}$ satisfies the Aldous condition.
\end{lemma}
\subsection{Tightness}
\begin{proposition}\label{taijinxing4}
Let $\{X_n\}_{n\in\mathbb{N}}$ be an adapted $X^{-\beta}$-valued process sequence, and it satisfies\\
{\rm(a)} $\sup\limits_{n\in\mathbb{N}}\mathbb{E}\big[|X_n|_{L^{\infty}(0,T;\mathbb{H}^1)}^2\big]<\infty$,\\
{\rm(b)} The Aldous condition is satisfied in $X^{-\beta}$,\\
then the law $\{\mathbb{P}^{X_n}\}_{n\in\mathbb{N}}$ is tight in $\mathscr{Z}_T$, i.e. for each $\varepsilon>0$, for all $n\in\mathbb{N}$, there exists a compact set $K_{\varepsilon}\subset\mathscr{Z}_T$ such that
\begin{align*}
\mathbb{P}^{X_n}(K_{\varepsilon})\geq1-\varepsilon.
\end{align*}
\end{proposition}
\subsection{Skorokhod-Jakubowski Theorem}
Now, we recall the Skorokhod-Jakubowski theorem. For more details, we refer to\cite{BH, BM2, Jaku, M}.
\begin{theorem}\label{2SKO}
Let $\mathscr{X}_1$ be a complete separable metric space, $\mathscr{X}_2$ be a topological space, then there exists a sequence of continuous functions $f_l:\mathcal{X}_2\rightarrow\mathbb{R}$ separating points in $\mathcal{X}_2$. Suppose $\mathcal{X}:=\mathcal{X}_1\times \mathcal{X}_2$ with the Tykhonoff topology induced by projections as
\begin{align*}
\pi_i: \mathcal{X}_1\times \mathcal{X}_2\rightarrow\mathcal{X}_i,~i=1,2.
\end{align*}
Let $(\Omega,\mathscr{F},\mathbb{P})$ be a probability space, $\chi_n:\Omega\rightarrow\mathcal{X}_1\times \mathcal{X}_2, n\in\mathbb{N}$ be a family of random variables, and the sequence $\{\mathcal{L}(\chi_n),n\in\mathbb{N}\}$ is tight on $\mathcal{X}_1\times \mathcal{X}_2$. A random variable $\rho: \Omega\rightarrow\mathcal{X}_1$ such that $\mathcal{L}(\pi_1\circ\chi_n)=\mathcal{L}(\rho), n\in\mathbb{N}$, then there exists a probability space $(\bar{\Omega},\bar{\mathscr{F}},\bar{\mathbb{P}})$, a subsequence $\{\chi_{n_k}\}_{k\in\mathbb{N}}$, and a family of $\mathcal{X}_1\times \mathcal{X}_2$-valued random variables $\{\bar{\chi}_k, k\in\mathbb{N}\}$ on $(\bar{\Omega},\bar{\mathscr{F}},\bar{\mathbb{P}})$ and a random variable $\chi_\ast:\Omega\rightarrow\mathcal{X}_1\times \mathcal{X}_2$ such that

{\rm{(i)}} $\mathcal{L}(\bar{\chi}_n)=\mathcal{L}(\chi_n)$, for all $k\in\mathbb{N}$;

{\rm{(ii)}} $\bar{\chi}_n\rightarrow\chi_\ast$, a.s., in $\mathcal{X}_1\times \mathcal{X}_2$, as $k\rightarrow\infty$;

{\rm{(iii)}} $\pi_1\circ\bar{\chi}_k(\bar{\omega})=\pi_1\circ\chi_\ast(\bar{\omega})$, for all $\bar{\omega}\in\bar{\Omega}$.
\end{theorem}
\subsection{Some Auxilliary Results}
\begin{lemma}[Strauss lemma, see\cite{Stra}]
Let $X, Y$ be Banach spaces, and $X\hookrightarrow Y$, $T>0$. Then
\begin{align*}
L^{\infty}(0,T;X)\cap C_{w}([0,T],Y)\subset C_{w}([0,T],X).
\end{align*}
\end{lemma}
\begin{lemma}\label{fl2}
Let $r>0$, $\bm{m}_n, \bm{m}\in C_w([0,T],X)$ be sequences with the properties $\sup\limits_{t\in[0,T]}|\bm{m}_n(t)|\leq r$ and $\bm{m}_n\rightarrow\bm{m}$ in $C_w([0,T],X)$ as $n\rightarrow\infty$, then $\bm{m}_n\rightarrow\bm{m}$ in $C_w([0,T],\mathbb{B}_X(0,r))$ for $n\rightarrow\infty$.
\end{lemma}
\begin{theorem}[Burkholder-Davis-Gundy inequalities, see \cite{P04,DM05}]
Let $M$ be a martingale with c\`{a}dl\`{a}g paths, and let $[M, M]_{t}$ be the quadratic variation of $M$, fixed $p\geq1$, let $M_{t}^{\ast}={\sup}_{s\leq t}|M_{s}|$, then there exist constants $c_{p}$ and $C_{p}$ such that for any such M
\begin{align*}
\mathbb{E}\big\{[M,M]_{t}^{\frac{p}{2}}\big\}^{\frac{1}{p}}
 \leq
c_{p}\mathbb{E}\big\{(M_{t}^{\ast})^{p}\big\}^{\frac{1}{p}}
 \leq
C_{p}\mathbb{E}\big\{[M,M]_{t}^{\frac{p}{2}}\big\}^{\frac{1}{p}},
\end{align*}
for all $0\leq t\leq\infty$. The constants do not depend on the change of $M$.
\end{theorem}

\section*{Acknowledgments}
H. Wang's research is supported by the National Natural Science Foundation of China (Grant No.~11901066), the
Natural Science Foundation of Chongqing (No.CSTB2023NSCQ-MSX0396).

\noindent {\bf Statements and Declarations}

\noindent Conflicts of interest/Competing interests: The authors declare that they have no
competing interests.

\noindent {\bf Data availability}

\noindent Data sharing not applicable to this article as no datasets were generated or analysed during the current study.

\end{document}